\documentclass[11pt, reqno]{amsart}

\usepackage{amsmath,amssymb,amsthm,amsfonts,verbatim}
\usepackage{microtype}
\usepackage[all,2cell]{xy}
\usepackage{mathtools}
\usepackage{graphicx}
\usepackage{pinlabel}
\usepackage{hyperref}
\usepackage{mathrsfs}
\usepackage{color}
\usepackage[dvipsnames]{xcolor}
\usepackage{enumitem}
\usepackage{cite}
\usepackage{soul}

\CompileMatrices

\usepackage[top=1.2in,bottom=1.2in,left=1in,right=1in]{geometry}

\usepackage{hyperref} 
\hypersetup{          
	colorlinks=true, breaklinks, linkcolor=[RGB]{51 102 204}, filecolor=Orchid, urlcolor=[RGB]{51 102 204},
	citecolor=Orchid, linktoc=all, }
\usepackage[nameinlink]{cleveref}

\theoremstyle{plain}
\newtheorem{theorem}{Theorem}[section]
\newtheorem{maintheorem}{Theorem}

\newtheorem{proposition}[theorem]{Proposition}
\newtheorem{lemma}[theorem]{Lemma}

\newtheorem{corollary}[theorem]{Corollary}

\theoremstyle{definition}
\newtheorem{definition}[theorem]{Definition}

\newtheorem{remark}[theorem]{Remark}

\newtheorem{convention}[theorem]{Convention}
\newtheorem{question}[theorem]{Question}

\newcommand{\nc}{\newcommand}
\nc{\dmo}{\DeclareMathOperator}

\nc{\cA}{\mathcal{A}}
\nc{\cB}{\mathcal{B}}
\nc{\sB}{\mathscr{B}}
\nc{\C}{\mathbb{C}}
\nc{\cC}{\mathcal{C}}
\nc{\sC}{\mathscr{C}}
\nc{\BB}{\mathbb{B}}
\nc{\LL}{\mathcal{L}}
\nc{\bd}{\mathbf{d}}
\nc{\DD}{\mathbb{D}}
\nc{\cD}{\mathcal{D}}
\nc{\sD}{\mathscr{D}}
\nc{\cE}{\mathcal{E}}
\nc{\bF}{\mathbb{F}}
\nc{\cF}{\mathcal{F}}
\nc{\cG}{\mathcal{G}}
\nc{\cI}{\mathcal{I}}
\nc{\cH}{\mathcal{H}}
\nc{\cK}{\mathcal{K}}
\nc{\cL}{\mathcal{L}}
\nc{\cM}{\mathcal{M}}
\nc{\bM}{\mathbf{M}}
\nc{\N}{\mathbb{N}}
\nc{\cN}{\mathcal{N}}
\nc{\cO}{\mathcal{O}}
\nc{\bp}{\mathbf{p}}
\nc{\bP}{\mathbb{P}}
\renewcommand{\P}{\mathbb{P}}
\nc{\cP}{\mathcal{P}}
\nc{\Q}{\mathbb{Q}}
\nc{\R}{\mathbb{R}}
\nc{\cS}{\mathcal{S}}
\nc{\cT}{\mathcal T}
\nc{\cU}{\mathcal U}
\nc{\bV}{\mathbb V}
\nc{\bx}{\mathbf x}
\nc{\cX}{\mathcal{X}}
\nc{\cY}{\mathcal{Y}}
\nc{\Z}{\mathbb{Z}}

\nc{\disk}{\mathbb{D}}
\nc{\hyp}{\mathbb{H}}
\nc{\CP}{\mathbb{CP}}
\nc{\RP}{\mathbb{RP}}
\dmo{\Mod}{Mod}
\dmo{\MCG}{MCG}
\dmo{\PMod}{PMod}
\dmo{\LMod}{LMod}
\dmo{\Diff}{Diff}
\dmo{\Homeo}{Homeo}
\dmo{\dist}{dist}
\dmo\BDiff{BDiff}
\dmo\SO{SO}
\dmo\Sp{Sp}
\dmo\SL{SL}
\dmo\Out{Out}
\dmo\Aut{Aut}
\dmo\Inn{Inn}
\dmo\GL{GL}
\dmo\PGL{PGL}
\dmo\Gr{Gr}
\dmo\PSL{PSL}
\dmo\BHomeo{BHomeo}
\dmo\EHomeo{EHomeo}
\dmo\EDiff{EDiff}
\dmo\Disc{Disc}
\dmo\Aff{Aff}

\dmo\Teich{Teich}
\dmo{\Vol}{Vol}
\dmo{\Push}{Push}
\dmo{\Conf}{Conf}
\dmo{\PConf}{PConf}
\dmo{\Br}{Br}
\dmo{\PB}{PB}
\dmo{\PBr}{PBr}
\dmo{\Jac}{Jac}
\dmo{\Pic}{Pic}
\dmo{\Arf}{Arf}
\dmo{\Gal}{Gal}
\dmo{\SU}{SU}
\dmo{\spin}{spin}
\dmo{\even}{even}
\dmo{\odd}{odd}
\dmo{\orb}{orb}
\dmo{\pr}{pr}
\dmo{\lab}{lab}
\dmo{\Sym}{Sym}
\dmo{\Rad}{Rad}
\dmo{\Ind}{Ind}
\dmo{\Div}{Div}
\dmo{\Hur}{Hur}

\dmo\Hom{Hom}
\dmo\rank{rank}
\dmo\sig{sig}
\nc{\Span}[1]{\operatorname{Span}(#1)}
\dmo{\vcd}{vcd}
\dmo{\codim}{codim}
\dmo{\Res}{Res}
\dmo{\Ann}{Ann}
\dmo{\Int}{Int}
\dmo{\lcm}{lcm}
\dmo{\ab}{ab}
\dmo{\opp}{op}
\dmo{\End}{End}
\dmo{\Stab}{Stab}
\dmo{\id}{id}
\dmo{\im}{im}
\dmo\Fix{Fix}
\dmo{\Isom}{Isom}
\dmo{\res}{res}

\nc{\pair}[1]{\ensuremath{\left\langle #1 \right\rangle}}
\nc{\abs}[1]{\ensuremath{\left| #1 \right|}}
\nc{\action}{\circlearrowright}
\nc{\norm}[1]{\ensuremath{\left | \left | #1 \right | \right |}}
\nc{\abcd}[4]{\ensuremath{\left(\begin{array}{cc} #1 & #2 \\ #3 & #4 \end{array}\right)}}
\nc{\into}\hookrightarrow
\newcommand{\onto}{\twoheadrightarrow}
\nc{\normal}{\vartriangleleft}
\renewcommand{\hat}{\widehat}
\renewcommand{\bar}{\overline}

\renewcommand{\epsilon}{\varepsilon}
\renewcommand{\tilde}{\widetilde}
\renewcommand{\le}{\leqslant}
\renewcommand{\ge}{\geqslant}
\renewcommand{\ng}[1]{\mathbin{\langle\mkern-4mu\langle #1 \rangle \mkern-4mu\rangle}}

\nc{\margin}[1]{\marginpar{\scriptsize #1}}
\nc{\para}[1]{\medskip\noindent\textbf{#1.}}
\definecolor{myblue}{RGB}{102,153, 255}
\definecolor{myred}{RGB}{204,0,0}
\definecolor{mygreen}{RGB}{0,204,0}
\definecolor{myorange}{RGB}{255,102,0}
\definecolor{mypurple}{RGB}{138,43,226}
\nc{\red}[1]{\textcolor{myred}{#1}}
\nc{\blue}[1]{\textcolor{myblue}{#1}}

\renewcommand{\ul}[1]{\underline{#1}}

\dmo{\PSBr}{PSBr}
\dmo{\SBr}{SBr}

\date{December 3, 2025}
\title[Monodromy and vanishing cycles on simply connected surfaces]{Monodromy and vanishing cycles for sufficiently ample linear systems on simply connected surfaces}

\author{Ishan Banerjee}
\author{Nick Salter}

\address{Ishan Banerjee: Math Tower, the Ohio State University, 231 W 18th St Columbus OH 43210}
\email{banerjee.238@osu.edu}

\address{Nick Salter: Department of Mathematics, University of Notre Dame, 255 Hurley Building, Notre Dame, IN 46556}
\email{nsalter@nd.edu}

\begin{document}
\begin{abstract}
    We compute the mapping class group-valued monodromy of any sufficiently ample linear system on any smooth simply connected projective surface, identifying this with the $r$-spin mapping class group associated to a maximal root of the adjoint line bundle. This gives a characterization of the simple closed curves that can arise as vanishing cycles for nodal degenerations in the linear system, as well as other corollaries concerning discriminants, Lefschetz fibrations, and surfaces in $4$-manifolds.
\end{abstract}
\maketitle

\section{Introduction}

We work over the complex numbers. Let $X$ be a smooth projective surface, and let $L \in \Pic(X)$ be a line bundle. Let $\abs{L} := \P H^0(X;L)$ be the linear system, and let $U_L \subset \abs{L}$ be the locus of smooth curves in $\abs{L}$, i.e. the complement of the discriminant locus. There is a tautological family $\cC_L \to U_L$, defined via
\begin{equation}
    \label{eq:tautfamily}
    \cC_L = \{(C,x) \mid C \in U_L, x \in C\} \subset U_L \times X.
\end{equation}

The purpose of this paper is to determine the {\em topological monodromy group} of $\cC_L/U_L$ for any sufficiently ample $L$ on any {\em simply connected} $X$. Recall that any fiber bundle $p: \cE \to B$ with fiber $S$ a $2$-manifold admits a {\em topological monodromy representation} 
\[
\rho: \pi_1(B) \to \Mod(S),
\]
where $\Mod(S) := \pi_0(\Diff^+(S))$ is the {\em mapping class group} of $S$, the group of orientation-preserving diffeomorphisms of $S$ up to isotopy. The topological monodromy group is the image of $\rho$. In the present context, the monodromy of $\cC_L/U_L$ will be written $\Gamma_L$. 

In previous work, the problem of computing $\Gamma_L$ has been treated in various special settings: quintic plane curves \cite{quintic}, toric surfaces \cite{CL1,CL2,saltertoric}, and complete intersections \cite{CImonodromy}. The conclusion in each case has been that $\Gamma_L$ admits a description as an {\em $r$-spin mapping class group}. These will be defined in \Cref{S:rspin}; for now, let us remark that these are the subgroups of $\Mod(S)$ that preserve a topological avatar of an $r$-spin structure, i.e. a distinguished $r^{th}$ root of the canonical bundle of a general member of the family. In this paper we find that this is a completely general phenomenon, so long as $X$ is simply connected and $L$ is sufficiently ample.

\begin{maintheorem}\label{theorem:main}
    Let $X$ be a smooth projective surface with $\pi_1(X) = 1$. Let $L$ be a line bundle on $X$ expressible as $L = L_1 \otimes L_2$, where
    \begin{enumerate}
        \item $L_1$ is $6$-jet ample (see \Cref{remark:jetample}),
        \item $L_2$ is very ample.
    \end{enumerate} 
    Let $E \in \abs{L}$ be a general smooth curve, and let $\Gamma_L \le \Mod(E)$ denote the topological monodromy group. Let $\phi_M$ be the $r$-spin structure associated to the maximal $r^{th}$ root $M$ of the adjoint line bundle $\omega_X \otimes L$, and let $\Mod(E)[\phi_M]$ denote the associated $r$-spin mapping class group. Then 
    \[
   \Gamma_L = \Mod(E)[\phi_M].
    \]
\end{maintheorem}

\begin{remark}[Jet ampleness]\label{remark:jetample}
    Recall that a line bundle $L$ is {\em $k$-jet ample} if for any choice of $r\ge 1$ points $x_1, \dots, x_r$ and integers $k_1, \dots, k_r$ such that $k+1 = \sum_{i = 1}^r k_i$, the evaluation map 
    \[
    X \times H^0(X; L) \to L/(L \otimes m_{x_1}^{k_1} \otimes \dots \otimes m_{x_r}^{k_r})
    \]
    is surjective, where $m_{x_i}$ denotes the maximal ideal at $x_i$. In other words, there must be enough sections to simultaneously separate $k_i$-jets. The case $k = 1$ is the ordinary notion of very ampleness, and it is not hard to show that if $L$ is $a$-jet ample and $L'$ is $b$-jet ample, then $L \otimes L'$ is $(a+b)$-jet ample. See \cite{BSjetample}.
\end{remark}

\begin{remark}[The ampleness hypothesis]
It is not immediately clear when $L = L_1 \otimes L_2$ admits a decomposition of the form required for \Cref{theorem:main}. Some hypothesis beyond ampleness is necessary, as there are examples of $L$ for which every section is hyperelliptic, constraining the monodromy within a {\em hyperelliptic mapping class group}. Following \Cref{remark:jetample}, we see that the hypothesis of \Cref{theorem:main} is a fairly general condition satisfied by any line bundle sufficiently deep in the interior of the nef cone, applicable, e.g., to any $L$ expressible as $L = L_1\otimes \dots \otimes L_7$, for $L_1, \dots, L_7$ very ample. In particular, this line of argument shows that if $\Lambda$ is a rational polyhedral cone contained in the interior of the nef cone of $X$, then \Cref{theorem:main} applies to all but finitely many of the line bundles lying in $\Lambda$. 
\end{remark}

\begin{remark}[Why simple connectivity?]
We emphasize that \Cref{theorem:main} requires $X$ to be simply connected. Without this assumption, \Cref{theorem:main} is dramatically false. While the $r$-spin mapping class groups are all of finite index in $\Mod(E)$, the first-named author showed in \cite{ishanpi1} that if $\pi_1(X)$ is infinite, there are examples of very ample linear systems on certain $X$ for which the monodromy group is of infinite index in $\Mod(E)$. This is caused by the fact that the topological monodromy must preserve the kernel of the inclusion map $i_*: \pi_1(E) \to \pi_1(X)$, a constraint which is vacuous in the simply connected setting.
\end{remark}

\begin{remark}[$r$ = 1?]
    It is worth remarking explicitly that in the case where $\omega_X \otimes L$ has no nontrivial root in $\Pic(X)$, the conclusion of \Cref{theorem:main} is simply that $\Gamma_L = \Mod(E)$ is the entire mapping class group of the fiber.
\end{remark}

\begin{remark}[``Curves'' and ``surfaces''] Our working environment straddles the worlds of complex algebraic geometry and topology, and unfortunately, the standard terminologies in these fields clash. Whatever the reader's ``native language'' (topology or algebraic geometry), they should be aware of the potential for ambiguity, though hopefully in most cases the correct meaning should be clear from context. Where necessary, we will attempt to disambiguate by prepending ``algebraic'' or ``topological'' as appropriate, so that e.g. the topological space underlying a smooth algebraic curve is a topological surface. Of particular importance will be {\em simple closed curves}, i.e. an embedded submanifold of a topological surface homeomorphic to the circle. We will also attempt to be consistent in our notation: $X$ always denotes an algebraic surface, $C,D,E$ denote algebraic curves, $S$ denotes a topological surface, and lowercase letters like $a,b,c$ denote simple closed curves in $S$. 
\end{remark}

\para{Corollary: vanishing cycles} Somewhat counter-intuitively, the monodromy of a linear system $\abs{L}$ encodes a great deal of information about the {\em non}-smooth sections of $L$, despite being constructed exclusively in terms of the locus of smooth sections $U_L \subset \abs{L}$. One aspect of this is the problem of characterizing the {\em vanishing cycles} for $L$. As recalled in \Cref{SS:PL}, a {\em nodal degeneration} in $L$ is a family $E_t \in \abs{L}$ defined for $t \in \disk$ the unit disk, for which $E_t$ is smooth for $t \ne 0$, and $E_0$ has a nodal singularity. A nodal degeneration has an associated {\em vanishing cycle}, the simple closed curve $a \subset E_1$ which contracts to a point as $t \to 0$. 

Motivated by questions arising in symplectic geometry and Floer theory, in \cite{donaldson}, Donaldson asks for a characterization of which simple closed curves arise as vanishing cycles for some nodal degeneration. This is a quintessentially {\em global} question: to make sense of it, one has to base all nodal degenerations at the same basepoint, or else understand how to systematically compare simple closed curves contained in different smooth fibers. Either approach requires an understanding of the monodromy group of the linear system (or perhaps the {\em monodromy groupoid}, cf. \Cref{SS:groupoid}). 

Following arguments given in \cite{CL1,saltertoric}, if $L$ is a line bundle for which $\Gamma_L = \Mod(E)[\phi_M]$ is the $r$-spin mapping class group associated to the maximal root $M$ of $\omega_X \otimes L$, there is a simple characterization of the set of vanishing cycles. As explained below in \Cref{SS:WNF}, $r$-spin structures on $E$ are encoded topologically in terms of {\em winding number functions}, which assign ``winding numbers'' valued in $\Z/r\Z$ to oriented simple closed curves on $E$. A simple closed curve $a \subset E$ is said to be {\em admissible} if it is non-separating ($E \setminus a$ is connected) and if $\phi_M(a) = 0$, where $\phi_M$ is the winding number function associated to $M$. 

\begin{corollary}
    Let $X,L$ satisfy the hypotheses of \Cref{theorem:main}, and let $E \in U_L$ be a general fiber. Then a nonseparating simple closed curve $a \subset E$ is a vanishing cycle if and only if it is admissible.
 \end{corollary}
\begin{proof}
    \cite[Lemma 11.5]{saltertoric} shows how this statement follows from the determination $\Gamma_L = \Mod(E)[\phi_M]$. 
\end{proof}

\para{Corollary: monodromy kernels} The monodromy homomorphism $\rho: \pi_1(U_L) \to \Mod(E)$ is defined on the complement $U_L \subset \abs{L}$ of the {\em discriminant hypersurface} $\mathscr{D}_L \subset\abs{L}$ (we note in passing that as $L$ is very ample, $\mathscr{D}_L$ is given as the dual variety to the projective embedding of $X$ given by $L$). The structure of the discriminant and its complement is subject of enduring interest. Even basic topological information such as a description of the fundamental group is known only in the simplest cases - in \cite{lonne}, L\"onne gives a presentation for the discriminant complement in the setting of projective hypersurfaces (of any dimension), but this is essentially the only known result of this sort.

It is natural to wonder if perhaps $\rho$ could be {\em injective}, thereby giving a characterization of $\rho(U_L)$ as an $r$-spin mapping class group whenever \Cref{theorem:main} applies. Building off of work of Kuno \cite{kuno}, the second-named author showed in \cite{monkernel} that if the monodromy group is of finite index in the mapping class group in fact {\em obstructs} injectivity. As $r$-spin mapping class groups have finite index, combining \cite[Theorem A]{monkernel} with \Cref{theorem:main} yields following result.

\begin{corollary}
    Let $X,L$ satisfy the hypotheses of \Cref{theorem:main}. Then $\rho: \pi_1(U_L) \to \Mod(E)$ has infinite kernel.
\end{corollary}

\para{Corollary: Lefschetz pencils} Recall that a {\em Lefschetz pencil} on $X$ is a morphism $f: X \dashrightarrow \CP^1$ with finite base locus and such that fibers have at worst nodal singularities. 
For $X$ a smooth complex projective surface, a Lefschetz pencil on $X$ can be constructed from an embedding $X \into \CP^N$ arising from a very ample line bundle $L$ by projecting onto a generic line. Work of Donaldson \cite{donaldsonlefschetz} shows that when $X$ is merely a compact symplectic $4$-manifold, there is still the structure of a ``topological Lefschetz pencil'' on $X$, where the fibers are symplectic submanifolds. 

A Lefschetz pencil has finitely many singular fibers, and can be encoded combinatorially as a ``positive factorization of the identity'' $T_{a_1} \dots T_{a_k}=1$, where $a_1, \dots, a_k$ are the vanishing cycles of the singular fibers, taken relative to a well-chosen set of paths sharing a common basepoint, and $T_a$ denotes the Dehn twist about $a$ (see, e.g. \cite{aurouxsurvey}). The results of this paper show that when $X$ is complex projective and simply connected, it is always possible to choose a Lefschetz pencil on $X$ for which the monodromy {\em group} generated by such twists is ``as large as possible'':
\begin{corollary}\label{corollary:LP}
    Let $X$ be a smooth projective complex surface with $\pi_1(X) = 1$. Then there is a Lefschetz pencil $f: X \dashrightarrow \CP^1$ with generic fiber $E$, for which the monodromy group $\Gamma_f \le \Mod(E)$ is the $r$-spin mapping class group $\Mod(E)[\phi_M]$ associated to the maximal $r^{th}$ root $M$ of the adjoint line bundle $\omega_X \otimes \cO(E)$.
\end{corollary}
\begin{proof}
    Following \Cref{remark:jetample}, if $L$ is any very ample line bundle on $X$, then $L^{\otimes 7}$ satisfies the hypotheses of \Cref{theorem:main}. By the Zariski-van Kampen theorem, the monodromy of a Lefschetz pencil in $\abs{L}$ coincides with the monodromy of the full linear system; the result now follows from \Cref{theorem:main}.
\end{proof}

\.{I}. Baykur has raised the question of characterizing those (necessarily simply connected) compact symplectic $4$-manifolds that possess a Lefschetz pencil with monodromy group the full mapping class group. In the algebraic setting, such a characterization is nearly given by the Picard rank, with only sporadic exceptions.

\begin{corollary}
Let $X$ be a smooth projective surface with $\pi_1(X) = 1$. Assume $X$ is not $\CP^2$ or a $K3$ surface of Picard rank $1$. Then $X$ admits a Lefschetz pencil $f: X \dashrightarrow \CP^1$ with fiber $E$ and monodromy group $\Gamma_f = \Mod(E)$ if and only if $\mathrm{rk}(\Pic(X)) \ge 2$.  
\end{corollary}
\begin{proof}
    If $\mathrm{rk}(\Pic(X)) \ge 2$, then there is $L \in \Pic(X)$ satisfying the hypotheses of \Cref{theorem:main} for which moreover $\omega_X\otimes L$ has no nontrivial root; now apply the argument of \Cref{corollary:LP}.
    Conversely, suppose $X$ is a surface with Picard rank $1$. Let $L \in \Pic(X) \cong \Z$ be a generator, and assume that $L$ is ample, so that in particular, any line bundle $mL$ for $m < 0$ has no global sections. Let $\omega_X \cong nL$ for $n \in \Z$. Then the adjoint line bundle to $mL$ is $(m+n)L$. This possesses a nontrivial root so long as $\abs{m+n} > 1$, so suppose that $\abs{m+n} \le 1$. If the linear system $\abs{mL}$ is nonempty, then $m \ge 1$: we conclude that $n \le 0$, so that $X$ is either a $K3$ surface (of Picard rank $1$, by hypothesis), or else a del Pezzo surface, but the only del Pezzo of Picard rank $1$ is $\CP^2$.
\end{proof}

\begin{remark}
    In the case $X = \CP^2$, any Lefschetz pencil of curves of degree $d \le 4$ has monodromy group the full mapping class group. It is not immediately clear to the authors if there are examples of $K3$ surfaces of Picard rank $1$ that possess Lefschetz pencils with full monodromy.
\end{remark}

It would be interesting to see if the techniques of the paper can be extended to the symplectic setting.
\begin{question}
    Can \Cref{corollary:LP} (or even \Cref{theorem:main}) be extended to the setting of $X$ a compact simply connected symplectic $4$-manifold?
\end{question}

\para{Corollary: flexible surfaces in $4$-manifolds} Let $S$ be a topological surface embedded in a $4$-manifold $M$ (for the purposes of this discussion, we will work in the smooth category, but there are evident counterparts in the topological and locally flat settings). Define the mapping class group $\MCG(M) := \pi_0(\Diff^+(M))$ as in the $2$-dimensional setting,\footnote{We write $\MCG$, not $\Mod$, here, since the notation ``$\Mod$'' references the {\em Teichm\"uller modular group}, an analogy appropriate for surfaces but not for manifolds of other dimensions.} and the mapping class group $\MCG(M,S) := \pi_0(\Diff^+(M,S))$ as the group of isotopy classes of orientation-preserving diffeomorphisms that restrict to (potentially nontrivial) orientation-preserving diffeomorphisms of $S$. There is an evident homomorphism
\[
\res: \MCG(M,S) \to \Mod(S);
\]
the image $\cE(X,S) := \im(\res) \le \Mod(S)$ is called the group of {\em extendable mapping classes}, and a surface $S \subset M$ is said to be {\em flexible} if $\res$ is a surjection. The study of (non)flexibility has been studied by Hirose and Yasuhara \cite{hirose2,hirose,HY}; see also recent work of Lawande-Saha \cite{LS}. 

The monodromy calculations of \Cref{theorem:main} give lower bounds on the size of the extendable mapping class group. Indeed, by the isotopy extension theorem, any path $\gamma: [0,1] \to U_L$ of smooth $E_t$ in $\abs{L}$ can be extended to an isotopy $\tilde \gamma_t: X \to X$, leading to a ``surface-pushing homomorphism'' (cf. the classical theory of ``point-pushing'' in mapping class groups of surfaces)
\[
\cP: \pi_1(U_L) \to \MCG(X,E).
\]
This induces a factorization
\[
\xymatrix{
\pi_1(U_L) \ar[r]^-{\cP} \ar@/_1pc/[rr]_{\rho_L}  & \MCG(X,E) \ar[r]^{\res} & \Mod(E).
}
\]
Thus \Cref{theorem:main} produces numerous new examples of flexibly-embedded surfaces, and many more examples that are ``nearly'' flexible.

\begin{corollary}
Let $X, L$ satisfy the hypotheses of \Cref{theorem:main}, and let $E \in U_L$ be a smooth member of the linear system (topologically, a smoothly embedded subsurface of the smooth $4$-manifold underlying $X$). Then there is a containment
\[
\Mod(E)[\phi_M] \le \cE(X,E),
\]
with $\phi_M$ the $r$-spin structure associated to the maximal root of $\omega_X \otimes L$ as in \Cref{theorem:main}. In particular the index $[\Mod(E): \cE(X,E)]$ is finite, and if $r = 1$ then $E$ is flexible.
\end{corollary}

\begin{remark}
    When $r > 1$, it is not clear the extent to which the containment $\Mod(E)[\phi_M] \le \cE(X,E)$ is sharp - is every {\em smooth} self-map of the pair $(X,E)$ isotopic to one induced by an isotopy of $E$ through algebraic curves in the linear system? Said differently, any such ``algebraic'' isotopy must preserve the $r$-spin structure $\phi_M$ on $E$. Must this be preserved by an arbitrary smooth isotopy? We think this is a very interesting question worthy of further study. 
    
    It is worth noting that there is one further piece of information available. If $r$ is even and the homology class $[E] \in H_2(X;\Z/2\Z)$ is {\em characteristic} (i.e. represents the second Stiefel-Whitney class), then every diffeomorphism of the pair $(X,E)$ must preserve a {\em Rokhlin quadratic form} on $H_1(E;\Z/2\Z)$, which coincides with the mod-$2$ reduction of the $r$-spin structure $\phi_M$. In particular, $E$ cannot be flexible. This is the case e.g. for plane curves of odd degree. See \cite{hirose2} or \cite{LS} for more discussion on this point.
\end{remark}
See also Hain \cite{hain} for a related discussion in higher dimensions.

\para{Previous work} In the 1980's, Janssen \cite{janssen1,janssen2} and Ebeling \cite{ebeling} (see also the survey of Beauville \cite{beauville}) considered the problem of describing the monodromy action of a linear system on the {\em homology} of the fiber - Janssen studied the odd-dimensional case, and Ebeling the even-dimensional case. Janssen's work shows that (under some ampleness hypotheses weaker than our own) the monodromy is either the full symplectic group $\Sp(2g,\Z)$ or the stabilizer $\Sp(2g,\Z)[q]$ of the mod-$2$ quadratic form obtained from the square root of the adjoint bundle, depending on whether such a square root exists (for a discussion of the connection between $r$-spin structures and quadratic forms over $\Z/2\Z$, see \cite[Section 3.4]{saltertoric}). 
Interest in the topological refinement of this question seems to have begun in the mid 2010's, with the near-simultaneous appearance of \cite{quintic} by the second-named author (treating the example of quintic plane curves), and the work \cite{CL1} of Cr\'etois-Lang, who studied the monodromy problem in toric surfaces by using tropical techniques to construct a large collection of vanishing cycles with controlled intersection patterns. Both of these papers conjectured that the monodromy group should equal an $r$-spin mapping class group in greater generality, but progress at this point was stymied by the lack of any sort of general theory of $r$-spin mapping class groups, particularly a method for constructing generating sets.

The paper \cite{saltertoric} by the second-named author subsequently developed some of this missing basic theory of $r$-spin mapping class groups, and used Cr\'etois-Lang's work to complete the analysis of monodromy groups for linear systems on smooth toric surfaces. Further progress on these questions then stalled, on account of the fact that for more general $X$, it is difficult to get a coherent ``picture'' of how the vanishing cycles for distinct nodal degenerations interact and intersect one another. In the case of $X$ a smooth complete intersection surface, this was resolved by the authors in \cite{CImonodromy} via induction, the basic idea being that if one can understand the monodromy group ``on a subsurface'' in sufficient detail, it is possible to circumvent a need to carefully build a configuration of vanishing cycles ``by hand''. But the method of \cite{CImonodromy} makes essential use of the fact that there are {\em base cases} available: the canonical embeddings of generic curves of genus $3,4,5$ all arise as complete intersections, giving an {\em a priori} understanding of their monodromy groups which can then be propagated inductively to general linear systems on complete intersection surfaces. For a completely arbitrary $(X,L)$, there does not seem to be a good way to relate it to a linear system whose monodromy is already understood.

\para{Overview of the argument I: general remarks} Establishing an equality of the form $\Gamma_L = \Mod(E)[\phi_M]$ has two halves - the containments in either direction. The containment $\Gamma_L \le \Mod(E)[\phi_M]$ (\Cref{lemma:rspincontain}) is straightforward, essentially following immediately once one understands that the set of $r$-spin structures carries an action by the mapping class group.

To establish the other containment $\Mod(E)[\phi_M] \le \Gamma_L$, it is necessary to understand when a set of mapping classes generates an $r$-spin mapping class group. In light of the Picard-Lefschetz formula (cf. \Cref{SS:PL}), $\Gamma_L$ contains the Dehn twist $T_a$ about any vanishing cycle $a \subset E$. It is reasonable to hope that $\Mod(E)[\phi_M]$ itself might be generated by some finite collection of such twists, and that is indeed what the theory shows. Classical results on generating the mapping class group, e.g. the Humphries generating set, require one to have a ``complete picture'' of the surface, and produce a collection $c_1, \dots, c_N$ of simple closed curves whose intersection pattern is precisely and totally specified, in order to conclude that the corresponding Dehn twists generate. As discussed in the previous paragraph, this would be extremely difficult to achieve in our present context. 

Luckily, there exist flexible and readily-satisfied {\em criteria} for a collection of Dehn twists to generate an $r$-spin mapping class group, formulated in terms of an {\em assemblage} of simple closed curves. Essentially, an assemblage describes a topological surface as being built by a sequence of $1$-handle attachments, by taking regular neighborhoods of a sequence of simple closed curves. The ``assemblage generating set criterion'' of \Cref{theorem:assemblages} gives a condition under which the Dehn twists about all of these curves together generate an $r$-spin mapping class group (more accurately, a related notion called a ``framed mapping class group''). The advantage of an assemblage is that precise control over the intersection pattern of the curves is unnecessary: it is only necessary to track how a curve $c_i$ in an assemblage intersects the {\em subsurface} $S_{i-1}$ determined by the previous curves, which is much easier to control. 

\para{Overview II: building vanishing cycles in $\mathbf{\abs{L}=\abs{L_1 \otimes L_2}}$} To summarize, the objective will be to manufacture a collection of vanishing cycles over which one has some modest topological control, and that end up ``filling'' the fiber. To do this, we decompose $L = L_1 \otimes L_2$ as in \Cref{theorem:main}, and study curves in $\abs{L}$ that arise by {\em smoothing} the union $C \cup D$ of smooth transversely-intersecting sections $C \in U_{L_1}$ and $D \in U_{L_2}$. We work throughout based at some $E \in \abs{L_1 \otimes L_2}$ constructed in this way: topologically $E$ decomposes as a union $E = \tilde C \cup \tilde D$, where $\tilde C$, resp. $\tilde D$ are subsurfaces obtained by taking the real oriented blowup of $C$, resp. $D$ at the points of intersection $C \cap D$. See \Cref{SS:CDE} for more details.

The theory of assemblages requires one to start from a ``core'' subsurface of $E$, on which one has a collection of vanishing cycles whose intersection pattern is more tightly controlled than in general. This step is carried out in \Cref{S:core}. The basic method is to ``implant'' isolated plane curve singularities, where there are classical techniques available to understand configurations of vanishing cycles. This is the origin of the requirement that $L_1$ be $6$-jet ample - this allows us in \Cref{SS:planecurvesing} to embed the versal deformation space of the $E_6$ singularity into the linear system, a key step in constructing the core.

From here, the problem is to fill out the rest of $E$ with a sequence of vanishing cycles. The key insight is $E$ can be filled by cycles $a \subset E$ that restrict on both halves $\tilde C, \tilde D$ as a single arc. Vanishing cycles of this type arise when the sections $C \in U_{L_1}$ and $D \in U_{L_2}$ approach a point of tangency (a {\em tacnode singularity}). In our previous work \cite{CImonodromy}, we showed how to construct such tacnodes with {\em complete control} on how the vanishing cycle intersects $\tilde C \subset E$. The inductive framework of \cite{CImonodromy} meant that it was not necessary to understand the intersection with the other half.

Here, to circumvent the need for an inductive hypothesis, we need to construct tacnodal vanishing cycles with control on the intersection with both $\tilde C$ and $\tilde D$. The ``Main Lemma'' \Cref{lemma:main} shows how this can be done. This is the most technically-demanding step in the argument, but the basic idea is to exploit a large auxiliary subgroup of the monodromy, arising by holding $C \in U_{L_1}$ fixed, letting $D \in U_{L_2}$ vary through smooth sections intersecting $C$ transversely, and smoothing. The monodromy of such a loop fixes the halves $\tilde C$ and $\tilde D$ and is describable on $\tilde C$ as a ``surface braid'', obtained by pushing the boundary components of $\tilde C$ around paths on $\tilde C$. A delicate argument in surface topology describes the further subgroup on which the monodromy on the $\tilde D$-side is trivial, and the action of this subgroup on arcs on $\tilde C$. \Cref{lemma:main}, which provides for an abundance of tacnodal vanishing cycles, then follows as a corollary. We remark that this is the only place in the argument in which the hypothesis $\pi_1(X) = 1$ is used.

\para{Organization of the paper} There are two halves to the paper: \Cref{S:rspin,S:smoothing,S:braids,S:monodromy} contain preliminary material, and \Cref{S:core,S:mainlemma,S:mainproof} give the sequence of arguments that culminate in the proof of \Cref{theorem:main}. In part one, \Cref{S:rspin} lays out the theory of $r$-spin and framed mapping class groups. \Cref{S:smoothing} discusses both topological and algebro-geometric aspects of sections of line bundles, especially the technical foundations of {\em smoothing} a union of sections from different line bundles (and related constructions in families). \Cref{S:braids} recalls the theory of surface braid groups and especially the subgroup of {\em simple} braids, which are essential to understand the monodromy elements induced by the smoothing procedure of the previous section. In \Cref{S:monodromy}, certain special subgroups of the monodromy group are constructed using techniques from singularity theory. 

In the second part, \Cref{S:core} expands on the local monodromy calculations of the previous section, using the methods of framed mapping class groups to show the containment of a useful intermediate subgroup. \Cref{S:mainlemma} states and proves the Main Lemma \Cref{lemma:main}, and finally \Cref{S:mainproof} establishes the last necessary lemma, then gives the proof of \Cref{theorem:main}.

\para{Acknowledgements}
The first author would like to thank Dan Abramovich for answering some questions on extending spin structures. The second-named author is supported in part by NSF CAREER grant no. DMS-2338485. He would like to thank \.{I}nan\c{c} Baykur and Eric Riedl for some helpful discussions.

\section{$r$-spin mapping class groups}\label{S:rspin}

\subsection{Framings and $r$-spin structures}\label{SS:WNF}
Here we recall the elements of the theory of framings on (topological) surfaces, mostly following \cite[Section 2]{strata3}. 

Let $S$ be an oriented surface with nonempty boundary. Fix an orientation $\theta$ of $S$ and a Riemannian metric $\mu$; then the structure group of the tangent bundle $TS$ of $S$ reduces to $\SO(2)$.  A {\em framing} of $S$ is a trivialization of the tangent bundle, i.e. an isomorphism of $\SO(2)$-bundles $\phi: TS \cong S \times \R^2$, or more concretely, a pair $\xi_1,\xi_2$ of vector fields that are everywhere linearly independent. Having fixed $\theta, \mu$, a single non-vanishing vector field $\xi$ on $S$ induces a framing of $S$ by taking $\xi_2$ to be the rotation of $\xi$ by $90$ degrees counter-clockwise. Two framings $\phi, \psi$ are said to be {\em isotopic} if the corresponding vector fields $\xi_\phi, \xi_\psi$ are isotopic through non-vanishing vector fields. 

Suppose that $\phi, \psi$ restrict to the same framing $\delta$ on $\partial S$. Then $\phi, \psi$ are said to be {\em relatively isotopic} if they are isotopic through framings restricting to $\delta$ on $\partial S$. With a choice of $\delta$ fixed, we say that $\phi$ is a {\em relative framing} if the restriction of $\phi$ to $\partial S$ is $\delta$. Often the particular $\delta$ is not terribly important, and we will abuse terminology, discussing (isotopy classes of) relative framings without reference to $\delta$. 

\para{(Relative) winding number functions} With $S$ as before a surface with nonempty boundary, let $\cS(S)$ denote the set of isotopy classes of oriented simple closed curves on $S$. Let $\phi$ be a framing of $S$. Then there is a {\em winding number function}
\[
\phi: \cS(S) \to \Z
\]
(abusively using the same symbol; cf. \Cref{prop:wn}) which measures the winding of the forward-pointing tangent vector of $c \in \cS$ relative to the global identification $\phi: TS \cong S \times \R^2$. 

Enumerate the boundary components of $S$ as $\Delta_1, \dots, \Delta_d$, and choose a basepoint $p_i \in \Delta_i$ for each $i$. Let $\cS^+(S)$ denote the set of isotopy classes of simple closed curves on $S$ along with the set of isotopy classes of arcs with boundary contained in $\{p_i\} \subset \partial S$. With slightly more care (see \cite[Section 2]{strata3}), such arcs can be assigned half-integral winding numbers, so that there is an extension of $\phi$ to a {\em relative winding number function}
\[
\phi: \cS^+(S) \to \tfrac{1}{2} \Z.
\]
It is clear that if two framings are (relatively) isotopic, then the (relative) winding number functions coincide. In fact, the (relative) winding number function is a {\em complete} encoding of the (relative) isotopy class, justifying our use of the same symbol for both.

\begin{proposition}[Cf. Proposition 2.4 of \cite{RW}]\label{prop:wn}
    Framings $\phi, \psi$ of a surface $S$ are isotopic if and only if there is an equality of the associated winding number functions. If moreover $\phi|_{\partial S} = \psi|_{\partial S}$, then $\phi$ and $\psi$ are relatively isotopic if and only if there is an equality of the associated relative winding number functions.
\end{proposition}

\para{Properties of winding number functions} Winding number functions satisfy some properties which turn out to tightly constrain their behavior.

\begin{lemma}[Cf. Lemma 2.4 of \cite{strata3}]\label{lemma:WNFprops}
Let $(S,\phi)$ be a framed surface, and let $\phi: \cS(S) \to \Z$ be the associated relative winding number function. Then $\phi$ satisfies the following properties:
\begin{enumerate}
    \item (Twist-linearity) Let $c \subset S$ be a simple closed curve, and let $a \in \cS^+(S)$ be either a curve or an arc. Let $T_c\in \Mod(S)$ denote the (right-handed) Dehn twist about $c$. Then
    \[
    \phi(T_c(a)) = \phi(a) + \pair{a,c}\phi(c),
    \]
    where $\pair{\cdot,\cdot}: H_1(S,\partial S; \Z) \times H_1(S;\Z) \to \Z$ denotes the relative intersection pairing.

    \item (Homological coherence) Let $S' \subset S$ denote a subsurface bounded by the simple closed curves $c_1, \dots, c_n$. Orient each $c_i$ so that $S'$ lies to the left. Then
    \[
    \sum_{i = 1}^n \phi(c_i) = \chi(S'),
    \]
    where $\chi(S')$ denotes the Euler characteristic of $S'$. 
\end{enumerate}
\end{lemma}
\begin{remark}\label{remark:reversible}
    In some sources, winding number functions are shown (or axiomatized) to possess a third property of {\em reversibility}, namely $\phi(\bar a) = - \phi(a)$, where $\bar a \in \cS(S)$ is the same curve as $a$ equipped with the opposite orientation. Note however that for $a$ a simple closed curve, this {\em follows} from homological coherence, as applied to an annulus bounded by two isotopic curves $a, \bar a$. Reversibility also holds for arcs, but we will not need that here.
\end{remark}

\para{$r$-spin structures} Topologically, an $r$-spin structure is a sort of ``mod-$r$ framing'' on a surface. They arise naturally in the context of algebraic geometry, although the connection with framings is initially obscure. Let $C$ be a Riemann surface and let $\omega_C$ be its canonical bundle. An {\em $r$-spin structure} is defined to be an $r^{th}$ root of $\omega_C$, i.e. a line bundle $L \in \Pic(C)$ (necessarily of degree $(2g-2)/r$) such that 
\[
L^{\otimes r} \cong \omega_C.
\]
In light of the fact that the real $2$-plane bundle underlying $\omega_C$ is the cotangent bundle $T^*C$, it is perhaps not surprising that $r$-spin structures should connect with the theory of framings. This is especially clear from the point of view of winding number functions.

\begin{proposition}
    Let $C$ be a Riemann surface, and let $L \in \Pic(C)$ be an $r$-spin structure. As above, let $\cS(C)$ denote the set of isotopy classes of oriented simple closed curves on $C$. Then $L$ determines a ``mod-$r$ winding number function''
    \[
    \phi_L: \cS(C) \to \Z/r\Z
    \]
    that satisfies the twist-linearity and homological coherence properties of \Cref{lemma:WNFprops}.
\end{proposition}
\begin{proof}
    See \cite[Sections 2,3]{quintic}.
\end{proof}

\subsection{$r$-spin mapping class groups and their generating sets} From this point forward, we will shift our perspective, prioritizing the role of mod-$r$ winding number functions and suppressing their origins as roots of the canonical bundle. As we have seen, a framing gives rise to a $\Z$-valued winding number function. Accordingly, we will allow $r = 0$ as well, in which case it is understood that the winding number function originates with a framing as opposed to a root of the canonical bundle. Passing to the level of winding number functions makes the action of the mapping class group on the set of $r$-spin structures transparent: given a winding number function $\phi: \cS(S) \to \Z/r\Z$ and a mapping class $f \in \Mod(S)$, define the winding number function $f \cdot \phi$ via
\begin{equation}\label{eq:wnfaction}
    (f \cdot \phi)(c) = \phi(f^{-1}(c))
\end{equation}
for any $c \in \cS(S)$. The same formula defines an action of $\Mod(S)$ on the set of {\em relative} winding number functions $\phi: \cS^+(S) \to \Z$.

\begin{definition}[$r$-spin mapping class group, framed mapping class group]
    Let $S$ be a surface and let $\phi: \cS(S) \to \Z/r\Z$ be a mod-$r$ winding number function on $S$ for some $r \ge 0$. The {\em $r$-spin mapping class group}, written
    \[
    \Mod(S)[\phi] \le \Mod(S),
    \]
    is defined as the stabilizer of $\phi: \cS(S) \to \Z/r\Z$ under the action \eqref{eq:wnfaction}, i.e. as the set of mapping classes that preserve all $\phi$-winding numbers. If $\phi: \cS^+(S) \to \Z$ is (the relative winding number function associated to) a relative framing, there is an analogous {\em framed mapping class group}, again defined as the stabilizer of $\phi$ and still written $\Mod(S)[\phi]$. 
\end{definition}

\begin{remark}
    Though it is not strictly relevant in this paper, let us emphasize the distinction between a {\em framed mapping class group} and an {\em $r$-spin mapping class group for $r = 0$}. By convention, a {\em framed mapping class group} always denotes the stabilizer of a framing up to {\em relative} isotopy, combinatorially encoded as the stabilizer of a {\em relative} winding number function. An $r$-spin mapping class group always denotes the stabilizer of a $\Z/r\Z$-valued ordinary winding number function (i.e. only simple closed curves, and not arcs, are assigned winding numbers). For $r= 0$, an $r$-spin mapping class group is the stabilizer of an isotopy class of framing where the isotopy is allowed to be nontrivial on the boundary. 
    In this paper, when working with surfaces with boundary (those that support any framing at all), we will exclusively work with framed mapping class groups. 
\end{remark}

\para{Admissible curves} Dehn twists are a particularly simple type of mapping class, and it is natural to wonder which Dehn twists appear in a given $r$-spin mapping class group. There is a simple characterization.
\begin{definition}[Admissible curve]
    Let $S$ be a surface, and let $\phi: \cS(S) \to \Z/r\Z$ be an $r$-spin structure for some $ r\ge 0$. A simple closed curve $a \subset S$ is said to be {\em admissible} if $a$ is nonseparating and if $\phi(a) = 0$ (necessarily for either choice of orientation; see \Cref{remark:reversible}). The same definition applies in the relative setting.
\end{definition}

\begin{lemma}\label{lemma:admissibletwists}
    Let $\phi$ be an $r$-spin structure for some $r \ge 0$ or a relative framing, and let $a \subset S$ be a nonseparating simple closed curve. Then $T_a \in \Mod(S)[\phi]$ if and only if $a$ is admissible.
\end{lemma}
\begin{proof}
    If $a$ is admissible, then the twist-linearity formula of \Cref{lemma:WNFprops} reduces to show that $\phi(T_a(c)) = \phi(c)$ for all $c \in \cS^{(+)}(S)$, and hence $T_a \in \Mod(S)[\phi]$. Conversely, if $a$ is not admissible, then there exists an oriented simple closed curve $b \subset S$ such that $\pair{a,b} = 1$; then
    \[
    \phi(T_a(b)) = \phi(b) + \phi(a) \ne \phi(b),
    \]
    showing that $T_a$ does not preserve $\phi$.
\end{proof}

\para{The framed change-of-coordinates principle} The classical {\em change-of-coordinates principle}, as discussed in \cite[Section 1.3]{FM}, lays out conditions under which a configuration of curves or subsurfaces exists, subject to given topological constraints (e.g. as an illustrative example, it ensures that if $a \subset S$ is any nonseparating simple closed curve, there exists a simple closed curve $b \subset S$ for which $i(a,b) = 1$). On a technical level, the change-of-coordinates principle is a reflection of various {\em transitivity} results for the action of the mapping class group on topological configurations.

The {\em framed change-of-coordinates principle} is an analogous principle when the surface $S$ is equipped with a framing or $r$-spin structure; it again follows from an understanding of transitivity properties of framed mapping class groups acting on simple closed curves etc. There is one obvious new obstruction to transitivity: if $\phi$ is an $r$-spin structure and $a,b \subset S$ are simple closed curves, then some $f \in \Mod(S)[\phi]$ exists for which $f(a) = b$ only if $\phi(a) = \phi(b)$. There is a second, more subtle {\em global} obstruction known as the ``Arf invariant'', which is relevant only in understanding orbits of ``large'' configurations (e.g. when the curves involved span $H_1(S;\Z)$). In this paper we will not need to invoke the framed change-of-coordinates principle in any setting where the Arf invariant is relevant, and so we formulate a version of the principle that sidesteps this by assuming the existence of a complementary subsurface of positive genus. See \cite[Proposition 2.15]{strata3} for full details. 

To give a precise formulation of the framed change-of-coordinates principle, we will restrict ourselves to a precise notion of ``configuration'' of simple closed curves. The notion of {\em $E$-arboreality} also defined below will be used momentarily in the definition of an {\em assemblage}.

\begin{definition}[Simple configuration, arboreal, of type $E$, filling]\label{def:simpleconfig}
    Let $S$ be a surface. A {\em simple configuration} on $S$ is a set $\cC = \{c_1, \dots, c_k\}$ of simple closed curves, subject to the condition that $i(c_i,c_j) \le 1$ for all pairs of curves $c_i, c_j \in \cC$ (here and throughout, $i(\cdot, \cdot)$ denotes the geometric intersection number).

    A simple configuration $\cC$ has an associated {\em intersection graph} $\Lambda_\cC$ with vertex set $\cC$ and with $c_i, c_j$ joined by an edge if and only if $i(c_i,c_j) = 1$. The {\em embedded type} of $\cC$ is the data of $\Lambda_\cC$ together with the homeomorphism type of the complement $S \setminus \cC$.
    
    A simple configuration is {\em arboreal} if $\Lambda_\cC$ is a tree, and is {\em $E$-arboreal} if moreover $\Lambda_\cC$ contains the $E_6$ Dynkin diagram as a full subgraph. $\cC$ is {\em spanning} if there is a deformation retraction of $S$ onto the union of the curves in $\cC$.
\end{definition}

As usual, a {\em $k$-chain} is a simple configuration $c_1, \dots, c_k$ for which $i(c_i, c_{i+1}) = 1$ and $i(c_i, c_j) = 0$ otherwise. 

\begin{proposition}[Framed change-of-coordinates principle]
    \label{prop:ccp}
    Let $(S, \phi)$ be a surface of genus $g \ge 2$ equipped with a (relative) mod-$r$ winding number function $\phi$ for some $r \ge 0$. Let $\cC' = \{c_1', \dots, c_k'\}$ be a simple configuration of curves on $S$, and suppose there is a subsurface $S' \subset S$ of genus $g(S') \ge 1$ such that $c_i' \in S \setminus S'$ for all $i$. For $1 \le i \le k$, let $w_i \in \Z/r\Z$ be given, and suppose there exists some (relative) mod-$r$ winding number function $\psi$ with $\psi(c_i') = w_i$ for all $i$. Then there is a simple configuration $\cC = \{c_1, \dots, c_k\}$ of the same embedded type as $\cC$ and with $\phi(c_i) = w_i$ for all $i$.
\end{proposition}
\begin{proof}
    See \cite[Proposition 2.15]{strata3}.
\end{proof}

\para{Assemblages} The identification in \Cref{theorem:main} of the monodromy of a linear system with an $r$-spin mapping class group is made possible by the existence of a flexible criterion for a set of mapping classes to generate some such $\Mod(S)[\phi]$. This is formulated in terms of a notion known an an {\em assemblage}, which we recall here.

The work of \cite{strata3} shows that the Dehn twists about curves in an $E$-arboreal spanning configuration (cf. \Cref{def:simpleconfig}) generate an associated framed mapping class group, but in fact, something much more general holds. By definition, the curves in a spanning configuration are constrained to pairwise intersect at most once. In practical applications, this would be cumbersome to verify, and moreover, it is unnecessarily restrictive. The notion of an {\em assemblage} relaxes the intersection condition, as follows. 

\begin{definition}[$h$-assemblage of type $E$, core]\label{def:assemblage}
    Let $S$ be a surface. A set $\cC = \{c_1, \dots, c_k, c_{k+1},\dots, c_\ell\}$ is said to be an {\em $h$-assemblage of type $E$} (or just an {\em assemblage}, if total precision is unnecessary) if the conditions below hold. Throughout, let  $S_j$ denote a regular neighborhood of $c_1, \dots, c_j$. 
    \begin{enumerate}
        \item $S_k$ is a subsurface of genus $h$, and the curves $c_1,\dots, c_k$ form an $E$-arboreal spanning configuration on $S_k$,
        \item For $k \le j \le \ell-1$, the intersection $c_{j+1} \cap S_j$ is a single arc (possibly, though not necessarily, with both endpoints lying on the same component of $\partial S_j$),
        \item There is a deformation retraction of $S$ onto $S_\ell$. 
    \end{enumerate}
    The subsurface $S_k$ is called the {\em core} of the assemblage.
\end{definition}

Thus in an assemblage, the curves $c_1, \dots, c_k$ lying in the core are restricted to satisfy $i(c_i,c_j) \le 1$, but further curves $c_j$ for $k+1 \le j \le \ell$ must only satisfy the far weaker condition of (2), that adding $c_j$ enlarges the supporting surface by performing a $1$-handle attachment; they are free to intersect prior curves $c_i$ for $i < j$ arbitrarily.

One useful property of an assemblage is that it determines a unique isotopy class of framing for which every curve in the assemblage is admissible. Note here the absence of the word ``relative'': an assemblage does {\em not} determine some preferred restriction of the framing to $\partial S$. 

\begin{definition}[Compatible framing]\label{def:compatible}
    Let $\cC = \{c_1, \dots, c_k, \dots, c_\ell\}$ be an assemblage on a surface $S$. A framing $\phi$ of $S$ is said to be {\em compatible with $\cC$} if every $c_i$ is admissible for $\phi$. 
\end{definition}

\begin{lemma}
    \label{lemma:compatibleframing}
    Let $\cC$ be an assemblage on a surface $S$. Then there is a unique isotopy class of framing $\phi$ compatible with $\cC$.
\end{lemma}
\begin{proof}
    Since the curves $c_1, \dots, c_k$ in the core form an {\em arboreal} spanning configuration on $S_k$, we can assume they are ordered in such a way that $c_{j+1} \cap S_{j}$ is a single arc for {\em all} $1 \le j \le \ell-1$, not just for $j \ge k$ as stipulated in property (3) of \Cref{def:assemblage}. It is clear that there is a framing $\phi_1$ of $S_1$ for which $c_1$ is admissible. Now proceed by induction, supposing the existence of a framing $\phi_j$ on $S_j$ for which $c_1, \dots, c_j$ are admissible. $S_{j+1}$ is constructed from $S_j$ by attaching a 1-handle $I \times I$ along $\partial I \times I$, a neighborhood of $c_{j+1} \cap \partial S_j$. \cite[Lemma 3.20]{CImonodromy} then shows that there is a unique isotopy class of extension of $\phi_j$ to a framing $\phi_{j+1}$ of $S_{j+1}$ for which which the attaching curve $c_{j+1}$ is admissible.
    \end{proof}

\para{Generating the framed mapping class group via assemblages} The Dehn twists about the curves in an $h$-assemblage of type $E$ for $h \ge 5$ generate the associated framed mapping class group:
        \begin{theorem}[Theorem B.II of \cite{strata3}]\label{theorem:assemblages}
    Let $S$ be a surface of genus $g \ge 5$ with $n \ge 1$ boundary components. Let $\phi$ be a relative framing of $S$, and let $\cC = \{c_1, \dots, c_\ell\}$ be an $h$-assemblage of type $E$ for some $h \ge 5$ such that $\phi(c_i) = 0$ for all $c_i \in \cC$. Then 
    \[
    \Mod(S)[\phi] = \pair{T_{c_i} \mid c_i \in \cC}.
    \]
    \end{theorem}
\begin{remark}
    There is a subtlety here, which, while not relevant in the body of the paper, is worth pointing out. As discussed in \Cref{lemma:compatibleframing}, the assemblage only determines an isotopy class of framing, not a {\em relative} isotopy class of such (where isotopies are required to restrict to identity on the boundary). Nevertheless, the admissible twists in the assemblage generate the framed mapping class group, i.e. the stabilizer of the framing up to relative isotopy. In other words, if $\phi_1, \phi_2$ are isotopic framings (where the isotopy need not restrict to identity on the boundary), then the framed mapping class groups $\Mod(S)[\phi_1] = \Mod(S)[\phi_2]$ are {\em equal} even if the relative winding number functions ${\phi_1}, {\phi_2}: \cS^+(S) \to \Z$ are distinct (although in such a situation, necessarily the restrictions of ${\phi_1},{\phi_2}$ to the set $\cS(S)$ of simple closed {\em curves} must coincide).
\end{remark}

\subsection{Functoriality} Let $S \into S'$ be an inclusion of surfaces. Let $\phi'$ be a framing of $S'$; then $\phi'$ restricts to a framing $\phi'|_S$ of $S$, encoded combinatorially as a restriction of the winding number function to the subset $\cS(S) \le \cS(S')$. Conversely, if $S \subset S'$ is a subsurface, $\phi$ is a framing of $S$, and $\phi'$ is a framing of $S'$, then we say that $(S,\phi) \into (S', \phi')$ is a {\em framed inclusion} if $\phi'|_S = \phi$. Note that these definitions are sensible at the level of (non-relative) isotopy. 

A full discussion of when a framed inclusion $(S,\phi) \into (S',\phi')$ induces an inclusion of framed mapping class groups would require establishing a theory of ``partitioned framed surfaces'' {\em a la} Putman's theory of the Torelli group \cite{putmancutpaste}. We will not need to pursue this here. We will, however, need the following more elementary observations. 

First, observe that if $\phi'$ is an $r'$-spin structure on $S$ for some $r' \ge 0$, then for any divisor $r$ of $r'$, there is a unique reduction $\phi' \pmod{r}$ to an $r$-spin structure. At the level of winding number functions, this is given by reduction mod $r$. If $r' >0$ and $S = C$ is a Riemann surface, then at the level of line bundles, if $L \in \Pic(C)$ is the $r'$-spin structure, then $L^{\otimes r'/r}$ is the reduction to an $r$-spin structure. Conversely, we say that an $r'$-spin structure $\phi'$ is a {\em $\Z/r'\Z$-refinement} of an $r$-spin structure $\phi$ if $\phi' \pmod{r} = \phi$.

More substantially, we will need to understand functoriality in the following special case of {\em capping boundary components}. This was developed in our previous paper.

\begin{lemma}[Lemma 3.21 of \cite{CImonodromy}]\label{lemma:framedontospin}
    Let $S \into \bar S$ be an inclusion of surfaces, given by capping off boundary components $d_1, \dots, d_N$ of $S$ with closed disks (note that $\bar S$ may be closed). Let $\phi$ be a framing on $S$, and define
    \[
    r = \gcd(\phi(d_1)+1, \dots, \phi(d_N) +1),
    \]
    where $d_1, \dots, d_N$ are the boundary components of $S$ capped off in $\bar S$, oriented with $S$ to the left. Then the inclusion $S \into \bar S$ induces a surjection
    \[
    \Mod(S)[\phi] \onto \Mod(\bar S)[\bar \phi],
    \]
    where $\bar \phi$ is the $r$-spin structure on $\bar S$ obtained by reducing $\phi$ mod $r$.
\end{lemma}

\subsection{Non-containment implies maximality} The proof of \Cref{theorem:main} identifies the monodromy group $\Gamma_L$ with the $r$-spin mapping class group $\Mod(E)[\phi_M]$ associated to the maximal root of the adjoint line bundle in a somewhat indirect way. Using the method of assemblages and the capping lemma (\Cref{lemma:framedontospin}), the main thrust of the argument shows that $\Gamma_L$ contains some $r'$-spin mapping class group, for some $r \mid r'$. To show that $r' = r$, we show in \Cref{prop:charlie} that $\Gamma_L$ exerts some control over the structure of the Picard group of $X$, so that if $\Gamma_L$ were contained in some $r'$-spin mapping class group for $r'$ a strict multiple of $r$, then the adjoint line bundle would have an $(r')^{th}$ root. This leaves the purely group-theoretic problem of promoting this ``non-containment'' statement to the stronger assertion $\Gamma_L = \Mod(E)[\phi_M]$, which we handle in this section, as \Cref{lemma:nonconmax}. This requires the preliminary \Cref{lemma:rspinchar}. 

    \begin{lemma}
        \label{lemma:rspinchar}
        Let $\phi$ be an $r$-spin structure on a closed surface $S$, and let $\Gamma \le \Mod(S)$ be a subgroup. Suppose that for all $\phi$-admissible curves $c \subset S$ and all $f \in \Gamma$, the equality
        \[
        \phi(f(c)) = \phi(c) \pmod r
        \]
        holds. Then $\Gamma \le \Mod(S)[\phi]$.
    \end{lemma}
    \begin{proof}
        By definition, $\Gamma \le \Mod(S)[\phi]$ if and only if $\phi(f(c)) = \phi(c) \pmod r$ for {\em all} simple closed curves, not just admissibles. We will show that the weaker condition implies the stronger by induction on (an integer lift of) the value $k = \abs{\phi(c)}$; the base case $k = 0$ holds by hypothesis.

        Let $c \subset S$ satisfying $\phi(c) = k \pmod r$ be given, and suppose that $\phi(f(d)) = \phi(d)$ for all simple closed curves $d \subset S$ such that $\abs{\phi(d)} \le k-1$ (for some integer lift of the residue class $\phi(d)$). Firstly, if $c$ is separating, then $\phi(c)$ is already specified by homological coherence (\Cref{lemma:WNFprops}.2), and so $\phi(f(c)) = \phi(c)$ holds automatically. If $c$ is non-separating, then by the framed change-of-coordinates principle (\Cref{prop:ccp}), there is a pair of pants $P$ with boundary components $a,b,c$, where $\phi(a) = 0, \phi(b) = 1-k$, and $\phi(c) = k$ (with $a,b,c$ oriented so that $P$ lies to the right). By induction, $\phi(f(a)) = 0$ and $\phi(f(b)) = 1-k$; applying homological coherence to $f(P)$ shows that $\phi(f(c)) = k$ as desired.
    \end{proof}

    \begin{lemma}
        \label{lemma:nonconmax}
        Let $S$ be a closed surface of genus $g \ge 3$, let $\Gamma \le \Mod(S)$ be given, and suppose there are containments
        \[
        \Mod(S)[\phi'] \le \Gamma \le \Mod(S)[\phi],
        \]
        where $\phi$ is an $r$-spin structure and $\phi'$ is an $r'$-spin structure, for $r,r' > 0$ and $r \mid r'$. Suppose that moreover, $\Gamma$ is not contained in any $\Mod(S)[\phi'']$ for any $r''$-spin structure for any proper multiple $r''$ of $r$. Then $\Gamma = \Mod(S)[\phi]$.
    \end{lemma}
    \begin{proof}
            Let $\cT_\phi \le \Mod(S)[\phi]$ denote the {\em admissible subgroup} of $\Mod(S)[\phi]$, the subgroup generated by all admissible twists $T_a$ for $a \subset S$ nonseparating and satisfying $\phi(a) = 0 \pmod{r}$; define $\cT_{\phi'}$ and $\cT_{\phi''}$ similarly. Since $g \ge 3$, \cite[Proposition 5.1]{strata2} shows that there is an equality $\cT_\phi = \Mod(S)[\phi]$ (and similarly for $\phi'$ and $\phi''$). 
            
            To show $\Gamma = \Mod(S)[\phi]$, it therefore suffices to show that $\cT_\phi \le \Gamma$, i.e. that $T_a \in \Gamma$ for any $a \subset S$ nonseparating and satisfying $\phi(a) = 0 \pmod r$. Since $\phi'$ is a $\Z/r'\Z$-refinement of $\phi$, any such $a$ satisfies $\phi'(a) = kr \pmod{r'}$ for some $k$. Consider the orbit $\Gamma a$ of some $a$ such that $\phi'(a) = 0$. Since $\Mod(S)[\phi'] \le \Gamma$ and by \cite[Lemma 5.8]{strata2}, $\Mod(S)[\phi']$ acts transitively on the set of nonseparating curves on $S$ of given $\phi'$-winding number, it follows that 
            \[
            \Gamma a = \{c \subset S\mid \phi'(c) = k_i r \pmod{r'}, \mbox{ $c$ nonseparating}\}
            \]
            for certain integers $k_i$, i.e. that $\Gamma a$ is a union of orbits of curves under $\Mod(S)[\phi']$.

            Let $\cC_k$ denote one such orbit, the set of simple closed curves
            \[
            \cC_k = \{c \subset S\mid \phi'(c) = kr \pmod{r'}, \mbox{ $c$ nonseparating}\}.
            \]
            We claim that if $\Gamma a$ contains $\cC_{k_1}$ and $\cC_{k_2}$, then it also contains $\cC_{k_1+k_2}$. By the framed change-of-coordinates principle (\Cref{prop:ccp}), there is $c_1 \in \cC_{k_1}$ and $c_2 \in \cC_{k_2}$ such that $i(c_1, c_2) = 1$. Since $\phi'(a) = 0$ by assumption and $\Mod(S)[\phi'] \le \Gamma$ (and in particular, $T_a \in \Gamma$), it follows that $T_c \in \Gamma$ for any $c \in \Gamma a$. Applying this to $c_1$ as above, we find that $T_{c_1}(c_2) \in \Gamma a$. By twist-linearity (\Cref{lemma:WNFprops}.1), $\phi'(T_{c_1}(c_2)) = \phi'(c_1) + \phi'(c_2) = k_1r+k_2r$, so that $\Gamma a$ contains one member of $\cC_{k_1+k_2}$ and hence the entire set, as claimed.

            Let $r'' = \gcd\{k_i r\}$, and let $\phi''$ be the $r''$-spin structure obtained by reducing $\phi'$ mod $r''$. We have therefore shown that $\Gamma a$ is the set of $\phi''$-admissible curves. Thus the condition
            \[
            \phi''(f(c)) = \phi''(c) \pmod{r''}
            \]
            holds for all $\phi''$-admissible curves $c$ and all $f \in \Gamma$. By \Cref{lemma:rspinchar}, $\Gamma \le \Mod(S)[\phi'']$. Since $\Gamma$ is not contained in any subgroup of the form $\Mod(S)[\phi'']$ for any $r''$ a proper multiple of $r$, we conclude that $r'' = r$ and hence $\phi'' = \phi$. Thus $\Gamma a$ is the set of all $\phi$-admissible curves, and so $\cT_\phi \le \Gamma$, implying $\Gamma = \Mod(S)[\phi]$ as claimed.
    \end{proof}

\section{Sections of (sums of) line bundles}\label{S:smoothing}
 The goal of this section is to discuss some of the basic topological features of (families of) smooth sections of line bundles on algebraic surfaces, particularly those that arise by amalgamating sections in two different line bundles.
We begin in \Cref{SS:groupoid} with a definition of monodromy in the language of groupoids, which will be useful for dealing with different basepoints. \Cref{SS:rspin} specializes to monodromy in the context of a linear system on a smooth algebraic surface, showing how the adjunction formula implies this is contained in an $r$-spin mapping class group. In \Cref{SS:PL}, we recall the basic elements of Picard-Lefschetz theory and the theory of vanishing cycles.
In \Cref{SS:CDE}, we discuss the sections of $\abs{L_1 \otimes L_2}$ that arise by smoothing the singular section $C \cup D$, and in \Cref{SS:stabilizer}, we discuss the various stabilizer subgroups of the monodromy that are associated to such a section. Finally in \Cref{SS:smoothing} we give the technical details of smoothing in families, and in \Cref{SS:basepoints}, we discuss the ramifications of this construction for the problem of choosing a global basepoint.

\subsection{The monodromy groupoid}\label{SS:groupoid} We begin with a slightly idiosyncratic treatment of the monodromy of a surface bundle, giving a basepoint-independent formulation via groupoids. Our later analysis will involve various local computations, necessarily at distinct basepoints, and this language will allow us to compare these local computations in a precise and meaningful way.

\begin{definition}
    Let $p: E \to B$ be a fiber bundle in the smooth category, with fibers diffeomorphic to some closed surface $\Sigma_g$ endowed with a preferred orientation. Let $\cF(E/B)$ be the category (indeed, groupoid) with object set $B$, and with morphism sets $\Hom_{\cF(E/B)}(b, b')$ given by the set of isotopy classes of orientation-preserving diffeomorphisms $f: F_b \to F_{b'}$ between the fibers.
\end{definition}

    Let $\gamma: I \to B$ be a path connecting $b$ to $b'$. Then the pullback bundle $\gamma^*(E)$ over the contractible base $I$ is trivial, inducing a well-defined isotopy class of {\em parallel transport diffeomorphism}
    \[
    \tau_\gamma: F_b \to F_{b'}.
    \]
    It is not hard to see that if $\gamma': I \to B$ is homotopic to $\gamma$ rel endpoints, then $\tau_{\gamma'}$ is isotopic to $\tau_{\gamma}$. We therefore obtain a functor
    \[
    \rho: \Pi(B) \to \cF(E/B)
    \]
    from the fundamental groupoid of $B$ to $\cF(E/B)$, acting via the identity on objects, and sending $[\gamma] \in \Hom_{\Pi(B)}(b,b')$ to the isotopy class of $\tau_\gamma$. 

    \begin{definition}[Monodromy functor/groupoid]
    Let $p: E \to B$ be a $\Sigma_g$ bundle. The {\em monodromy functor} is the functor $\rho: \Pi(B) \to \cF(E/B)$ described above, and the {\em monodromy groupoid} $\cM(E/B)$ is the image, the groupoid with object set $B$ and with $\Hom_{\cM(E/B)}(b,b')$ the set of isotopy classes of parallel transport maps $\tau_{\gamma}: F_b \to F_{b'}$, indexed by $\gamma \in \Hom_{\Pi(B)}(b,b')$.    
    \end{definition}
    With this notion in place, the monodromy {\em representation} of $p: E \to B$ with respect to a chosen basepoint $b \in B$ can be defined as the restriction of $\rho$ to the automorphism set $\pi_1(B,b) = \Hom_{\Pi(B)}(b,b)$, taking values in $\Hom_{\cF(E/B)}(F_b,F_b) = \Mod(F_b)$. As is customary, we write
    \[
    \rho_b: \pi_1(B,b) \to \Mod(F_b).
    \]

    \subsection{Topological monodromy of a linear system and $r$-spin structures}\label{SS:rspin}
    We now specialize to the setting of interest. Let $X$ be a smooth projective surface, let $L \in \Pic(X)$ be a line bundle, and let $\abs{L}$ be the complete linear system, with $U_L \subset \abs{L}$ the open locus of smooth sections. Define the tautological family $\cC_L \to U_L$ by
    \[
    \cC_L = \{(C,x) \mid C \in U_L,\ x \in C\} \subset U_L \times X.
    \]
    
    This is a $\cC^\infty$ fiber bundle, giving rise to a monodromy functor
    \[
    \rho_L: \Pi(U_L) \to \cF(\cC_L/U_L).
    \]
    Choosing a basepoint $E \in U_L$, we obtain a monodromy representation
    \[
    \rho_{L,E}: \pi_1(U_L,E) \to \Mod(E);
    \]
    the image
    \[
    \Gamma_L := \im(\rho_{L,E}) \le \Mod(E)
    \]
    is the {\em topological monodromy group} (suppressing the dependence on $E$ will not impose any complications).

    \para{The adjunction formula} The algebro-geometric origin of this bundle imposes a nontrivial constraint on $\Gamma_L$. Recall that the adjunction formula states that for a smooth section $C \in U_L$,
    \[
    \omega_C = (\omega_X \otimes L)|_C,
    \]
    where $\Omega_C, \omega_X$ denote the canonical bundles of $C, X$, respectively. This has the following consequence for the topological monodromy group $\Gamma_L$. For a proof (written in the setting of $X = \CP^2$ but applicable in full generality), see \cite{quintic}.
    \begin{lemma}\label{lemma:rspincontain}
        Let $X$ be a smooth projective algebraic surface, let $L \in \Pic(X)$ be a line bundle, let $C \in U_L$ be a smooth section, and let $\Gamma_L \le \Mod(C)$ be the topological monodromy group. Suppose the adjoint line bundle $\omega_X \otimes L$ admits an $r^{th}$ root $M \in \Pic(X)$. Then 
        \[
        \Gamma_L \le \Mod(C)[\phi_{M}],
        \]
        where $\Mod(C)[\phi_{M}] \le \Mod(C)$ is the $r$-spin mapping class group associated to the distinguished root $M|_C$ of $\omega_C$.
    \end{lemma}

    \begin{remark}
        \Cref{lemma:rspincontain} is valid even in the setting when the group $\Pic_0(X)$ of line bundles of degree $0$ has nontrivial $r$-torsion, in which case the root $M$ of $\omega_X \otimes L$ is not unique. In such a situation, the conclusion is that $\Gamma_L$ must be contained in the {\em intersection} of the $r$-spin mapping class groups associated to each such root. However, in this paper we consider only simply connected $X$, in which case $\Pic_0(X)$ is trivial. See \cite{ishanpi1} for a further exploration of topological monodromy in the non-simply connected case. 
    \end{remark}
    
    \subsection{Nodal degenerations, vanishing cycles and Picard-Lefschetz theory}\label{SS:PL}
    Here we briefly recall these notions and how they relate to one another.
\begin{definition}[Nodal degeneration]
   Let $L$ be a line bundle on a smooth projective surface $X$. Let $f \in H^0(X, L)$, with zero locus $Z(f) = C$, assumed to be smooth. 
   
    Let $\disk \subset \C$ denote the unit disk, and let $\varphi: \disk \to H^0(X, L) $ determine a family, with $C_t$ the zero locus of $\varphi(t)$. 
We say that $\varphi$ is a {\em nodal degeneration} of $C$ if it satisfies the following conditions:
   
   \begin{enumerate}
       \item  $C_1 = C$.
       \item  The curves $C_t$ are smooth for $t \neq 0$.
       \item  The curve $C_0$ has a single node at some point $p \in C_0$.
       \item $\frac{\partial }{\partial  \bar t} (\varphi) (p)|_{t =0} =0$ and $\frac{\partial}{\partial t} (\varphi) (p)|_{t =0}  \neq 0 .$
   \end{enumerate}
\end{definition}
A nodal degeneration as above determines a family of smooth curves over $\disk \setminus \{0\}$ arising from pulling back the universal family along $\varphi$. Topologically, near the nodal point, the curves $C_t$ have embedded {\em cylinders} of radius shrinking to $0$ with $t$. The core curve $\alpha$ of such a cylinder is a simple closed curve called the {\em vanishing cycle} of the nodal degeneration. If $C_0$ is irreducible, $\alpha$ is nonseparating. 
The {\em Picard-Lefschetz formula} states that the monodromy associated to this family is given by a right-handed Dehn twist $T_\alpha$ about the vanishing cycle. This leads to the fundamental connection between vanishing cycles and admissible simple closed curves.

\begin{lemma}
    Let $\alpha \subset C$ be a nonseparating vanishing cycle. Then $\alpha$ is admissible for the $r$-spin structure $\phi_{M}$ associated to the maximal root $M$ of the adjoint line bundle $\omega_X \otimes L$.
\end{lemma}
\begin{proof}
    By the Picard-Lefschetz formula, $T_\alpha \in \Gamma_L$. By \Cref{lemma:rspincontain}, $\Gamma_L \le \Mod(C)[\phi_{M}]$, and by \Cref{lemma:admissibletwists}, $\alpha$ must be admissible for $\phi_{M}$.
\end{proof}

\begin{remark} We emphasize that the vanishing cycle of a nodal degeneration is only well-defined {\em locally}, or said differently, depends globally on a choice of nodal degeneration connecting a given smooth curve $C_1$ to a nodal curve $C_0$. Different choices of degeneration from $C_1$ to $C_0$ can lead to different curves $\alpha, \alpha'$ on $C_1$ appearing as the vanishing cycle.
\end{remark}

    \subsection{The topological model: $\tilde C, \tilde D$, and $E$}\label{SS:CDE} In the remainder of this section, we specialize our working environment to the case where $L = L_1 \otimes L_2$. Suppose $C \in U_{L_1}$ and $D \in U_{L_2}$ are smooth sections intersecting transversely. Then $C \cup D$ is a reducible element of $\abs{L_1 \otimes L_2}$ with $d = C\cdot D$ nodal singularities. Let $E \in \abs{L_1 \otimes L_2}$ be a small generic perturbation of $C \cup D$. Then $E$ is smooth, and topologically decomposes as 
    \[
    E \cong \tilde C \cup \tilde D.
    \]
    Here, $\tilde C$ is the surface with $d$ boundary components $\Delta_1, \dots, \Delta_d$ obtained by deleting neighborhoods of the points of $C \cap D$, and $\tilde D$ is obtained from $D$ similarly. Alternatively, $\tilde C$ can be viewed as the real oriented blowup of $C$ at the points of $C \cap D$, and similarly for $\tilde D$. The subsurfaces $\tilde C$ and $\tilde D$ are joined together along their common boundary components to form $E$. See \Cref{fig:CDE} for a cartoon depiction of $C,D, \tilde C, \tilde D, E$, and their relationship, and see \Cref{def:smoothing} for the precise meaning of ``small generic perturbation''.

        \begin{figure}[h]
\centering
		\labellist
        \small
        \pinlabel $C$ at 70 172
        \pinlabel $D$ at 220 172
        \pinlabel $C$ at 460 220
        \pinlabel $D$ at 420 245
        \pinlabel $\tilde C$ at 70 20
        \pinlabel $\tilde D$ at 200 20
        \pinlabel $E$ at 460 70
        \pinlabel $\Delta_1$ at 135 122
        \pinlabel $\Delta_2$ at 135 95
        \pinlabel $\Delta_3$ at 135 67
        \pinlabel $\Delta_4$ at 135 37
        \pinlabel $\Delta_5$ at 135 9
		\endlabellist
\includegraphics[width=\textwidth]{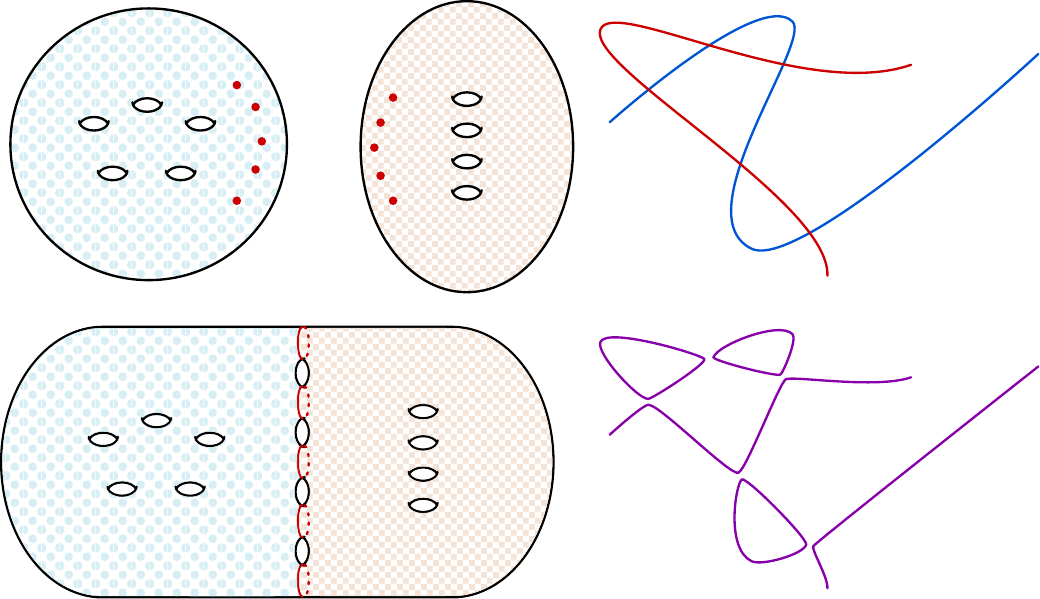}
\caption{Smoothing transversely-intersecting $C \in U_{L_1}$ and $D \in U_{L_2}$ (shown in color versions in blue and red, respectively) to a section $E = \tilde C \cup \tilde D$ of $\abs{L_1 \otimes L_2}$. }
\label{fig:CDE}
\end{figure}
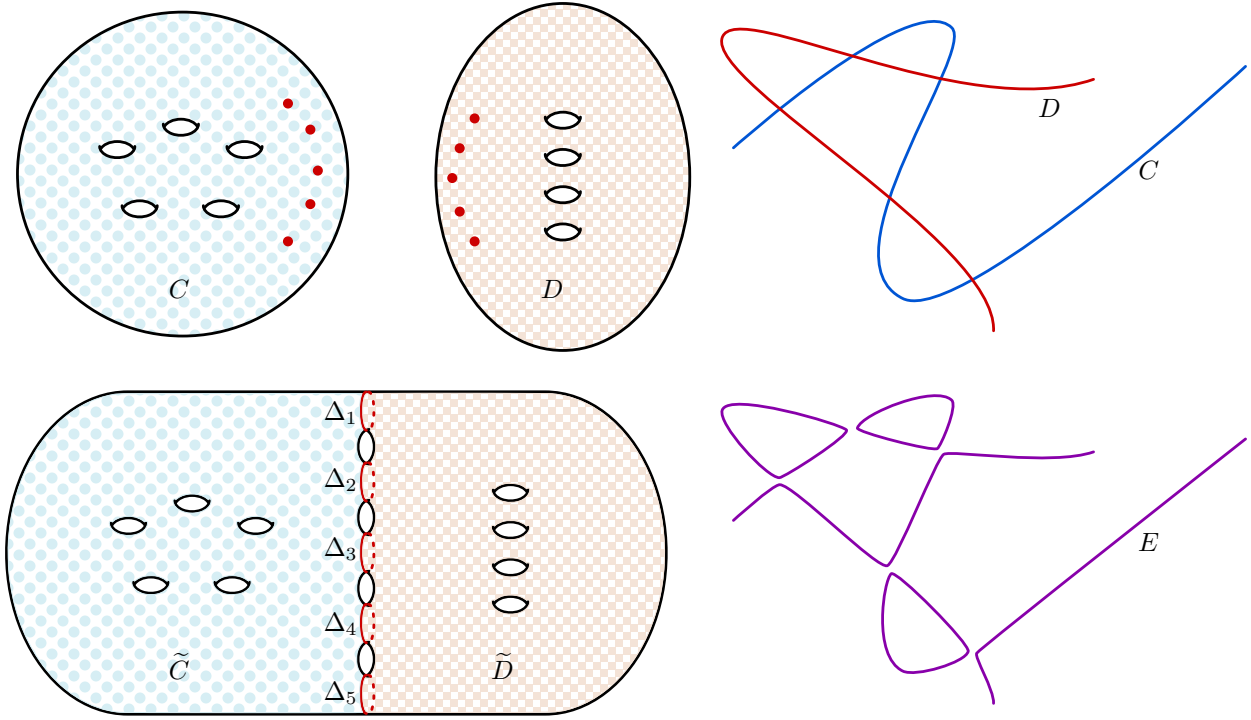
    
    \subsection{The (principal) stabilizer}\label{SS:stabilizer}
    Our analysis of the monodromy of the linear system $\abs{L_1 \otimes L_2}$ will make extensive use of sections $E = \tilde C \cup \tilde D$ obtained by joining together sections of $L_1$ and $L_2$. Accordingly, the monodromy will have a distinguished subgroup of elements that topologically preserve this decomposition.
    
    \begin{definition}[Stabilizer, principal stabilizer, restriction maps]\label{def:stab}
        The {\em stabilizer} of $\tilde C$, written $\Gamma_0[\tilde C]$, is the subgroup of $\Gamma_{L} \le \Mod(E)$ that preserves $\partial \tilde C = \Delta_1\cup \dots \cup \Delta_d$ as an oriented multicurve, and hence preserves $\tilde C \subset E$. Since the individual $\Delta_i$ are not necessarily fixed, there is not a restriction map to the mapping class group of $\tilde C$ (as this group fixes the boundary pointwise by definition), but there is a good substitute. Let $(C,\bx)$ be the marked surface obtained from $\tilde C$ by capping off each $\Delta_i$ with a marked disk. Then there is a restriction map 
        \[
        \bar \res: \Gamma_0[\tilde C] \to \Mod(C, \bx). 
        \]

        The {\em principal stabilizer},\footnote{As a mnemonic, the terminology and notation here are chosen to imitate the standard language in the theory of congruence subgroups of $\SL_2(\Z)$.} written $\Gamma[\tilde C]$, is the subgroup of $\Gamma_0[\tilde C]$ that restricts to the identity on $\tilde D$. In particular, each $\Delta_i \in \partial \tilde C$ is individually fixed, and so there is an injective restriction homomorphism 
        \[
        \res: \Gamma[\tilde C] \to \Mod(\tilde C). 
        \]
    \end{definition}

    \begin{remark}\label{remark:CDsymmetric}
        There are of course analogous definitions of $\Gamma_{(0)}[\tilde D]$ for the other subsurface $\tilde D \subset E$. In fact, the definition of $\Gamma_0[\tilde C]$ is formulated completely symmetrically in terms of $\tilde C, \tilde D$, so that $\Gamma_0[\tilde C] = \Gamma_0[\tilde D]$. The principal stabilizers, however, are distinct subgroups.
    \end{remark}

The following lemma and its corollary are straightforward consequences of the definitions.
    \begin{lemma}\label{lemma:principalnormal}
        $\Gamma[\tilde C]$ is a normal subgroup of $\Gamma_0[\tilde C]$. 
    \end{lemma}

    \begin{corollary}
        \label{lemma:againVC}
        Let $a \subset \tilde C$ be a vanishing cycle for $L_1 \otimes L_2$, and let $\beta \in \Gamma_0[\tilde C]$ be arbitrary. Then $\beta(a) \subset \tilde C$ is again a vanishing cycle for $L_1 \otimes L_2$.  
    \end{corollary}

    \subsection{Loci of smoothings}\label{SS:smoothing}
     In this section we establish some technical constructions and results about families of curves obtained by smoothing $C \cup D$, for $C \in U_{L_1}$ and $D \in U_{L_2}$ intersecting transversely. Throughout this discussion, $X$ is tacitly endowed with a Riemannian metric, and all vector spaces are endowed with a norm $\norm{\cdot}$. For $x \in X$, we write $\bar B(\epsilon, x) \subset X$ for the closed $\epsilon$-ball in $X$ centered at $x$ (relative to the chosen metric on $X$), and $B(\epsilon, x) \subset \bar B(\epsilon, x)$ denotes its interior. The following gives a precise construction of how to smooth a reducible curve of the form $C \cup D$: the transversality condition \eqref{eq:transverse} quantifies the intuition that the perturbation should be small enough that the smoothed curve does not stray too far from the original $C \cup D$.

    \begin{definition}[Smoothing, $\tilde U_{L_1\otimes L_2}(C, D, \epsilon, \eta)$]\label{def:smoothing}
        Let $C \in U_{L_1}$ and $D \in U_{L_2}$ be transversely-intersecting smooth curves, given as the vanishing loci of $f \in H^0(X;L_1)$ and $g \in H^0(X;L_2)$, respectively. A {\em smoothing} of $C \cup D$ is a section of $L_1 \otimes L_2$ of the form $fg+h$, for $h \in H^0(X;L_1\otimes L_2)$, such that $Z(fg+h)$ is smooth, and that moreover satisfies the following condition: there is $\epsilon > 0$ such that the closed $\epsilon$-balls $\bar B(\epsilon, x_i) \subset X$ are disjoint as $x_i$ ranges over the points of $C \cap D$, and the intersection
        \begin{equation}
            \label{eq:transverse}
            Z(fg+th) \pitchfork \partial \bar B(\epsilon, x_i)
        \end{equation}
        is {\em transverse} for all $x_i \in C \cap D$ and all $t \in [0,1]$.

        For $C \in U_{L_1}$ and $D \in U_{L_2}$ intersecting transversely, define $\tilde U_{L_1\otimes L_2}(C, D, \epsilon, \eta) \subset U_{L_1\otimes L_2}$ as the locus of all smoothings $fg+h$ for the indicated $\epsilon$ and for $\norm{h} < \eta$. Note that $\tilde U_{L_1\otimes L_2}(C, D, \epsilon, \eta) \subset U_{L_1\otimes L_2}$ is the complement of the discriminant hypersurface in $H^0(X;L_1\otimes L_2)$ in the $\eta$-ball around $fg$; in particular, it is connected. To streamline notation, if the particular $\epsilon, \eta$ are not important, they will be suppressed, and we will write $\tilde U_{L_1 \otimes L_2}(C,D)$.
    \end{definition} 

    We will frequently consider {\em families} of smoothings, i.e. smoothings of $C_t \cup D$, for fixed $D \in U_{L_2}$, and $C_t \in U_{L_1}$ varying through smooth curves in $|L_1|$ transverse to $D$. The following definition establishes these ``universal smoothings''.

    \begin{definition}[$U_{L_1}(D)$ and $\tilde U_{L_1}(D)$]\label{def:globalsmoothing}
        Fix $D \in U_{L_2}$ defined as $D = Z(g)$ for $g \in H^0(X;L_2)$. Define
        \[
        U_{L_1}(D) = \{ C \in U_{L_1} \mid C \pitchfork D \mbox{ is transverse}\}
        \]
        as the locus of smooth curves in $\abs{L_1}$ transverse to $D$. In order to define $\tilde U_{L_1}(D) \subset U_{L_1 \otimes L_2}$ a locus of smoothings of $C \cup D$ for $C \in U_{L_1}(D)$, we must allow the parameters $\epsilon$ and $\eta$ in the definition of a smoothing to vary with $C \in U_{L_1}(D)$. Accordingly, let 
        \[
        \epsilon: U_{L_1}(D) \to \R_{>0}
        \]
        be a continuous function such that for any $C \in U_{L_1}(D)$, the balls $\bar B(\epsilon(C), x_i(C)) \subset X$ are disjoint, where $x_i(C)$ ranges over the points of intersection $C \cap D$. Define
        \[
        \pi^{-1}(U_{L_1}(D)) \subset H^0(X;L_1)\setminus\{0\}
        \]
        as the preimage of $U_{L_1}(D) \subset \abs{L_1} = \P H^0(X;L_1)$ under the projectivization map, and note that $\epsilon$ lifts to a function on $\pi^{-1}(U_{L_1}(D))$. Let
        \[
        \eta: \pi^{-1}(U_{L_1}(D)) \to \R_{>0}
        \]
        be a continuous function such that the intersections
        \[
        Z(fg+th) \pitchfork \partial \bar B(\epsilon(f),x_i(f))
        \]
        are transverse for all $x_i(f) \in Z(f) \cap D$, all $t \in [0,1]$, and all $h \in H^0(X;L_1\otimes L_2)$ such that $\norm{h} < \eta(f)$. With these established, define
        \begin{align*}
            \tilde U_{L_1}(D) &= \{(f,h) \mid f \in \pi^{-1}(U_{L_1}(D)), Z(fg+h) \in \tilde U_{L_1 \otimes L_2}(Z(f),D,\epsilon(f), \eta(f))\}\\ &\subset \pi^{-1}(U_{L_1}(D)) \times H^0(X;L_1\otimes L_2);
        \end{align*}
        we will frequently identify $\tilde U_{L_1}(D)$ with its image in $U_{L_1 \otimes L_2}$ under the map $(f,h) \mapsto Z(fg+h)$. Since $U_{L_1}(D)$ and $\tilde U_{L_1 \otimes L_2}(C,D, \epsilon, \eta)$ are connected, so is $\tilde U_{L_1}(D)$.
    \end{definition} 

Again note that there are analogous constructions with the roles of $C$ and $D$ exchanged.
     
\subsection{Changing basepoints}\label{SS:basepoints}
    To study the monodromy of $\cC_L/U_L$, we will amalgamate various local monodromy computations, necessarily carried out at different basepoints.
    In practice, it will not be possible to specify a completely explicit path $\gamma$ connecting chosen basepoints $b, b'$. {\em A priori} this leads to a large ambiguity in the monodromy group: given $\alpha \in \pi_1(U_L,E)$, only the {\em conjugacy class} determined by $\rho_{L,E}(\alpha)$ in $\im(\rho_{L,E'}) \le \Mod(E')$ can be specified unambiguously, and if the monodromy group is not already understood, this is not terribly helpful. 
    
    Our solution will be to describe certain (connected, but not simply connected) {\em loci} $W \subset U_L$ for which some information about the ``$W$-local monodromy group'' 
    \[
    \rho_{L,W,E}: \pi_1(W,E) \to \Mod(E)
    \]
    can be obtained. If $E, E'$ are basepoints that are both contained in $W$, then by choosing a path $\gamma$ connecting $E$ and $E'$ in $W$, the ambiguity {\em shrinks} to the potentially more manageable conjugacy class within $\im(\rho_{L,W,E}) \le \Mod(E)$. In favorable circumstances, the $W$-local monodromy group can be understood completely, thereby resolving the ambiguity and leaving us free to choose basepoints freely within $W$.

    The loci $W$ we will use are loci of smoothings as discussed above in \Cref{SS:smoothing}, both the local version with both $C, D$ fixed (as in \Cref{def:smoothing}), and the more global version with $D$ fixed but $C$ varying (as in \Cref{def:globalsmoothing}). We have the following descriptions of the local monodromy groups.

 \begin{lemma}
        \label{lemma:TDelta}
        Let $L_1,L_2$ be very ample line bundles on $X$, and let $C$ and $D$ be smooth sections of $L_1, L_2$ respectively, intersecting transversely at $d$ points. Let $\tilde U_{L_1 \otimes L_2}(C,D,\epsilon, \eta)$ be the space of smoothings as in \Cref{def:smoothing}, and let $E$ be a basepoint. Then for $\epsilon, \eta$ both sufficiently small, the local monodromy
        \[
        \rho_{C,D}: \pi_1(\tilde U_{L_1 \otimes L_2}(C,D,\epsilon, \eta), E) \to \Mod(E)
        \]
        is the subgroup generated by the Dehn twists about the boundary components $\Delta_1, \dots, \Delta_d$ separating the subsurfaces $\tilde C, \tilde D$ of $E$, as in \Cref{SS:CDE}. In particular, each $\Delta_i$ is a vanishing cycle.
    \end{lemma}

    \begin{proof}
        This was established in \cite[Section 6.2]{CImonodromy} - see in particular \cite[Lemma 6.7]{CImonodromy}. The results there are formulated in the setting of complete intersection surfaces, but hold {\em mutatis mutandis} replacing the line bundles $\cO(d_{n-1}), \cO(1)$, and $\cO(d_{n-1}+1)$ discussed there with arbitrary very ample line bundles $L_1, L_2$, and $L_1 \otimes L_2$.
    \end{proof}

Now fix $D \in U_{L_2}$, and let $\tilde U_{L_1}(D) \subset U_{L_1 \otimes L_2}$ be the associated universal smoothing. 

\begin{lemma}
    \label{lemma:ptsprt}
    Let $\gamma \in \Pi(\tilde U_{L_1}(D))$ be a path connecting $E = \tilde C \cup \tilde D$ to $E' =\tilde C' \cup \tilde D$. Then the associated parallel transport $\tau_\gamma: E \to E'$ restricts to
    \[
    \qquad (\tau_\gamma)\mid_{\tilde C}: \tilde C \to \tilde C'\qquad \mbox{and} \qquad (\tau_\gamma) \mid_{\tilde D}: \tilde D \to \tilde D ,
    \]
    well-defined at the level of isotopy classes of subsurfaces of $E$ and $E'$. In particular, $\tau_\gamma$ induces a bijection between the sets of vanishing cycles supported on $\tilde C \subset E$ and on $\tilde C' \subset E'$.
\end{lemma}
\begin{proof}
    By the constructions of \Cref{def:smoothing,def:globalsmoothing}, for every $fg+h \in \tilde U_{L_1}(D)$, the intersection of $Z(fg+h)$ with $\partial \bar B(\epsilon, x_i(f))$ is transverse, and so forms a locally trivial family over $\tilde U_{L_1}(D)$. Each such intersection is a Hopf link in $\partial \bar B(\epsilon, x_i(f)) \cong S^3$; in $E = \tilde C \cup \tilde D$, the two components form parallel copies of the boundary curves $\Delta_i$ in $E$. Thus this set of curves is preserved under parallel transport within $\tilde U_{L_1}(D)$: if $\Delta_1 \cup \dots \cup \Delta_d$ is the boundary of $\tilde C \subset E$ and $\Delta'_1 \cup \dots \cup \Delta'_d$ is the boundary of $\tilde C' \subset E'$, and if $\gamma$ is a path connecting $E = \tilde C \cup \tilde D$ to $E' = \tilde C' \cup \tilde D$, then 
    \[
    \tau_\gamma(\Delta_1 \cup \dots \cup \Delta_d) = \Delta'_1 \cup \dots \cup \Delta'_d,
    \]
    which easily implies that $\tau_\gamma$ identifies $\tilde C$ to $\tilde C'$ and $\tilde D \subset E$ to $\tilde D \subset E'$ as isotopy classes of subsurfaces of $E, E'$. 
\end{proof}

\begin{corollary}
    \label{lemma:localmonodromy}
    Let $L_1,L_2$ be very ample line bundles on $X$. Let $D \in U_{L_2}$ be given, let $\tilde U_{L_1}(D) \subset U_{L_1\otimes L_2}$ be the universal smoothing constructed in \Cref{def:globalsmoothing}, and let $E \cong \tilde C \cup \tilde D$ be a basepoint. Then the image of the local monodromy
    \[
    \rho_D: \pi_1(\tilde U_{L_1}(D),E) \to \Mod(E)
    \]
    is contained in the stabilizer $\Gamma_0[\tilde C] = \Gamma_0[\tilde D]$ (recalling \Cref{remark:CDsymmetric}). 
\end{corollary}

With these results in hand, we can establish the basepoint convention that will be in effect for the remainder of the paper.

    \begin{convention}\label{conv}
        Fix a smooth curve $D \in \abs{L_2}$. Then the basepoint $E \in U_{L_1 \otimes L_2}$ is chosen to be any point in $\tilde U_{L_1}(D)$. Such $E$ has a topological description as $E = \tilde C \cup \tilde D$. By \Cref{lemma:ptsprt}, statements about vanishing cycles on the subsurfaces $\tilde C, \tilde D$ are valid independently of the chosen basepoint.
    \end{convention}

\section{(Simple) braid groups}\label{S:braids}

Fixing a section $C \in U_{L_1}$ and letting $D \in U_{L_2}(C)$ vary through smooth sections transverse to $C$, one obtains a family of configurations of $d$ points on $C$. Smoothing these out as in \Cref{S:smoothing}, this creates a family of smooth sections of $\abs{L_1 \otimes L_2}$ whose monodromy is valued in the stabilizer $\Gamma_0[\tilde C]$. This subgroup will play a prominent role in our arguments; in this section, we recall some of the basic language and theory of surface braid groups, which will allow us to formulate our results.

\subsection{Configuration spaces and surface braid groups}
Let $C$ be a closed surface.\footnote{There is a certain amount of tension in the notation in this section. On the one hand, our ultimate application of all of these results is to the Riemann surfaces $C,\tilde C \subset X$, etc. as discussed in \Cref{S:smoothing}. On the other hand, in certain places, it is clearer to explicitly keep track of boundary components and punctures, at which point we will switch to ``general'' notation, writing $\Sigma_{g,n}^d$ for the surface of genus $g$ with $n$ punctures and $d$ boundary components. We hope the reader is not too disoriented by this switch.} The configuration space of unordered $d$-tuples of points in $C$ will be denoted $\Conf_d(C)$, and the space of {\em ordered} $d$-tuples will be denoted $\PConf_d(C)$. The covering $\PConf_d(C) \to \Conf_d(C)$ is Galois with deck group $S_d$ the symmetric group on $d$ letters. The fundamental group of $\Conf_d(C)$ is the {\em surface braid group}, written
\[
\Br_d(C) := \pi_1(\Conf_d(C));
\]
likewise define the {\em pure surface braid group} via
\[
\PBr_d(C) := \pi_1(\PConf_d(C)).
\]

\para{Basic elements: simple half-twists, point-pushes} Here we recall two basic classes of elements in a surface braid group. Let $C$ be a surface endowed with a set $\bx = \{x_1, \dots, x_d\}$ of $d$ distinct points. Let $\sigma \subset C$ be an arc with endpoints at distinct elements $x_i, x_j$ of $\bx$. The {\em simple half twist} about $\sigma$ is the surface braid obtained by exchanging $x_i$ and $x_j$ by following $\sigma$, with each of $x_i, x_j$ displaced slightly to the right (see \Cref{fig:braids}). By abuse of notation, this will often be denoted by $\sigma$ as well. Note that the square $\sigma^2$ is a pure braid, and that as a mapping class, it is given by the Dehn twist $T_{\sigma^+}$, where $\sigma^+$ is the simple closed curve given as the boundary of an $\epsilon$-neighborhood of $\sigma$. 
Next, let $\ul{a}$ be an oriented simple closed curve based at some $x_k \in \bx$. Then the {\em point-push} of $x_k$ along $\ul{a}$ is the surface braid obtained by pushing $x_k$ along $\ul{a}$. It is again a pure braid.

    \begin{figure}[h]
\centering
		\labellist
        \tiny
        \pinlabel $\sigma$ at 70 80
        \pinlabel $\ul{a}$ at 70 40
        \pinlabel $\sigma^+$ at 160 40
		\endlabellist
\includegraphics{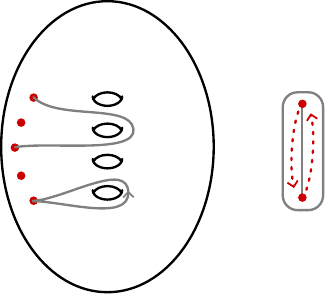}
\caption{At left: two basic examples of surface braids: a simple half-twist about an arc, and a point-push. At right, an illustration of the ``follow to the right'' convention used in defining simple half-twists, and the simple closed curve $\sigma^+$ associated to an arc $\sigma$.}
\label{fig:braids}
\end{figure}

\para{The Fadell-Neuwirth fibration} Basic structural information about pure configuration spaces is provided by the {\em Fadell-Neuwirth fibration}, introduced in \cite{FN}. 

\begin{proposition}[Fadell-Neuwirth, cf. \cite{FN}]\label{prop:FN}
    Let $\Sigma_{g,n}$ be a surface of genus $g\ge 0$ with $n \ge 0$ punctures. Let $p: \PConf_d(\Sigma_{g,n}) \to \PConf_{k}(\Sigma_{g,n})$ denote the projection onto the first $k$ components of an ordered configuration of $d$ points in $\Sigma_{g,n}$. Then $p$ realizes $\PConf_d(\Sigma_{g,n})$ as the total space of a fiber bundle over $\PConf_{k}(\Sigma_{g,n})$, with fiber $\PConf_{d-k}(\Sigma_{g,n+k})$.

    In particular, so long as $g + n \ge 1$, the long exact sequence of homotopy groups degenerates to a short exact sequence of fundamental groups
    \[
    1 \to \PBr_{d-k}(\Sigma_{g,n+k}) \to \PBr_{d}(\Sigma_{g,n}) \to \PBr_{k}(\Sigma_{g,n}) \to 1.
    \]
    There is an analogous short exact sequence when the punctured surfaces $\Sigma_{g,m}$ are replaced with surfaces with boundary $\Sigma_g^m$.
\end{proposition}
\begin{proof}
    The existence of the fiber bundle is established in \cite{FN}. The short exact sequence follows from the fact that $\pi_i(\PBr_k(\Sigma_{g,n})) = 0$ for $i \ge 2$ so long as $g + n \ge 1$; this fact is established by induction on $k$ using the long exact sequence in homotopy. The final claim follows from the fact that the spaces $\PConf_k(\Sigma_{g,n}) \simeq \PConf_k(\Sigma_g^n)$ are homotopy equivalent.
\end{proof}

\para{Configurations of tangent vectors/disks} We will have occasion to consider a variant of these notions. Let $\tilde\Conf_d(C)$ be the configuration space of $d$-tuples of {\em unit tangent vectors} $\{v_1, \dots, v_d\}$ in $C$ (tacitly relative to some Riemannian metric on $C$), such that $v_i \in T_{x_i} C$ and the basepoints $\{x_1, \dots, x_d\}$ are pairwise distinct. Note that there is an ordered variant of this construction, written $\tilde \PConf_d(C)$. The corresponding fundamental groups are written
\[
\tilde \Br_d(C) := \pi_1(\tilde \Conf_d(C)) \quad \mbox{and} \quad \tilde \PBr_d(C) := \pi_1(\tilde \PConf_d(C)).
\]
At the level of fundamental groups, these are just central extensions of ordinary surface braid groups.
\begin{lemma}
    \label{lemma:brextension}
    There are short exact sequences
    \[
    1 \to \Z^d \to \tilde \Br_d(C) \to \Br_d(C) \to 1
    \]
    and
    \[
    1 \to \Z^d \to \tilde \PBr_d(C) \to \PBr_d(C) \to 1.
    \]
\end{lemma}
\begin{proof}
   The projection $\tilde \Conf_d(C) \to \Conf_d(C)$ realizes $\tilde \Conf_d(C)$ as the total space of a $(S^1)^d$-bundle over $\Conf_d(C)$. As $\pi_k(\Conf_d(C)) = 0$ for $k \ge 1$, the long exact sequence in homotopy reduces to the first of the above short exact sequences of fundamental groups. The bundle $\tilde \PConf_d(C) \to \PConf_d(C)$ is the pullback of $\tilde \Conf_d(C) \to \Conf_d(C)$ along the covering map $\PConf_d(C) \to \Conf_d(C)$; the second short exact sequence follows.
\end{proof}

Now suppose that $\tilde C$ is a surface with $d$ boundary components $\Delta_1, \dots, \Delta_d$, and let $C$ be the closed surface obtained by filling in each $\Delta_i$. Define
\[
\Br(\tilde C) := \tilde \Br_d(C) \quad \mbox{and} \quad \PBr(\tilde C) := \tilde \PBr_d(C). 
\]
While initially defined in terms of configuration spaces of tangent vectors on $C$, we will also require a second construction of $\PBr(\tilde C)$, as a subgroup of the mapping class group of $\Mod(\tilde C)$.
\begin{lemma}
    \label{lemma:birman}
    Let $p: \Mod(\tilde C) \to \Mod(C)$ be the map induced by the inclusion $\tilde C \into C$. Then there is an isomorphism
    \[
    \PBr(\tilde C) \cong \ker(\Mod(\tilde C) \to \Mod(C)).
    \]
\end{lemma}
\begin{proof}
    This is an instance of the Birman exact sequence. See \cite[Theorem 9.1 and Section 4.2.5]{FM}. 
\end{proof}

\para{Boundary-adjacent curves and spheres} In several places, we will have need of the following technical notion.

\begin{definition}[Boundary-adjacent curve, boundary-adjacent sphere]\label{def:boundaryadj}
    Let $\tilde C$ be a compact surface with nonempty boundary. A subsurface $B \subset \tilde C$ is a {\em $k$-holed boundary-adjacent sphere} if $B \cong \Sigma_0^k$ as an abstract surface, and if $k-1$ of the boundary components of $B$ are also boundary components of $\tilde C$. The unique boundary component of such $B$ that is not a boundary component of $\tilde C$ is called a {\em boundary-adjacent curve}, and is said to {\em enclose} the remaining boundary components of $B$.
    A boundary-adjacent sphere (or its associated boundary-adjacent curve) is {\em maximal} if every boundary component of $\tilde C$ is a boundary component of $B$.
\end{definition}

    \begin{figure}[h]
\centering
		\labellist
        \small
		\endlabellist
\includegraphics{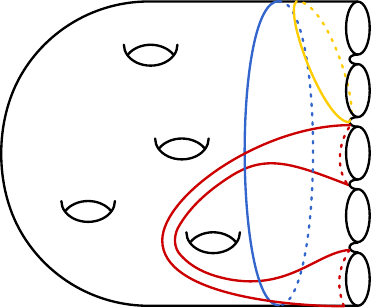}
\caption{Three boundary-adjacent curves/spheres, one of which is maximal.}
\label{fig:bdry}
\end{figure}

Our first use of boundary-adjacent spheres is in constructing a particular set of generators for certain surface braid groups; this will be used in \Cref{S:mainlemma}.

\begin{lemma}
    \label{lemma:bespokepbr}
    Let $\Sigma_g^n$ be a compact surface with boundary components $\Delta_1, \dots, \Delta_n$, and let $\PConf_d(\Sigma_g^n)$ be the configuration space of ordered $d$-tuples of points in $\Sigma_g^n$. Let $\bx = (x_1, \dots, x_d) \in \PConf_d(\Sigma_g^n)$ be a basepoint, and let $B' \subset B \subset \Sigma_g^n$ be maximal boundary-adjacent spheres such that $x_i \in B \setminus B'$ for all $x_i \in \bx$. Then $\PBr_d(\Sigma_g^n) = \pi_1(\PConf_d(\Sigma_g^n), \bx)$ has a generating set consisting of two classes of elements: point-push maps based at some $x_i \in \bx$ supported on $\Sigma_g^n \setminus B'$, and elements supported on $B$.
\end{lemma}
\begin{proof}
    We proceed by induction on $d$. In the base case $d = 1$ (where $\PConf_1(\Sigma_g^n) = \Sigma_g^n$ itself), the result is clear. In general, consider the Fadell-Neuwirth exact sequence
    \[
    1 \to \pi_1(\PConf_{d-1}(\Sigma_g^{n+1})) \to \PBr_d(\Sigma_g^n) \to \pi_1(\Sigma_g^n) \to 1
    \]
    associated to the map $\PConf_d(\Sigma_g^n) \to \Sigma_g^n$ given by projection onto the first coordinate. The fiber over $x_1 \in \Sigma_g^n$ has the homotopy type of $\PConf_{d-1}(\Sigma_g^{n+1})$, with the additional boundary component obtained by deleting a neighborhood of $x_1$. By induction, $\pi_1(\PConf_{d-1}(\Sigma_g^{n+1}))$ has a generating set of the required form. Let $a_1, \dots, a_{2g+n-1} \in \pi_1(C, x_1)$ be a set of generators. These can be lifted to elements of $\pi_1(\PConf_d(\Sigma_g^n), \bx)$ supported away from $B'$, completing a set of generators of the required type.
\end{proof}

\subsection{The simple braid group}\label{SS:SBr} The surface braid group has a subgroup of ``simple'' braids (introduced by Shimada \cite{shimada}), which will play an important role in our arguments. This is defined as follows. As above, $C$ denotes a closed surface. The inclusion
\[
\iota: \PConf_d(C) \into C^d
\]
gives rise to the homomorphism $\lambda: \PBr_d(C) \to H_1(C;\Z)$ via the following composition:
\[
\xymatrix{
\PBr_d(C) \ar[r]^{\iota_*} \ar@/_1pc/[rrr]_{\lambda}  & \pi_1(C^d) \ar[r] & H_1(C;\Z)^{\oplus d} \ar[r]^\sigma & H_1(C;\Z),
}
\]
with $\sigma: H_1(C;\Z)^{\oplus d} \to H_1(C;\Z)$ the map that sums the entries in the factors. As the target is abelian, $\lambda$ factors through $H_1(\PBr_d(C);\Z)$. Moreover, $\lambda$ is valued in the subgroup 
\[
\Hom(H_1(\PBr_d(C);\Z),H_1(C;\Z))^{S_d}
\]
of $S_d$-invariant homomorphisms, and so extends to a homomorphism
\[
\lambda: \Br_d(C) \to H_1(C;\Z).
\]
Topologically, $\lambda$ sends a surface braid $\beta$ on $C$ to the homology class on $C$ that it traces out. We note in passing that when $C$ is a Riemann surface, $\lambda$ can be interpreted as the {\em Abel-Jacobi map} that takes a divisor of $d$ points on $C$ to the associated line bundle (recall the space $\Pic_d(C)$ of line bundles of degree $d$ has $\pi_1(\Pic_d(C)) \cong H_1(C;\Z)$). 

\begin{definition}[Simple braid group]
    The {\em simple braid group} $\SBr_d(C)$ is the kernel of $\lambda$:
    \[
    \SBr_d(C) := \ker(\lambda: \Br_d(C) \to H_1(C;\Z)).
    \]
    We likewise define the {\em pure simple braid group}
    \[
    \PSBr_d(C) := \ker(\lambda: \PBr_d(C) \to H_1(C;\Z)).
    \]
    In the setting of a surface $\tilde C$ with boundary, the simple braid group is defined similarly:
    \[
    \mathrm{(P)SBr}(\tilde C) := \ker\left(\mathrm{(P)Br}(\tilde C) \to \mathrm{(P)Br}_d(C) \xrightarrow{\lambda} H_1(C;\Z)\right).
    \]
\end{definition}

We record the following simple but important observation: {\em twists about boundary-adjacent curves are pure simple braids}.
\begin{lemma}
    \label{lemma:boundaryadjacentsimple}
    Let $c \subset \tilde C$ be a boundary-adjacent curve. Then $T_c \in \PSBr(\tilde C)$.
\end{lemma}
\begin{proof}
    First, we observe that under the capping map $\tilde C \into C$, the curve $c$ becomes inessential, so that $T_c \in \ker(\Mod(\tilde C) \to \Mod(C)) \cong \PBr(\tilde C)$, applying \Cref{lemma:birman}. It is easy to see that $\lambda(T_c) = 0$, so that $T_c$ is moreover a simple braid, as claimed. 
\end{proof}

\subsection{Transitivity results}
While the simple braid group has infinite index in the full surface braid group, it nevertheless acts transitively on various topological structures on the surface. We record some results of this type here for later use.

\begin{lemma}
    \label{lemma:sbdisktrans}
    Let $\tilde C$ be a surface with boundary components $\Delta_1, \dots, \Delta_d$. For some $k < d$, let $B, B' \subset S$ be $k$-holed boundary-adjacent spheres. Then there exists $\beta \in \SBr(\tilde C)$ such that $\beta B = B'$. 
\end{lemma}
\begin{proof}
        By the usual change-of-coordinates principle \cite[Section 1.3]{FM}, there is $\beta' \in \Br(\tilde C)$ such that $\beta' B = B'$. We must find some such $\beta$ for which moreover $\lambda(\beta) = 0$. Let $G_B \le \Br(\tilde C)$ denote the stabilizer of $B$. To construct $\beta$ satisfying both constraints, it suffices to show that the restriction
    \[
    \lambda: G_B \to H_1(C; \Z)
    \]
    remains a surjection. This is easily shown. There are generators for $H_1(C;\Z)$ disjoint from $B$; convert these to based loops based at some $\Delta_i$ not enclosed by $B$. The corresponding point-push maps furnish a set of elements of $G_B$ surjecting onto $H_1(C;\Z)$ under $\lambda$.  
\end{proof}

A very similar argument (this time taking care to ensure that all braids involved are pure, when necessary) establishes the other transitivity result we will use.
\begin{lemma}
    \label{lemma:sbrarctrans}
    Let $\tilde C$ be a surface with $d \ge 3$ boundary components. Let $\alpha, \alpha'$ be simple arcs on $\tilde C$ such that $\alpha$ connects distinct components of $\partial \tilde C$ and the same holds for $\alpha'$. Then there exists $\beta \in \SBr(\tilde C)$ such that $\beta \alpha = \alpha'$. If the components of $\partial \tilde C$ associated to $\alpha$ are the same as those for $\alpha'$, then moreover there is $\beta \in \PSBr(\tilde C)$ such that $\beta \alpha = \alpha'$.
\end{lemma}

\section{Construction of monodromy elements}\label{S:monodromy}

In this section, we show that line bundles satisfying the hypotheses of \Cref{theorem:main} have monodromy groups rich enough to contain certain useful subgroups. We continue to work in the environment of \Cref{SS:CDE}, so that the line bundle $L = L_1 \otimes L_2$ decomposes into two summands each satisfying certain ampleness conditions, and our basepoint $E \in \abs{L_1 \otimes L_2}$ is obtained as a smoothing $E = \tilde C \cup \tilde D$ of smooth sections $C \in U_{L_1}$ and $D \in U_{L_2}$
intersecting transversely. 

\subsection{Simple braids}
The first class of monodromy subgroup we construct is a subgroup of the stabilizer $\Gamma_0[\tilde C]$, which moreover restricts on $\tilde C$ to the simple braid group discussed in the previous section. A similar construction was carried out in \cite{CImonodromy}. The key point is that when $L_2$ admits a Lefschetz pencil $p: X \dashrightarrow \CP^1$ (realized as a family $D_t \in \abs{L_2}$), and $C \subset X$ is a smooth curve disjoint from the base locus, then $p$ realizes $C$ as a branched cover of $\CP^1$, and moreover, the regular fibers of $p|_C$ form a family of configurations of $d = C \cdot D_t$ distinct points on $C$. After smoothing, this yields a subgroup of $\Gamma_0[\tilde C]$ realized on $C$ as surface braids. We first give a condition under which this setup can be realized.

\begin{definition}[Maximally generic pencil]\label{def:generic}
    A pencil $D_t$ of curves in $X$ is {\em maximally generic rel $C$} if the following conditions hold:
    \begin{enumerate}
        \item $C$ is disjoint from the base locus of $D_t$,
        \item $C$ intersects every singular fiber $D_{t_s}$ transversely and in the smooth locus,
        \item Each $D_t$ has at most one simple tangency with $C$.
        \item Let $D_{t_i}$ be a curve that intersects $C$ at $p$ with a simple tangency. Then there is a local coordinate $s$ for $\CP^1$ centered at $t_i$ and a local coordinate $z$ for $C$ centered at $p$ such that the equation of $D_t$ restricted to $p$ is $z^2 - s$.
    \end{enumerate}
\end{definition}

\begin{lemma}\label{lemma:maxgenexists}
    Let $L$ be a very ample line bundle on $X$, and let $C \subset X$ be any smooth curve. Let $D_0 \in \abs{L}$ be a smooth section intersecting $C$ transversely. Then there is a pencil $D_t$ of curves in $\abs{L}$, with $D_0 = D$, that is maximally generic rel $C$.
\end{lemma}
\begin{proof}
    See \cite[Lemma 6.9]{CImonodromy}, being careful to note that the roles of $C,D$ have been interchanged. The result there is stated for complete intersection curves $C, D$ on a smooth complete intersection surface, but the proof only uses that the restriction of $L$ to $C$ is very ample, which holds here by hypothesis.
\end{proof}

We apply this in our setting as follows. As usual, let $C \in U_{L_1}$ and $D \in U_{L_2}$ be smooth sections intersecting transversely, and let $E = \tilde C \cup \tilde D$ be a smoothing of the section $C \cup D \in \abs{L_1 \otimes L_2}$. Let $\tilde U_{L_2}(C) \subset U_{L_1 \otimes L_2}$ be the space of smoothings as in \Cref{def:globalsmoothing}, and recall from \Cref{lemma:localmonodromy} that the local monodromy $\rho_C: \pi_1(\tilde U_{L_2}(C)) \to \Mod(E)$ is contained in $\Gamma_0[\tilde C]$. Recall also from \Cref{SS:stabilizer} that there is a restriction map $\bar \res: \Gamma_0[\tilde C] \to \Mod(C, \bx)$, where one collapses the boundary components of $\tilde C$ to marked points and tracks the resulting mapping class. 

\begin{lemma}\label{lemma:simpleinmono}
    The image of the composition 
    \[
    \bar \res \circ \rho_C: \pi_1(\tilde U_{L_2}(C)) \to \Gamma_0[\tilde C] \to \Mod(C, \bx)
    \]
     contains the simple braid group $\SBr_d(C) \le \Mod(C, \bx)$, and the same statement holds exchanging the roles of $C$ and $D$.
\end{lemma}
\begin{proof}
    See \cite[Proposition 7.1]{CImonodromy}. The result there is stated in the setting of $X$ a smooth complete intersection surface, but in the proof, only the hypothesis that $D$ extends to a maximally generic pencil $D_t$ rel $C$ is used (again $C$ and $D$ are swapped). The result in general therefore follows from \Cref{lemma:maxgenexists}.
\end{proof}

\begin{remark}
    It is much less clear what the intersection $\SBr(\tilde C) \cap \Gamma[\tilde C]$ of the simple braid group with the {\em principal} stabilizer should be. This determination, achieved in \Cref{lemma:SBrSES}, is really the key technical step in the entire argument.
\end{remark}

\subsection{Monodromy of plane curve singularities}\label{SS:planecurvesing}
The second class of monodromy subgroup we construct comes from importing the monodromy of some isolated plane curve singularities. We will need two such results: one for the $E_6$ singularity, with local expression $x^3+y^4=0$, and the other for the $A_7$ singularity, with local expression $y^2 + x^8$, or equivalently $y(y+x^4)$. If $C,D$ are smooth curves that intersect with multiplicity $4$ at some point $p$, then $C\cup D$ has an $A_7$ singularity at $p$.

\para{Versal deformation spaces, Morsifications, and configurations of vanishing cycles} We first briefly recall the notion of a {\em versal deformation space} and some related concepts, following \cite{AGZV} but eliding some details that will not play a major role. Let $f: \C^2 \to \C$ be the germ of a holomorphic function with an isolated singularity at $0$. Then there is a {\em versal deformation space} $V_f$ realized as an $\epsilon$-ball about $0$ in a certain vector space $\C^\mu$. The space $\C^\mu$ can be realized concretely as the quotient $\C[[x,y]]/J_f$ of the completed local ring at $0 \in \C^2$ by the {\em Jacobian ideal} $J_f = (\partial_x f, \partial_y f)$. Lifting from $\C^\mu \cong \C[[x,y]]/J_f$ to (some convergent representative in) $\C[[x,y]]$, we view elements of $\C^\mu$ as function germs: $g \in \C^\mu$ induces the perturbation $f+g$ of $f$. 

A {\em Morsification} of $f$ is a perturbation $f + g$ for some $g \in \C^\mu$ such that every critical point of $f+g$ near zero is nondegenerate and the associated critical values are distinct. There is an associated family over $\disk \subset \C$, with $z \in \disk$ corresponding to the germ of $f + g + \epsilon z$ (for some suitably small $\epsilon$). Since the critical points of $f+g$ are nondegenerate, the singular members of this family correspond to nodal degenerations of the plane curve germs. Classical results of singularity theory then assert that for the singularities $A_7$ and $E_6$ under study here, it is possible to choose a Morsification and a system of nodal degenerations for which the vanishing cycles intersect in the pattern of their namesake Dynkin diagram. We record this as follows; these can be easily proved using the theory of {\em divides} - see \cite{acampo} or \cite[Section 3]{nickpablo}. The rather bespoke-looking assertion (4) in \Cref{lemma:morsifyA7} will be essential in the arguments of the following section.

\begin{lemma}\label{lemma:morsifyA7}
    Let $f: \C^2 \to \C$ be given by the expression $f(x,y) = y(y+x^4)$, and let $V_{A_7} \subset \C^7$ be a versal deformation space for $f$. Then there is a disk $\disk \subset V_{A_7}$ parameterizing deformations $f+ g_t$ for $t \in \disk$, such that the following assertions hold:
    \begin{enumerate}
        \item There are seven distinct points $x_1, \dots, x_7 \in \disk$ such that the vanishing set $Z(f+g_t)$ has a nodal singularity near $0 \in \C^2$ for $t= x_1, \dots, x_7$, 
        \item Every other $Z(f+g_t)$ for $t \ne x_1,\dots, x_7$ is smooth near $0$; this is true in particular for $t = 1$,
        \item There are paths $\alpha_1, \dots, \alpha_7 \subset \disk$, with $\alpha_i$ connecting $1$ to $x_i$ and avoiding all $x_j$ in its interior, such that the associated vanishing cycles $a_1,\dots, a_7 \subset Z(f+g_1)$ form a simple configuration in the sense of \Cref{def:simpleconfig} whose intersection graph is the $A_7$ Dynkin diagram (i.e. a chain of length $7$),
        \item Let $C_{loc} = Z(y+x^4+\epsilon)$ and $D_{loc} = Z(y) \subset \C^2$ be affine plane curves (taking $\epsilon > 0$ sufficiently small), and let $E_{loc} = \tilde C_{loc} \cup \tilde D_{loc}$ be a smoothing as in \Cref{def:smoothing}. Then $a_1, a_3, a_5, a_7$ correspond to the boundary components $\Delta_1, \dots, \Delta_4$ separating $\tilde C_{loc}$ and $\tilde D_{loc}$, and $a_2, a_4, a_6$ each intersect $\tilde C_{loc}$ in a single arc connecting distinct boundary components. 
    \end{enumerate}
\end{lemma}

\begin{lemma}\label{lemma:morsifyE6}
    Let $f: \C^2 \to \C$ be given by the expression $f(x,y) = x^3+y^4$, and let $V_{E_6} \subset \C^6$ be a versal deformation space for $f$. Then there is a disk $\disk \subset V_{E_6}$ parameterizing deformations $f+g_t$ for $t \in \disk$, such that the following assertions hold:
    \begin{enumerate}
        \item There are six distinct points $x_1, \dots, x_6 \in \disk$ such that the vanishing set $Z(f+g_t)$ has a nodal singularity near $0 \in \C^2$ for $t= x_1, \dots, x_6$, 
        \item Every other $Z(f+g_t)$ for $t \ne x_1,\dots, x_6$ is smooth near $0$; this is true in particular for $t = 1$,
        \item There are paths $\beta_1, \dots, \beta_6 \subset \disk$, with $\beta_i$ connecting $1$ to $x_i$ and avoiding all $x_j$ in its interior, such that the associated vanishing cycles $b_1,\dots, b_6 \subset Z(f+g_1)$ form a simple configuration in the sense of \Cref{def:simpleconfig} whose intersection graph is the $E_6$ Dynkin diagram.
    \end{enumerate}
\end{lemma}

\para{From local to global} We now show how to embed these local results into the linear systems under study. We first formulate a general criterion (\Cref{lemma:embedE6}), then apply this to the $A_7$ and $E_6$ singularities in turn.

\begin{lemma}
    \label{lemma:embedE6}
    Let $L$ be a line bundle on $X$, and suppose that $L$ is $k$-jet ample for some $k \ge 1$. Let $\varphi(x,y) \in 
    \C[x,y]$ be a polynomial of degree $n\le k-2$ defining an isolated singularity at $0$, and suppose that there is a basis for $\C[[x,y]]/(\partial_x f, \partial_y f)$ lifting to polynomials $g_1, \dots, g_\mu \in \C[x,y]$ of degree $\le k$. 
    Let $V_{\varphi} \subset \C^\mu$ denote a versal deformation space of $\varphi$ and let $p \in X$ be given. Then there is a section $f_0 \in H^0(X;L)$ such that $Z(f_0)$ has the singularity type of $\varphi$ at $p$ and is smooth everywhere else, and there is an embedding
    \[
    i_p:V_{\varphi} \into H^0(X;L)
    \]
    with the following properties:
    \begin{enumerate}
        \item There is a small closed ball $\bar B \subset X$ centered at $p$ such that every $Z(f_0 +i_p(g))$ is smooth outside of $\bar B$, transverse to $\partial \bar B$, and the family of topological surfaces $Z(f_0 +i_p(g)) \setminus \bar B$ is trivial,
        \item $i_p$ is locally equisingular in the sense that the singularity type of $Z(f_0 + i_p(g)) \cap \bar B$ is determined by the corresponding germ $g \in V_{\varphi}$,
        \item If $D \subset X$ is any curve not passing through $p$, then $i_p$ can be constructed so that $Z(f_0+ i_p(g))$ is transverse to $D$ for every $g \in V_{\varphi}$, and there are disjoint balls $B_1, \dots, B_d \subset X$ centered at the points of $Z(f_0) \cap D$ such that 
        \[
        Z(f_0 + i_p(g)) \cap D \subset \bigcup_{i=1}^d B_i
        \]
        for all $g \in V_{\varphi}$.
    \end{enumerate}
\end{lemma}
\begin{proof}
Let $\hat\cO_p$ denote the completion of the local ring of functions at $p$ and let $x,y$ be local parameters. Fix a trivialization of $L$ in a neighborhood of $p$. With respect to this trivialization, we may think of sections of $L$ as functions in $\hat \cO_p$, and in particular compute their $k$-jets for all $k$. 

Since $L$ is $k$-jet ample and so {\em a fortiori} $n$-jet ample, there is $f_0 \in H^0(X;L)$ with $n$-jet at $p$ given by $\varphi(x,y)$. If $Z(f_0)$ has an isolated singularity at some other $q \in X$, then since $L$ is $n+2$-jet ample, there is $g \in H^0(X;L)$ that vanishes to order $n$ at $p$ and has any desired $1$-jet at $q$. Then for well-chosen such $g$ and $\epsilon > 0$ suitably small, $Z(f_0 + \epsilon g)$ is smooth near $q$, has no new singular points, and still has the singularity type of $\varphi$ at $p$. Repeating as necessary, we can assume that $f_0 \in H^0(X;L)$ has the singularity type of $\varphi$ at $p$ and is smooth everywhere else. By analogous arguments, if $Z(f_0)$ is not transverse to $D$ at some $q \in X$, there is a perturbation of $f_0$ to a nearby curve with the singularity type of $\varphi$ at $p$, smooth everywhere else, and transverse to $D$ at $q$. After a finite number of such perturbations, we can moreover assume that $Z(f_0)$ is transverse to $D$.

Let $J_{f_0} \subseteq \hat \cO_p$ be the ideal generated by $\partial_x f_0, \partial_y f_0$, so that $\hat \cO_p/J_{f_0} \cong \C^\mu$ and a versal deformation space $V_{\varphi}$ can be realized as a small ball about $0$ as in the above discussion. By hypothesis, the vector space $\cO_p/J_{f_0}$ is spanned by the images of polynomials $\bar g_1, \dots, \bar g_\mu$ of degree at most $k$. Since $L$ is $k$-jet ample, there are $g_1, \dots, g_\mu \in H^0(X; L)$ whose images form a basis for $\hat \cO_p/J_{f_0}$. Letting $V \le H^0(X;L)$ be the span of $g_1, \dots, g_\mu$, one can identify the versal deformation space $V_{\varphi}$ with a sufficiently small $\epsilon$-ball in $V$; this identification is clearly equisingular. The smoothness and transversality properties in (1) and (3) are all open conditions satisfied by $f_0$, and so hold for all $f_0 + i_p(g)$ so long as $\epsilon$ is chosen to be sufficiently small. The family of topological surfaces $Z(f_0+i_p(g)) \setminus \bar B$ is defined over the contractible set $V_{\varphi}$, and hence is trivial.
\end{proof}

\begin{lemma}
    \label{lemma:A7VCs}
    Let $L = L_1 \otimes L_2$ be given, with $L_1$ $6$-jet ample and $L_2$ very ample. Let $E = \tilde C \cup \tilde D \in \tilde U_{L_1}(D)$ be a basepoint, obtained by smoothing $C \in U_{L_1}$ and $D \in U_{L_2}$ intersecting transversely. Then there are loops $\alpha_1,\dots, \alpha_7 \in \pi_1(U_{L_1 \otimes L_2},E)$ for which the monodromy of each $\rho_L(\alpha_i) = T_{a_i}$ is the Dehn twist about a vanishing cycle $a_i \subset E$, such that $a_1, \dots, a_7$ form the $A_7$ configuration, $a_1, a_3, a_5, a_7$ form distinct boundary components of $\tilde C$, and $a_2, a_4, a_6$ each intersect $\tilde C$ in a single arc connecting distinct boundary components.
\end{lemma}
\begin{proof}
    Choose a point $p \in D \subset X$. Since $L_1$ is $6$-jet ample and hence $4$-jet ample, there is a section $f_0 \in H^0(X;L_1)$ such that $Z(f_0)$ is smooth at $p$ and intersects $D$ there with multiplicity $4$. As in the proof of \Cref{lemma:embedE6}, the $6$-jet ampleness of $L_1$ implies the existence of such an $f_0$ which is moreover smooth and transverse to $D$ away from $p$. 

    Since $L_1$ is $6$-jet ample and $L_2$ is very ample, i.e. $1$-jet ample, $L_1 \otimes L_2$ is $7$-jet ample \cite[Lemma 2.2]{BSjetample}. Thus $Z(f_0) \cup D$ can be perturbed so as to resolve each transverse intersection of $D$ with $Z(f_0)$, preserving the local structure of an order-4 tangency at $p$ (locally given as $y(y+x^4)$). Let the associated section of $L_1 \otimes L_2$ be denoted by $F_0$:
    \begin{equation}
        \label{eqn:F0}
        F_0 := f_0 g_D + h_0,
    \end{equation}
    where $g_D \in H^0(X;L_2)$ defines $D$. 

    Setting $\varphi(x,y) = y^2 + yx^4$, there is a basis for $\C[[x,y]]/(\partial_x \varphi, \partial_y \varphi)$ given by $1,x,x^2,x^3,y,xy,x^2y$, polynomials of degree at most $3$. Thus \Cref{lemma:embedE6} applies for the $7$-jet ample line bundle $L = L_1 \otimes L_2$, and moreover one can take $f_0$ as in the statement of \Cref{lemma:embedE6} to be $F_0$ as in \eqref{eqn:F0}. Let $i_p: V_{A_7} \into H^0(X;L_1 \otimes L_2)$ be the embedding of the versal deformation space, and let $\disk \subset V_{A_7}$ be the disk given by \Cref{lemma:morsifyA7}. We can assume that $1 \in \disk$ corresponds to a point $E' \in \tilde U_{L_1}(D)$ under $i_p$ given by smoothing $C' \cup D$ to $E' = \tilde C' \cup \tilde D$, for some $C' \in U_{L_1}$.

    Let $\bar B \subset X$ be a small ball centered at $p$ as in \Cref{lemma:embedE6}. The based loops $\alpha'_1, \dots, \alpha'_7 \subset \disk$ of \Cref{lemma:morsifyA7} then induce nodal degenerations of $E'$ for which the vanishing cycles $a'_1,\dots, a'_7$ are contained in the subsurface $E' \cap \bar B$. By \Cref{lemma:morsifyA7}, the local model for $E' \cap \bar B$ is given by smoothing the transversely-intersecting curves $C_{loc}$ and $D_{loc}$. Thus the topological properties of $a'_1,\dots, a'_7 \subset E'$ claimed in the statement of the lemma are satisfied, since they hold on the local model by \Cref{lemma:morsifyA7}.

    It remains only to consider the effect of changing basepoint, from $E' \in \tilde U_{L_1}(D)$ near the $A_7$ singularity, back to the original (and arbitrary) $E \in \tilde U_{L_1}(D)$. For this, we recall the results of \Cref{SS:basepoints}. Let $\gamma \subset \tilde U_{L_1}(D)$ be a path connecting $E'$ to $E$, and for $i = 1, \dots, 7$, define $a_i := \tau_\gamma(a_i')$. By \Cref{lemma:ptsprt}, $\tau_\gamma$ restricts to an identification $\tau_\gamma: \tilde C' \to \tilde C$. In particular, this shows that $a_1, a_3, a_5, a_7$ are boundary components of $\tilde C$, since $a_1', a_3', a_5', a_7'$ are boundary components of $\tilde C'$. Since $a'_2, a'_4, a'_6$ each intersect $\tilde C'$ in a single arc, it follows that $a_2, a_4, a_6$ each intersect $\tilde C$ in a single arc, and since $a'_1, \dots, a'_7$ form the $A_7$ configuration, the same is true of $a_1, \dots, a_7$.
\end{proof}

\begin{lemma}
    \label{lemma:E6VCs}
    Let $L = L_1 \otimes L_2$ be given, with $L_1$ $6$-jet ample and $L_2$ very ample. Let $E \cong \tilde C \cup \tilde D \in \tilde U_{L_1}(D)$ be a basepoint. Then there are loops $\tilde \beta_1,\dots, \tilde \beta_6 \in \pi_1(\tilde U_{L_1}(D),E)$ for which the monodromy of each $\rho_L(\tilde \beta_i) = T_{b_i}$ is the Dehn twist about a vanishing cycle $b_i \subset \tilde C$, such that $b_1, \dots, b_6$ form the $E_6$ configuration.
\end{lemma}
\begin{proof}
 For the polynomial $\varphi(x,y) = x^3 + y^4$ defining the $E_6$ singularity, the vector space $\C[[x,y]]/(\partial_x \varphi, \partial_y \varphi)$ is spanned by monomials of the form $x^ay^b$, with $0 \le a \le 1$ and $0 \le b \le 2$. As the maximal degree of such a monomial is $3$, and $\varphi$ itself has degree $4$, \Cref{lemma:embedE6} is applicable to the $6$-jet ample line bundle $L_1$. Accordingly, choose $p \in X$ not lying on $D$, and let $i_p: V_{E_6} \into H^0(X; L_1)$ be the embedding given by \Cref{lemma:embedE6}, so that every curve in the image is transverse to $D$. Let $C_1$ be a smooth curve defined as $C_1 = Z(f_1)$ with $f_1 = f_0 + i_p(g_1)$ for $f_0$ as in \Cref{lemma:embedE6}, and take $E = \tilde C_1 \cup \tilde D$ to be a smoothing of $C_1 \cup D$ defined by the equation $f_1g_D + h_1$, where  $D \in U_{L_2}$ is defined as $D = Z(g_D)$ for $g_D \in H^0(X;L_2)$ and $h_1 \in H^0(X;L_1\otimes L_2)$ is a suitable smoothing perturbation as in \Cref{def:smoothing}.

Let $\disk \subset V_{E_6}$ be the disk of \Cref{lemma:morsifyE6}, with six points $x_1, \dots, x_6 \in \disk$ corresponding to nodal deformations and the associated loops $\beta_1, \dots, \beta_6 \in \pi_1(\disk \setminus \{x_1,\dots,x_6\},1)$; we can assume that $i_p(1) = g_1$ such that $Z(f_0+i_p(g_1))$ is our chosen $C_1$. We construct based loops $\tilde\beta_1, \dots, \tilde\beta_6$ in $(\tilde U_{L_1}(D),E)$ as follows: first lift the loop $\beta_i(t)$ to the loop in $H^0(X;L_1 \otimes L_2)$ given by $(f_0+i_p(\beta(t))g_D$ (as a family of curves, this is just the union of $D$ with the curves in $U_{L_1}$ parameterized by $\beta_i(t)$). Now perturb this to a loop $\tilde \beta_i(t) \subset \tilde U_{L_1}(D)$ defined by the formula
\[
\tilde \beta(t) = Z((f_0 + i_p(\beta(t)))g_D + h_t),
\]
where $h_t \in H^0(X;L_1 \otimes L_2)$ defines a smoothing of $Z(f_0+i_p(\beta(t))) \cup D$ and $h_0=h_1$ is as above. Such a family $h_t$ exists, since for any fixed $t$, the set of disallowed smoothing parameters (those for which $Z((f_0 + i_p(\beta(t)))g_D + h_t)$ fails to be smooth) forms an analytic hypersurface in $H^0(X; L_1 \otimes L_2)$ which can be avoided by an arbitrarily small perturbation.

Since $\tilde \beta_i \subset \tilde U_{L_1}(D)$, the monodromy image $\rho_D(\tilde \beta_i)$ is contained in $\Gamma_0[\tilde C_1]$ by \Cref{lemma:localmonodromy}, but in fact $\rho_D(\tilde \beta_i) \le \Gamma[\tilde C_1]$ is contained in the {\em principal} stabilizer. To see this, note that by property (3) of \Cref{lemma:embedE6}, the points of intersection $y_1(g), \dots, y_d(g)$ of $Z(f_0+i_p(g)) \cap D$ are contained in fixed disjoint balls. After smoothing, the monodromy induced on $\tilde D$ is given by pushing the boundary components $\Delta_1, \dots, \Delta_d$ along the braid induced by the points of intersection; since this braid is trivial here, the monodromy on $\tilde D$ is trivial, i.e. $\rho_D(\tilde \beta_i) \le \Gamma[\tilde C_1]$, as claimed.

Recall from \Cref{SS:stabilizer} that the restriction $\res: \Gamma[\tilde C_1] \to \Mod(\tilde C_1)$ is injective, and that the kernel of $\bar \res: \Gamma[\tilde C_1] \to \Mod(C_1, \bx)$ is generated by the twists $T_{\Delta_1}, \dots, T_{\Delta_d}$. Since each $T_{\Delta_i} \in \Gamma[\tilde C_1]$ by \Cref{lemma:TDelta}, it suffices to consider the image $\bar \res \circ \rho_D(\tilde \beta_i) \in \Mod(C_1,\bx)$. By properties (1) and (3) of \Cref{lemma:embedE6}, these mapping classes are supported on the subsurface $S \subset C_1$ of genus $3$ containing the vanishing cycles $b_1, \dots, b_6$. Thus, possibly after altering $\tilde \beta_i$ by some loop in some $\tilde U_{L_1 \otimes L_2}(C_1,D)$ to remove unwanted boundary twists $T_{\Delta_i}$, the monodromy of $\tilde \beta_i$ is the Dehn twist $T_{b_i}$ about some vanishing cycle $b_i \in \tilde C_1$, and $b_1, \dots, b_6$ form the $E_6$ configuration.
\end{proof}

As a corollary of these results, we can construct vanishing cycles on $E = \tilde C \cup \tilde D$ which intersect each half of $E$ in a single arc, and which have any desired intersection with one half, say $\tilde D$.

\begin{lemma}\label{corollary:Dsidecontrol}
    Let $E = \tilde C \cup \tilde D$  be a basepoint as in \Cref{conv}, and let $\alpha \subset \tilde D$ be any simple arc connecting distinct boundary components of $\tilde D$. Then there is a vanishing cycle $a \subset E$ such that $a \cap \tilde D = \alpha$. 
\end{lemma}
\begin{proof}
    Let $a_1, \dots, a_7$ be the vanishing cycles provided by \Cref{lemma:A7VCs}; in particular, $a_2$ intersects $\tilde D$ in a single arc, say $\alpha_2 \subset \tilde D$. By \Cref{lemma:sbrarctrans}, there is $\beta \in \SBr(\tilde D)$ such that $\beta(\alpha_2) = \alpha$. By \Cref{lemma:simpleinmono}, there is $\tilde \beta \in \Gamma_0[\tilde D]$ such that $\bar \res (\tilde \beta) = \beta$, and by \Cref{lemma:againVC}, it follows that $a:= \tilde \beta(a_2)$ is a vanishing cycle with $a \cap \tilde D = \alpha$ as required.
\end{proof}

\section{Building a core}\label{S:core}

In this section, we leverage the work of the previous section to construct a ``core'' subsurface of the fiber on which the monodromy group is quite large. This is necessary to invoke the theory of assemblage generating sets for $r$-spin mapping class groups in the proof of \Cref{theorem:main}, and some of the specific results here (in the guise of \Cref{corollary:allBs}) will be used in the proof of the Main Lemma in the next section.

        \begin{figure}[h]
\centering
		\labellist
        \pinlabel $\tilde{C}$ at 100 10
        \pinlabel $\tilde{D}$ at 300 10
        \small
        \pinlabel $S_0$ at 160 140
        \tiny
        \pinlabel $a_1$ at 232 196
        \pinlabel $a_2$ at 234 180
        \pinlabel $a_3$ at 232 167
        \pinlabel $a_4$ at 234 154
        \pinlabel $a_5$ at 230 137
        \pinlabel $a_6$ at 235 122
        \pinlabel $a_7$ at 230 105
        \pinlabel $b_6$ at 200 145
        \pinlabel $b_1$ at 85 185
		\endlabellist
\includegraphics{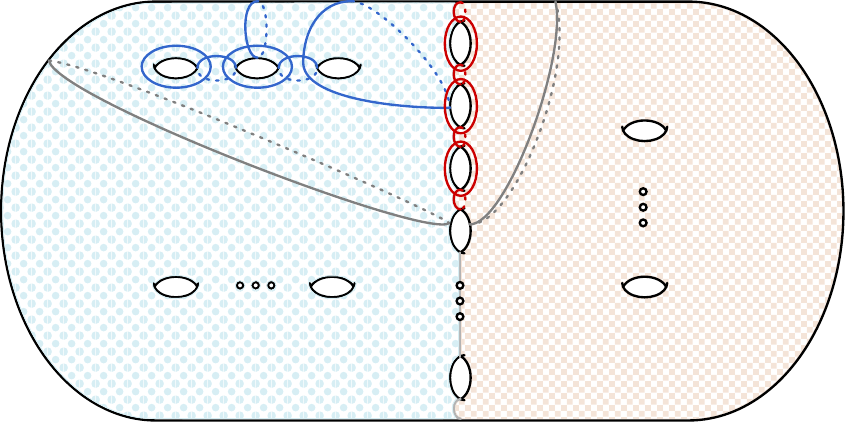}
\caption{The subsurface $S_0 \subset E$ and the vanishing cycles $a_1, \dots, a_7, b_1, \dots, b_6$.}
\label{fig:core}
\end{figure}
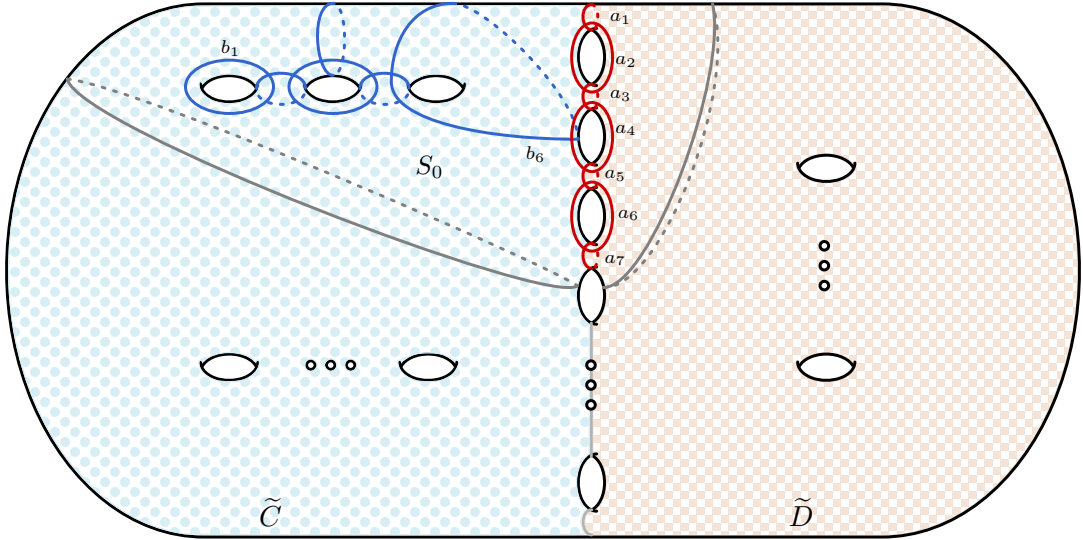

\begin{lemma}\label{lemma:core}
   There exists a subsurface $S_0 \subset E$ homeomorphic to $\Sigma_6^2$, and vanishing cycles $a_1, \dots, a_7, b_1, \dots, b_6$ as shown in \Cref{fig:core}, such that 
    \[
    \Mod(S_0)[\phi] \le \Gamma,
    \]
    where $\phi$ is the framing of $S_0$ determined by the condition that all $a_i, b_i$ be admissible. 
\end{lemma}
\begin{proof}
    To show the containment $\Mod(S_0)[\phi] \le \Gamma$, we will appeal to the framed mapping class group generation criterion, \Cref{theorem:assemblages}. To do this, we will construct vanishing cycles $a_1,\dots,a_7,b_1,\dots, b_6$ in the configuration shown in \Cref{fig:core}, which visibly form a $6$-assemblage of type $E$. This will be accomplished by combining the vanishing cycles in singularities of type $A_7$ and $E_6$, using the action of the simple braid group to control intersections between the two sets of vanishing cycles.

    Take $E = \tilde C \cup \tilde D$ as a basepoint, as in \Cref{conv}.
    By \Cref{lemma:A7VCs}, there are nodal degenerations based at $E$ with vanishing cycles $a_1, \dots, a_7$ in the $A_7$ configuration, with $a_1, a_3, a_5, a_7$ forming distinct boundary components of $\tilde C$ and $a_2, a_4, a_6$ each intersecting $\tilde C$ in a single arc connecting distinct boundary components.
    Separately, by \Cref{lemma:E6VCs}, there are nodal degenerations based at some $E' = \tilde C' \cup \tilde D$ with vanishing cycles $b_1', \dots, b_6' \subset \tilde C'$ forming the $E_6$ configuration. Let $\gamma \subset \tilde U_{L_1}(D)$ be a path connecting $E'$ and $E$. By \Cref{lemma:ptsprt}, under parallel transport along $\gamma$, the vanishing cycles $b_1', \dots, b_6'$ are transported to vanishing cycles $b''_1, \dots, b''_6 \subset \tilde C \subset E$. 

    Let $B \subset \tilde C$ be the five-holed boundary-adjacent sphere with boundary components $a_1, a_3, a_5, a_7$ and that contains the intersections of $a_1, a_3, a_5$ with $\tilde C$. Let $B' \subset \tilde C$ be a five-holed boundary-adjacent sphere with boundary components $a_1, a_3, a_5, a_7$ disjoint from $b''_1, \dots, b''_5$, and such that $b''_6 \cap B'$ is a single arc separating the four remaining boundary components into two sets of two.
    
    By \Cref{lemma:sbdisktrans}, there is $\beta \in \SBr(\tilde C)$ such that $\beta B' = B$, and by \Cref{lemma:simpleinmono}, this lifts to $\tilde \beta \in \Gamma_0[\tilde C]$ such that $\bar \res(\tilde \beta) = \beta$. Define
    \[
    b_1 := \tilde \beta(b''_1), \dots, b_5 := \tilde \beta (b''_5).
    \]
    By \Cref{lemma:againVC}, each $b_i$ is a vanishing cycle, and since $B'$ is disjoint from $b''_1,\dots, b''_5$, it follows that $B$ is disjoint from $b_1, \dots, b_5$; in particular, the vanishing cycles $a_1, \dots, a_7$ are disjoint from $b_1, \dots, b_5$.

    By this same logic, $\beta(b''_6)$ is a vanishing cycle that intersects $B$ in a single arc separating $a_1, a_3, a_5, a_7$ into two sets of two. Since the four-stranded braid group acts transitively on such arcs, the same is true of the simple braid group $\SBr(\tilde C)$, and so there is $\beta' \in \SBr(\tilde C)$ supported on $B$ such that 
    \[
    i(\beta' \beta(b''_6), a_2) = i(\beta' \beta(b''_6), a_6) = 0 \quad \mbox{and} \quad i(\beta' \beta(b''_6), a_4) = 1.
    \]
    Define
    \[
    b_6:= \beta' \beta(b''_6),
    \]
    and note that arguing as before, $b_6$ is a vanishing cycle. Then the vanishing cycles  $a_1, \dots, a_7, b_1, \dots, b_6$ form a $6$-assemblage of type $E$, and so by \Cref{theorem:assemblages}, the associated Dehn twists generate a framed mapping class group $\Mod(S_0)[\phi]$. Since each $c_i$ is a vanishing cycle, each associated $T_{c_i} \in \Gamma$, as was to be shown.
\end{proof}

        \begin{figure}[h]
\centering
		\labellist
        \small
        \pinlabel $B'$ at 195 120
        \pinlabel $S_0$ at 130 90
        \pinlabel $x$ at 372 100
        \pinlabel $y$ at 425 50
        \pinlabel $z$ at 372 50
        \pinlabel $w$ at 425 100
		\endlabellist
\includegraphics[width=\textwidth]{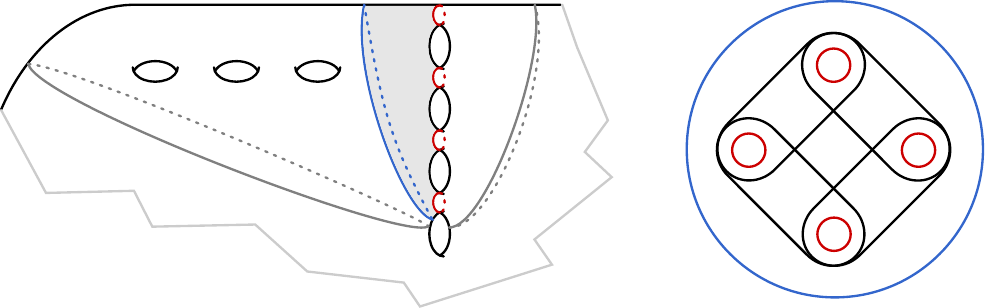}
\caption{At left, the five-holed boundary-adjacent sphere $B'$, as embedded in $S_0\subset E$, with $S_0 \subset E$ a subsurface as in \Cref{fig:core}. At right, an abstract view of a five-holed sphere depicting the curves $x,y,z,w$.}
\label{fig:5hole}
\end{figure}
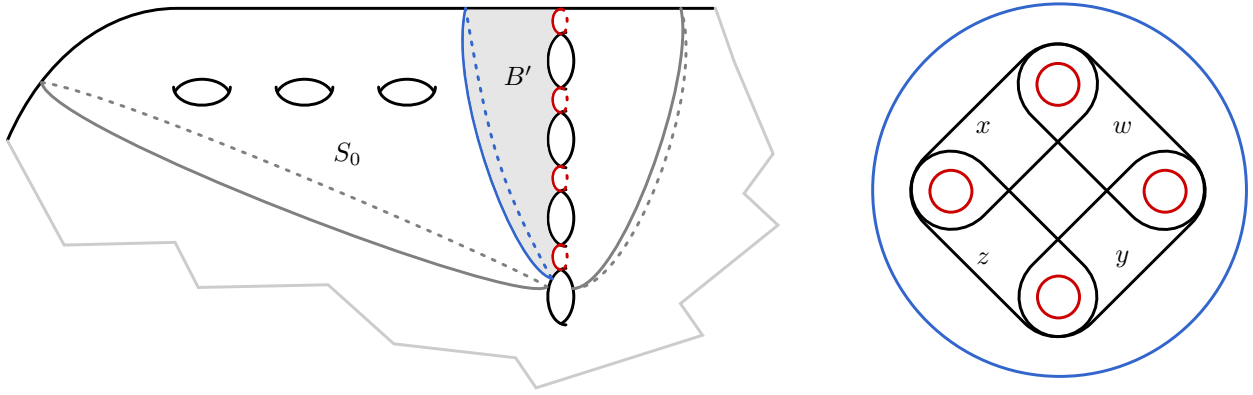

We will make use of the following result in the next section.

\begin{corollary}\label{corollary:allBs}
   Let $B \subset \tilde C$ be a five-holed boundary-adjacent sphere. Then
    \[
    [\Mod(B),\Mod(B)] \le \Gamma[\tilde C],
    \]
    and also
    \[
    T_x T_y T_z^{-1} T_w^{-1} \in \Gamma[\tilde C],
    \]
    where $x,y,z,w$ are as shown in \Cref{fig:5hole}.
\end{corollary}
\begin{proof}

    Let $S_0 \subset E$ be a subsurface as in \Cref{lemma:core}, so that $\Mod(S_0)[\phi] \le \Gamma$. Let $B' \subset S_0 \cap \tilde C$ be a five-holed boundary-adjacent sphere. We claim that $[\Mod(B'), \Mod(B')]$ and $T_x T_y T_z^{-1}T_w^{-1}$ as shown in \Cref{fig:5hole} are contained in $\Gamma[\tilde C]$. Each of these clearly is contained in $\Mod(\tilde C)$, so it remains to show they are contained in $\Gamma$. By \Cref{lemma:core}, it suffices to show that they preserve the framing $\phi$ on the subsurface $S_0$. 

    This will be a consequence of the basic properties of winding number functions. The group $\Mod(B')$ is generated by Dehn twists $T_c$, for boundary-adjacent curves $c$ enclosing some pair $\Delta_i, \Delta_j$ of boundary components of $B'$ (such $\Delta_i, \Delta_j$ are also boundary components of $\tilde C$). At the level of homology, when suitably oriented, $[c] = [\Delta_i]+ [\Delta_j]$, and by homological coherence (\Cref{lemma:WNFprops}.2), $\phi(c) = \phi(\Delta_i) + \phi(\Delta_j) + 1$. 
    
    The commutator subgroup $[\Mod(B'), \Mod(B')]$ is generated by elements of the form $[T_{c'}, T_{c}]$, or equivalently by elements $T_{T_{c'}(c)} T_c^{-1}$. But the curve $T_{c'}(c)$ {\em also} encloses the same $\Delta_i, \Delta_j$ as $c$. In particular, both its homology class and its winding number equal those of $c$. We conclude that $T_{T_{c'}(c)} T_c^{-1} \in \Mod(S_0)[\phi]$, as was to be shown. Similarly, with reference to \Cref{fig:5hole}, one sees that each of the pairs $x,y$ and $z,w$ together enclose the same four boundary components, and so by the above discussion have the same effect on winding numbers; we conclude $T_xT_y T_z^{-1} T_w^{-1} \in \Mod(S_0)[\phi]$ as well.

    Now let $B \subset \tilde C$ be any other five-holed boundary-adjacent sphere. By \Cref{lemma:sbdisktrans}, there is $\beta \in \SBr(\tilde C)$ such that $\beta B' = B$. By \Cref{lemma:simpleinmono}, there is $\tilde \beta \in \Gamma_0[\tilde C]$ such that $\bar \res (\tilde \beta) = \beta$, and by \Cref{lemma:principalnormal}, $\Gamma[\tilde C]$ is normalized by $\Gamma_0[\tilde C]$. We conclude that 
    \[
    [\Mod(B),\Mod(B)] = \tilde \beta [\Mod(B'), \Mod(B')] \tilde \beta^{-1} \le \Gamma[\tilde C],
    \]
    and a similar argument addresses the second kind of elements.
\end{proof}

\section{The main lemma}\label{S:mainlemma}

The goal of this section is to establish \Cref{lemma:main}. This will allow us to construct vanishing cycles with maximal control over the intersection with both halves $\tilde C, \tilde D$ of $E$, so long as this intersection is a single arc on each half. Once this has been established, it will quickly lead to the proof of \Cref{theorem:main} in \Cref{S:mainproof}.

\subsection{Statement of main results}
Here we state \Cref{lemma:main} and the key technical result \Cref{lemma:SBrSES} which implies it. We then give a proof of \Cref{lemma:main} contingent on \Cref{lemma:SBrSES}.

\begin{lemma}\label{lemma:main}
    Let $X$ be a smooth algebraic surface with $\pi_1(X) = 1$, let $L_1, L_2$ be line bundles on $X$ such that $L_1$ is $6$-jet ample and $L_2$ is very ample, let $D \in U_{L_2}$ be chosen, and let $E = \tilde C \cup \tilde D \in \tilde U_{L_1}(D) \subset U_{L_1 \otimes L_2}$ be a basepoint as in \Cref{conv}.
    
    Let $a \subset E$ be a simple closed curve such that $a \cap \tilde C$ and $a \cap \tilde D$ are both arcs with endpoints on boundary components $\Delta_i \ne \Delta_j$. Let $c \subset \tilde C$ be a boundary-adjacent curve (\Cref{def:boundaryadj}) enclosing $\Delta_i$ and some third boundary component $\Delta_k \ne \Delta_i, \Delta_j$, and such that $i(a,c) = 1$. Then there is $\ell \in \Z$ such that $T_c^\ell(a)$ is a vanishing cycle.
\end{lemma}

        \begin{figure}[h]
\centering
		\labellist
        \pinlabel $\tilde{C}$ at 100 20
        \pinlabel $\tilde{D}$ at 300 20
        \small
        \pinlabel $c$ at 190 180
        \pinlabel $a$ at 260 170
		\endlabellist
\includegraphics[scale=0.66]{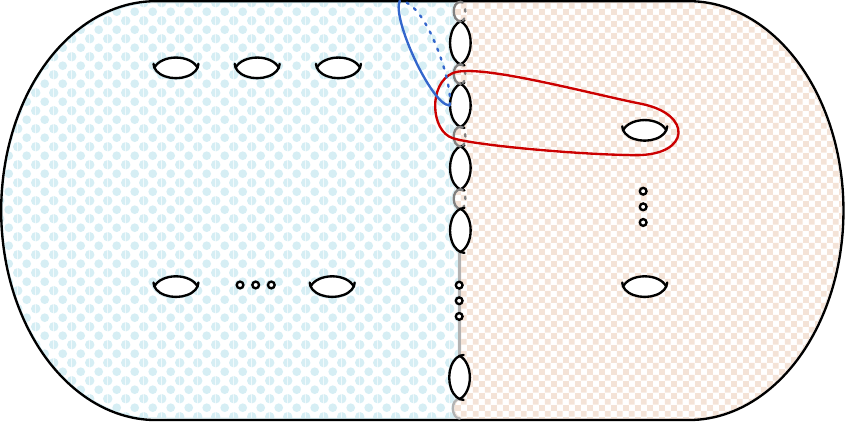}
\caption{The configuration of \Cref{lemma:main}.}
\end{figure}
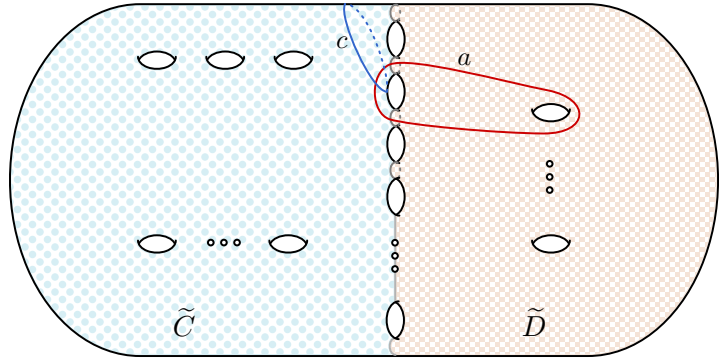

\Cref{lemma:main} is a relatively straightforward consequence of the following result.

\begin{lemma}
    \label{lemma:SBrSES}
    Let $d \ge 3$, let $\Z\pair{e_1, \dots, e_d}$ be the free abelian group of rank $d$, and let $A$ be the subgroup of index two spanned by elements of the form $e_i + e_j$. Then there is a homomorphism $\psi: \PSBr(\tilde C) \to A$ satisfying $\psi(T_{c_{ij}}) = e_i + e_j$, where $c_{ij}$ is any boundary-adjacent curve enclosing $\Delta_i$ and $\Delta_j$, and the following sequence is exact:
    \[
    1 \to \PSBr(\tilde C) \cap \Gamma[\tilde C] \to \PSBr(\tilde C) \xrightarrow{\psi} A \to 1.
    \]
\end{lemma}

\begin{proof}[Proof of \Cref{lemma:main}, assuming \Cref{lemma:SBrSES}]
    We first observe that since $L_1$ is $6$-jet ample, 
    \[
    d := C \cdot D \ge 6.
    \]
    By \Cref{corollary:Dsidecontrol}, there is a vanishing cycle $a'$ with $a' \cap \tilde D = a \cap \tilde D$. By \Cref{lemma:sbrarctrans}, there is $\beta_1 \in \PSBr(\tilde C) \le \Mod(\tilde C)$ such that $\beta_1 (a' \cap \tilde C) = a \cap \tilde C$, and hence $\beta_1 a' = a$. Set 
    \[
    \psi(\beta_1) = \sum_{i = 1}^d k_i e_i.
    \]
    For notational simplicity, assume without loss of generality that $a \cap \tilde C$ connects the boundary components $\Delta_1$ and $\Delta_2$, and that $c$ encloses $\Delta_2$ and $\Delta_3$. For $4 \le i \le d$, let $c_i \subset \tilde C$ be a curve enclosing $\Delta_3$ and $\Delta_i$ that is disjoint from $a$. Define
    \[
    \beta_2 = T_{c_4}^{-k_4} \dots T_{c_d}^{-k_d} \beta_1,
    \]
    and note that
    \[
    \psi(\beta_2) = k_1' e_1 + k_2'e_2 + k_3'e_3
    \]
    for some integers $k_i'$ such that $k_1' + k_2' + k_3' \equiv 0 \pmod 2$, and that
    \[
    \beta_2 a' = \beta_1 a' = a.
    \]
    Let $c_{45} \subset \tilde C$ be a boundary-adjacent curve enclosing $\Delta_4$ and $\Delta_5$ and disjoint from $a$. Next, define
    \[
    \beta_3 = (T_{c_4}T_{c_5}T_{c_{45}}^{-1})^{(k_2'- k_1'-k_3')/2}\beta_2,
    \]
    (the exponent is an integer) and note that
    \[
    \psi(\beta_3) = k_1' e_1 + k_2' e_2 + (k_2'-k_1') e_3,
    \]
    and that $\beta_3 a' = a$. Finally, let $c_{12} \subset \tilde C$ be the unique boundary-adjacent curve enclosing $\Delta_1$ and $\Delta_2$ and disjoint from $a$, and recall that $c$ is boundary adjacent, enclosing $\Delta_2$ and $\Delta_3$. Then define
    \[
    \beta = T_c^{k_1'-k_2'} T_{c_{12}}^{-k_1'}\beta_3.
    \]
    Since $\beta_3 a' = a$, it follows that $\beta a' = T_c^{k_1'-k_2'}a$.
    By construction, $\psi(\beta) = 0$, so that by \Cref{lemma:SBrSES}, $\beta \in \Gamma[\tilde C]$, and so $\beta a' = T_c^{\ell} a$ is a vanishing cycle for $\ell= k_1'-k_2'$, as claimed.
\end{proof}

\subsection{First steps: organizational plan, and an exact sequence}
The proof of \Cref{lemma:SBrSES} relies on a number of intermediate results. These are organized into the remaining five subsections. The major stepping-stone to \Cref{lemma:SBrSES} is \Cref{lemma:insteadK}, or perhaps the commutative diagram appearing in the proof thereof. This organizes the relationships between the various groups being considered here. With that in mind, the objective of the next two subsections is to show that each of the columns that will appear in this diagram is exact. The first of these is relatively straightforward.

Recalling the discussion of \Cref{S:braids}, define $\tilde \iota_*: \PBr(\tilde C) \to \pi_1(C^d)$ as the composition
\[
\PBr(\tilde C) = \tilde{\PBr_d}(C) \to \PBr_d(C) \xrightarrow{\iota_*} \pi_1(C^d),
\]
and then define the subgroups
\begin{equation}\label{eq:K}
    K:= \ker \tilde \iota_*
\end{equation}
and
\[
\pi_1(C^d)_0 := \ker \left( \pi_1(C^d) \to H_1(C;\Z)^{\oplus d} \xrightarrow{\sigma} H_1(C;\Z)\right).
\]

 \begin{lemma}\label{lemma:middlecol}
            There is an exact sequence $1 \to K \to \PSBr(\tilde C) \xrightarrow{\tilde \iota_*} \pi_1(C^d)_0 \to 1$.
             \end{lemma}
             
      \begin{proof} 
      Since $\PConf_d(C)$ is obtained from $C^d$ by deletion of the fat diagonal hypersurface, it follows (e.g. by the Seifert-van Kampen theorem) that the induced map on fundamental group $\iota_*: \PBr_d(C) \to \pi_1(C^d)$ is surjective. As $\lambda: \PBr_d(C) \to H_1(C;\Z)$ factors through the maps $\pi_1(C^d) \to H_1(C;\Z)^{\oplus d} \xrightarrow{\sigma} H_1(C;\Z)$ defining $\pi_1(C^d)_0$, it follows that $\ker(\lambda) = \PSBr_d(C)$ surjects onto $\pi_1(C^d)_0$, and {\em a fortiori} also $\PSBr(\tilde C)$ surjects onto $\pi_1(C^d)_0$ as claimed. It remains only to see that $K \le \PSBr(\tilde C)$. This holds because $\lambda: \PBr(\tilde C) \to H_1(C;\Z)$ factors through $\iota_*: \PBr_d(C) \to \pi_1(C^d)$.
      \end{proof}

\subsection{Exact sequence two: braids in monodromy}
Our next major objective is \Cref{lemma:leftcol}, which shows that there is a similar surjection even when restricting to the intersection $\PSBr(\tilde C) \cap \Gamma[\tilde C]$. The key construction of suitable monodromy elements is given in \Cref{lemma:pi1cd0img}, and then \Cref{lemma:VCsgenerate} relates the simple-connectivity hypothesis $\pi_1(X) = 1$ to the structure of the set of vanishing cycles as elements of the fundamental group of the fiber.

To proceed, we must make a few comments about basepoints. 
               Let $\bx:= (x_1, \dots, x_d) \in \PConf_d(C) \subset C^d$ be a basepoint. Then there is an isomorphism
               \[
               \pi_1(C^d,\bx) \cong \prod_{i = 1}^d \pi_1(C, x_i). 
               \]
               Via this isomorphism, we view a pure braid $\beta \in \pi_1(\PConf_d(C), \bx)$ as giving rise to a $d$-tuple of elements $(\beta_1, \dots, \beta_d) \in \prod_{i=1}^d \pi_1(C, x_i)$ under $\iota_*$; the key point is that each $\beta_i$ is based at the corresponding basepoint $x_i$.

        \begin{figure}[h]
\centering
		\labellist
        \small
        \pinlabel $\tau$ at 142 300
        \pinlabel $\Delta_i$ at 172 300
        \pinlabel $\Delta_j$ at 172 240
        \pinlabel $\tau'$ at 142 246
        \pinlabel $a$ at 110 305
        \pinlabel $T_a(\sigma)$ at 342 250
        \pinlabel $\ul{a}_\tau$ at 346 125
        \pinlabel $\ul{a}_{\tau'}$ at 293 135
		\endlabellist
\includegraphics[width=\textwidth]{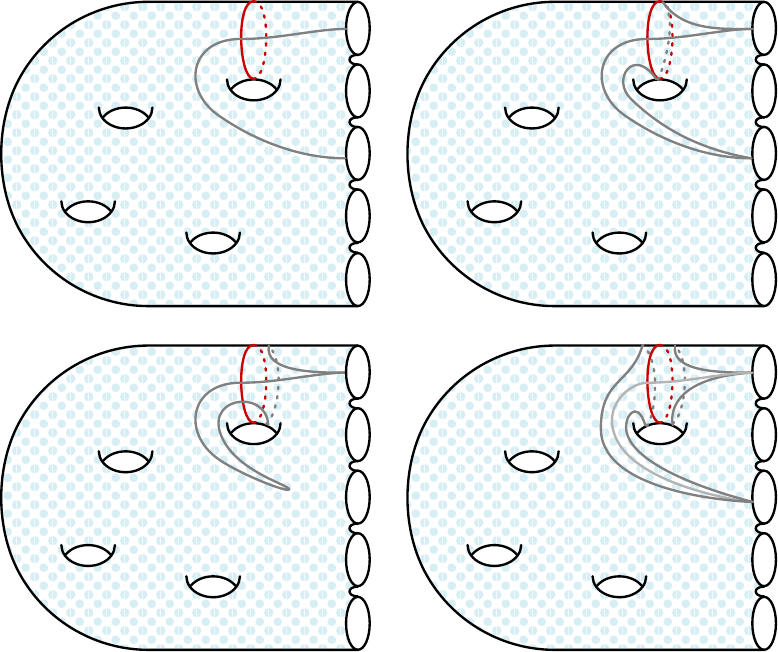}
\caption{Constructing a pair of point-push maps by commuting a simple half-twist and the Dehn twist about a vanishing cycle, as in \Cref{lemma:pi1cd0img}. The bottom-left panel shows how $[T_a,\sigma]$ is isotopic to a pair of point-push maps based at $\Delta_i$ and $\Delta_j$; for clarity, only one of the two ($\ul{a}_\tau$) is shown.}
\label{fig:tethertwist}
\end{figure}
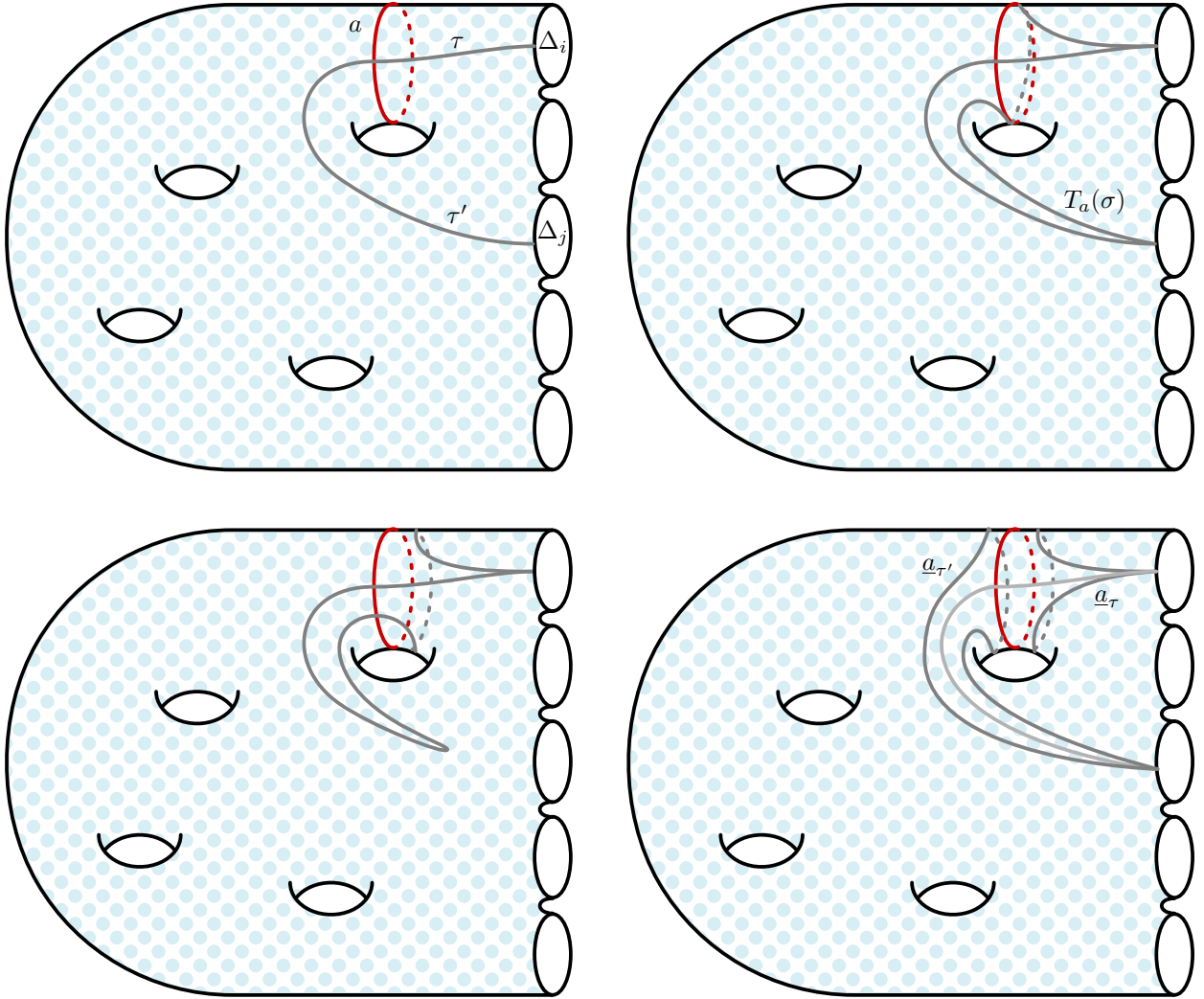

\begin{lemma}
    \label{lemma:pi1cd0img}
  For any pair of indices $i \ne j$ and any based loop $\ul{a} \in \pi_1(C,x_i)$ in the conjugacy class of some vanishing cycle $a \subset C$, there is a based loop $\ul{a}' \in \pi_1(C,x_j)$ also in the conjugacy class of $a$, such that the element 
                \[
                (1, \dots, \ul{a}, \dots, \ul{a}'^{-1}, \dots, 1) \in \prod_{i = 1}^d \pi_1(C, x_i)
                \]
                (with all components other than $i$ and $j$ trivial) is contained in the image of $\Gamma[\tilde C]\cap \PSBr(\tilde C)$ under $\tilde \iota_*$.
\end{lemma}
\begin{proof}
    We first discuss a construction called {\em tethering}. Let $(C,x)$ be a based surface, and let $a \subset C$ be an unbased simple closed curve endowed with a choice of orientation. Let $\tau$ be a simple arc on $C$ with one endpoint at $x$, the other at some point of $c$, and such that the interior is disjoint from $a$. The {\em tethering of $a$ by $\tau$} is the element $\ul{a}_\tau \in \pi_1(C, x)$ given by following $\tau$ from $x$ to $a$, then around $a$, then back along $\tau$.

    Conversely, let $\ul{a} \in \pi_1(C,x)$ be {\em any} element in the conjugacy class of $a$. We claim that $\ul{a} = \ul{a}_\tau$ for some tether $\tau$. To see this, note that certainly $\ul{a}$ is {\em conjugate} to some element of the form $\ul{a}_{\tau'}$, say $\ul{a} = g \ul{a}_{\tau'} g^{-1}$, for some $g \in \pi_1(C,x)$. But then $\ul{a}$ is realized as the image of $\ul{a}_{\tau'}$ under $g$ when viewing $g$ as a point-pushing diffeomorphism of $(C,x)$, showing that $\ul{a} = \ul{a}_{g(\tau')}$.

    Following this discussion, we show how to construct elements of the form $(1, \dots, \ul{a}_{\tau}, \dots, \ul{a}_{\tau'}^{-1}, \dots, 1)$ for an arbitrary tether $\tau$. Given $\tau$ a tether connecting $x_i$ to $a$, let $\tau'$ be a tether continuing from $\tau \cap a$ to $x_j$, that does not cross $a \cup \tau$. Let the union of $\tau$ and $\tau'$ be $\sigma$; by abuse of notation, we let the half-twist along this arc, an element of $\SBr_d(C)$, also be denoted $\sigma$. As shown in \Cref{fig:tethertwist}, the commutator $[T_a, \sigma] = T_a(\sigma) \sigma^{-1}$ is isotopic to a pair of point-pushes:
    \[
    \tilde \iota_*([T_a,\sigma]) = (1, \dots, \ul{a}_{\tau}, \dots, \ul{a}_{\tau'}^{-1}, \dots, 1).
    \]
    On the other hand, $T_a \in \Gamma[\tilde C]$, and by \Cref{lemma:simpleinmono}, there is an element $\tilde \sigma \in \Gamma_0[\tilde C]$ such that $\bar \res (\tilde \sigma) = \sigma$. By \Cref{lemma:principalnormal}, $[T_a,\tilde \sigma] \in \Gamma[\tilde C]$; this completes the argument.
\end{proof}

\begin{lemma}
    \label{lemma:VCsgenerate}
    Let $X$ be a smooth projective algebraic surface, and let $L$ be a very ample line bundle on $X$; let $C$ be a smooth section. Then there is an isomorphism
    \[
    \pi_1(X) \cong \pi_1(C)/\ng{a \subset C \mbox{ vanishing cycle}},
    \]
    describing $\pi_1(X)$ as the quotient of $\pi_1(C)$ by the subgroup generated by the conjugacy classes of all vanishing cycles in $\abs{L}$. In particular, if $\pi_1(X) = 1$, then $\pi_1(C)$ is generated by the conjugacy classes of vanishing cycles in $\abs{L}$.
\end{lemma}
\begin{proof}
This is a basic structural result in the theory of Lefschetz fibrations; we sketch a proof. The key point is that by definition, every vanishing cycle $a \subset C$ is freely homotopic to a point (the singular point of the associated nodal fiber) in $X$, so that the inclusion map $\iota_*: \pi_1(C) \to \pi_1(X)$ factors through the quotient $\pi_1(C)/\ng{a \subset C \mbox{ vanishing cycle}}$. An appeal to the Seifert-van Kampen theorem (after realizing $X$ as the total space of a Lefschetz pencil, blowing up the base locus, and deleting neighborhoods of the singular fibers) then shows that $\iota_*$ is a surjection and that $\ker(\iota_*)$ is precisely the normal subgroup generated by the conjugacy classes of vanishing cycles.
\end{proof}

        \begin{lemma}\label{lemma:leftcol}
            For $d \ge 3$, there is a surjection
            \[
            \PSBr(\tilde C)\cap \Gamma[\tilde C] \onto \pi_1(C^d)_0,
            \]
            and hence an exact sequence $1 \to K\cap \Gamma[\tilde C] \to \PSBr(\tilde C)\cap \Gamma[\tilde C] \to \pi_1(C^d)_0 \to 1$.
             \end{lemma}
             
           \begin{proof}
    Let $g = (g_1, \dots, g_d) \in \pi_1(C^d)_0 \le \prod_{i=1}^d \pi_1(C, x_i)$ be arbitrary. Since $\pi_1(X) = 1$, it follows from \Cref{lemma:VCsgenerate} that $\pi_1(C)$ (for any basepoint) is generated by the conjugacy classes of the vanishing cycles on $C$ associated to the linear system $L_1$. Thus by \Cref{lemma:pi1cd0img}, for $i = 1, \dots, d-1$, there are elements $\alpha_1, \dots, \alpha_{d-1} \in \PSBr(\tilde C)\cap\Gamma[\tilde C]$ such that 
                 \[
                 \tilde \iota_*(\alpha_i) = (1, \dots, g_i, \dots, g_i'^{-1}),
                 \]
                 with $g_i$ appearing in the $i^{th}$ factor, $g_i' \in \pi_1(C,x_d)$ in the conjugacy class of $g_i$, and all other entries trivial. Then
                 \[
                 \tilde \iota_*(\alpha_1\dots \alpha_{d-1}) = (g_1, \dots, g_{d-1}, g_{d}'),
                 \]
                 with $g_d' \in \pi_1(C, x_d)$ satisfying $[g_d'] = [g_d] \in H_1(C;\Z)$.

                 Thus to finish the argument, it suffices to show that $\tilde \iota_*(\PSBr(\tilde C)\cap\Gamma[\tilde C])$ contains the subgroup 
                 \[
                 (1, \dots, 1, [\pi_1(C,x_d), \pi_1(C,x_d)]).
                 \]
                 Since $\pi_1(C,x_d)$ is generated by the conjugacy classes of vanishing cycles, $[\pi_1(C,x_d), \pi_1(C,x_d)]$ is generated by elements of the form $[\ul{a},\ul{b}]$ for $\ul{a}, \ul{b} \in \pi_1(C,x_d)$ freely homotopic to vanishing cycles. By \Cref{lemma:pi1cd0img}, since $d \ge 3$, there are elements $\alpha, \beta \in \PSBr(\tilde C)\cap\Gamma[\tilde C]$ such that
                 \[
                 \tilde \iota_*(\alpha) = (\ul{a}'^{-1}, 1 \dots, 1, \ul{a}) \quad \mbox{and} \quad \tilde \iota_*(\beta) = (1,\ul{b}'^{-1}, 1 \dots, 1, \ul{b});
                 \]
                 then
                 \[
                 \tilde \iota_*([\alpha, \beta]) = (1, \dots, 1, [\ul{a}, \ul{b}])
                 \]
                 as required.
           \end{proof}

\subsection{The key diagram} As discussed above, the commutative diagram appearing in the proof of \Cref{lemma:insteadK} below summarizes the structural relationships between everything needed to prove \Cref{lemma:SBrSES}. This will require one preparatory result.
           
        \begin{lemma}\label{lemma:psbrnormal}
            The intersection $\PSBr(\tilde C) \cap \Gamma[\tilde C]$ is normal in $\PSBr(\tilde C)$.
        \end{lemma}
        \begin{proof}
          By \Cref{lemma:simpleinmono}, there is a subgroup $G \le \Gamma_0[\tilde C]$ such that $\bar \res(G) = \SBr_d(C) \le \Mod(C, \bx)$. Let $G' \le G$ be the subgroup that fixes each boundary component $\Delta_i \in \partial \tilde C$. By \Cref{lemma:TDelta}, each $\Delta_i$ is a vanishing cycle; let $G'' \le \Gamma_0[\tilde C]$ be the subgroup generated by $G'$ and the twists $T_{\Delta_i}$ for $i = 1,\dots, d$. Since $G''$ preserves each $\Delta_i$, there is a restriction map $\res: G'' \to \Mod(\tilde C)$, and by construction, $\res(G'') = \PSBr(\tilde C)$.  By \Cref{lemma:principalnormal}, the principal stabilizer $\Gamma[\tilde C] \normal \Gamma_0[\tilde C]$ is normal in $\Gamma_0[\tilde C]$; the claim follows.
        \end{proof}
    
        Following \Cref{lemma:psbrnormal}, define 
        \[
        A' := \PSBr(\tilde C) / (\PSBr(\tilde C) \cap \Gamma[\tilde C]).
        \]

      \begin{lemma}\label{lemma:insteadK}
            There is an exact sequence $1 \to K \cap \Gamma [\tilde C] \to K \to A' \to 1$.
        \end{lemma}
        \begin{proof}
                The sequence in question appears as the top row of the following commutative diagram:
                    \[
    \xymatrix{
                    & 1 \ar[d]                                  &   1 \ar[d]                                &                                   &   \\
        1 \ar[r]    & K \cap \Gamma[\tilde C] \ar[r] \ar[d]     & K \ar[r] \ar[d]      & A' \ar@{=}[d] \ar[r]  & 1 \\
        1 \ar[r]    & \PSBr(\tilde C) \cap \Gamma[\tilde C] \ar[r] \ar[d]            & \PSBr(\tilde C) \ar[r] \ar[d]             & A'  \ar[r]            & 1 \\
                    & \pi_1(C^d)_0 \ar@{=}[r] \ar[d]            & \pi_1(C^d)_0 \ar[d]                       &                                   &   \\
                    & 1                                         & 1                                         &                                   &
    }
    \]
        The remaining rows and columns are all exact: the middle row is exact by definition, the left column is exact by \Cref{lemma:leftcol}, and the middle is exact by \Cref{lemma:middlecol}. Exactness of the top row now follows from a diagram chase.
        \end{proof}

\subsection{A generating set for $\PSBr(\tilde C)$} The content of \Cref{lemma:SBrSES} can at this point be summarized by the assertion that the quotient $A'$ of $\PSBr(\tilde C)$ is abelian, and that there is a specific formula for the quotient map. To establish this, we will work with a particularly convenient generating set for $\PSBr(\tilde C)$.

Recall from \Cref{def:boundaryadj} that a subsurface $B \subset \tilde C$ is a {\em boundary-adjacent sphere} if $B \cong \Sigma_0^k$ is a sphere with $k$ boundary components, exactly one of which is not a boundary component of $\tilde C$. 

\begin{lemma}
    \label{lemma:psbrgenset}
    Let $d \ge 3$, and let $\sigma \subset \tilde C$ denote an arc connecting $\Delta_1$ to $\Delta_2$ (and, abusively, the associated half-twist). Then $\PSBr(\tilde C)$ admits a generating set consisting of three types of elements: (1) elements $\tau$ such that $\tau$ and $\sigma$ are both supported on some boundary-adjacent sphere $B \subset \tilde C$, (2) elements $\tau'$ whose support is disjoint from $\sigma$, and (3) $\PSBr(\tilde C) \cap \Gamma[\tilde C]$.
\end{lemma}

\begin{proof}
    The boundary twists $T_{\Delta_i}$ that generate the kernel of the map $\PSBr(\tilde C) \to \PSBr_d(C)$ are of the third type, so it suffices to exhibit a generating set for $\PSBr_d(C)$ of the required form. For this, we consider the forgetful map $\PConf_d(C) \to \PConf_2(C)$, forgetting all but the endpoints of $\sigma$. By \Cref{prop:FN}, there is a Fadell-Neuwirth short exact sequence
    \begin{equation}\label{eq:FN}
        1 \to \PBr_{d-2}(C^{\circ\circ}) \to \PBr_d(C) \to \PBr_2(C) \to 1,
    \end{equation}
    where $C^{\circ\circ}$ is the surface obtained from $C$ by deleting two distinct points.

    Let $\PSBr_{d-2}(C^{\circ\circ})$ denote the intersection $\PBr_{d-2}(C^{\circ\circ}) \cap \PSBr_d(C)$. We claim that \eqref{eq:FN} restricts to give the short exact sequence
    \[
    1 \to \PSBr_{d-2}(C^{\circ\circ}) \to \PSBr_d(C) \to \PBr_2(C) \to 1.
    \]
    To see this, note that since $d \ge 3$, any element of $\PBr_2(C)$ can be lifted to an element of $\PSBr_d(C)$, using the movement of the third point to enforce the condition that the total homology class be zero.

    Thus $\PSBr_d(C)$ is generated by two types of elements: (A) elements mapping onto a generating set of $\PBr_2(C)$, and (B) generators for $\PSBr_{d-2}(C^{\circ\circ})$. We claim that both of these types can be expressed using the desired generators of types (1),(2),(3). 

    We first consider elements of type (A). The Fadell-Neuwirth sequence for $\PConf_2(C) \to C$ realizes $\PBr_2(C)$ as an extension
    \[
    1 \to \pi_1(C\setminus{x_1}, x_2) \to \PBr_2(C) \to \pi_1(C,x_1) \to 1,
    \]
    implying that $\PBr_2(C)$ is generated by any set of elements surjecting onto a generating set for $\pi_1(C^2, (x_1,x_2))$. Applying \Cref{lemma:leftcol}, it follows that there is a surjection 
    \[
    \PSBr(\tilde C) \cap \Gamma[\tilde C] \to \pi_1(C^d)_0 \to \pi_1(C^2, (x_1, x_2))
    \]
    (the latter map being the projection onto the first two coordinates) as required.

    To describe a set of generators for $\PSBr_{d-2}(C^{\circ\circ})$, we begin with the problem of generating $\PBr_{d-2}(C^{\circ\circ})$. Following \Cref{lemma:bespokepbr}, $\PBr_{d-2}(C^{\circ\circ})$ is generated by two classes of elements: (i) point-push maps $a_i^j$ based at the points $x_j$ for $3 \le j \le d$, along paths $a_i^j$ disjoint from $\sigma$ that satisfy $[a_i^j] \ne 0$ as elements of $H_1(C;\Z)$, and (ii) $\Mod(B)$ for some boundary-adjacent sphere $B$ containing $\sigma$ and all the boundary components of $\tilde C$. Elements of the first class are of type (2), and elements of the second are of type (1). 

    The subgroup $\PSBr_{d-2}(C^{\circ\circ}) \le \PBr_{d-2}(C^{\circ\circ})$ is given as the kernel of $\lambda: \PBr_d(C) \to H_1(C;\Z)$, restricted to $\PBr_{d-2}(C^{\circ\circ}) \le \PBr_d(C)$. By the previous paragraph, an element $\tau \in \PSBr_{d-2}(C^{\circ\circ})$ can be expressed as
    \[
    \tau = a_1 b_1 \dots a_k b_k,
    \]
    with each $a_i$ of the first type $a_i^j$ and each $b_i \in \Mod(B)$. Since $\tau \in \PSBr_{d-2}(C^{\circ\circ}) \le \ker(\lambda)$ and $\lambda(b) = 0$ for any $b \in \Mod(b)$, it follows that $\sum_{i = 1}^k [a_i] = 0$ as elements of $H_1(C;\Z)$.

    We factor $\tau$ as above as follows:
     \[
     \tau = (a_1 b_1 a_1^{-1})(a_1 a_2) b_2 (a_1 a_2)^{-1} \dots (a_1 \dots a_k) b_k (a_1 \dots a_k)^{-1} a_1 \dots a_k.
     \]
     Note that if $B'$ is any boundary-adjacent sphere containing $\sigma$, and if $a_i^j$ is any point-push map of the type above (based at $x_j$ for $j \ge 3$ and disjoint from $\sigma$), then $a_i^j(B')$ is another boundary-adjacent sphere containing $\sigma$. In particular, if $b \in \Mod(B')$ is a generator of type (1), so is any conjugate $a_i^j b (a_i^{j})^{-1}$. Thus each element $(a_1 \dots a_i) b_i (a_1 \dots a_i)^{-1}$ in the above expression for $g$ is of type (1). The final term $a_1 \dots a_k$ belongs to $\PSBr_{d}(C)$ since $\sum_{i=1}^k [a_i] = 0$, and is composed of elements of support disjoint from $\sigma$, and hence is of type (2). 
\end{proof}
        
\subsection{Proving \Cref{lemma:SBrSES}}
\begin{proof}[Proof of \Cref{lemma:SBrSES}]
     Following \Cref{lemma:insteadK}, to determine the quotient $A' = \PSBr(\tilde C) / (\PSBr(\tilde C) \cap \Gamma[\tilde C])$, we can instead study the quotient $K/(K \cap \Gamma[\tilde C])$. We proceed in three steps.

    \para{Step 1: Generators for $\mathbf{K}$}
    Recall from \eqref{eq:K} that $K$ is the kernel of the composition 
    \[
    \PBr(\tilde C) \to \PBr_d(C) \xrightarrow{\iota_*} \pi_1(C^d).
    \]
    By definition, $\PBr(\tilde C):= \tilde \PBr_d(C)$. Recall from \Cref{lemma:brextension} that the kernel of $\tilde \PBr_d(C) \to \PBr_d(C)$ is $\Z^d$, with the $i^{th}$ factor generated by the element $\delta_i$ that rotates the $i^{th}$ tangent vector in place, fixing all other tangent vectors. 
    Under the embedding $\PBr(\tilde C) \le \Mod(\tilde C)$ of the Birman exact sequence (\Cref{lemma:birman}), $\delta_i$ corresponds to the twist $T_{\Delta_i}$. 

    Thus $K$ is generated by the twists $\{T_{\Delta_i}\}$ along with lifts of a set of generators for the kernel of $\iota_*: \PBr_d(C) \to \pi_1(C^d)$. To determine such a set of generators, observe that $\Conf_d(C) \subset C^d$ is the complement of the fat diagonal hypersurface \[
    \Delta = \bigcup_{1 \le i < j \le d} \Delta_{ij},
    \]
    with the component $\Delta_{ij} \subset C^d$ defined by the equation $z_i = z_j$. By the Seifert-van Kampen theorem, if $H \subset X$ is a hypersurface in a smooth variety $X$, then the kernel of the inclusion map $\pi_1(X \setminus H) \to \pi_1(X)$ is generated by the conjugacy classes of the {\em meridians} (fibers of the unit normal bundle of the smooth locus of $H$ inside $X$). 

    Let $\sigma_{ij}$ denote a meridian around the component $\Delta_{ij}$. As a loop in $\PConf_d(C)$, this is given by the $j^{th}$ point orbiting once around the $i^{th}$ (or vice versa). Lifted into $\PBr(\tilde C)$, this can be represented by the Dehn twist about some curve $c_{ij}$ bounding $\Delta_i$ and $\Delta_j$. We say that $\Delta_i$ and $\Delta_j$ are the {\em endpoints} of the meridian $\sigma_{ij}$.

    To summarize, we have shown that $K$ is generated by the following set of elements:
    \[
    K = \pair{\mbox{Meridians }\sigma_{ij}, \mbox{boundary twists } T_{\Delta_1}, \dots, T_{\Delta_d}}.
    \]
    We emphasize that we take here {\em all} meridians, not just some finite subcollection. As a final observation before the next step, note that by \Cref{lemma:TDelta}, $T_{\Delta_i} \in \Gamma[\tilde C]$ for $i = 1, \dots, d$. Thus to determine the quotient $K/(K \cap \Gamma[\tilde C])$, it suffices to understand the images of the meridians alone.

   \para{Step 2: The quotient is abelian} Let $\sigma \in K$ be a meridian. By considering the diagram in the proof of \Cref{lemma:insteadK}, we see that to show $A'$ is abelian, it suffices to show the {\em a priori} weaker claim that the image of $\sigma$ is {\em central} in the quotient $\PSBr(\tilde C)/ (\PSBr(\tilde C) \cap \Gamma[\tilde C])$. For this, we will show that $[\sigma, \tau] \in \PSBr(\tilde C) \cap \Gamma[\tilde C]$ for each of the generators $\tau$ of $\PSBr(\tilde C)$ found in \Cref{lemma:psbrgenset}. 

      If $\tau$ is of type (1), then $\tau$ is supported on some boundary-adjacent sphere $B$ also containing $\sigma$. Then the commutator $[\sigma,\tau]$ is contained in $[\Mod(B), \Mod(B)]\le \Gamma[\tilde C]$, the latter holding by \Cref{corollary:allBs}. If $\tau$ is of type (2), then the commutator $[\sigma,\tau] = 1$ is trivial in $\Mod(E)$. And if $\tau \in \PSBr(\tilde C) \cap \Gamma[\tilde C]$ is of type (3), then $[\sigma, \tau] \in \Gamma[\tilde C]$ by \Cref{lemma:principalnormal}.

   \para{Step 3: Determining the quotient} Let $\psi: K \to A'$ denote the quotient map. We next show that if $\sigma$ and $\sigma'$ are meridians with the same endpoints, then $\psi(\sigma) = \psi(\sigma')$. By \Cref{lemma:sbrarctrans}, there exists $\beta \in \PSBr(\tilde C)$ such that $\beta \sigma \beta^{-1} = \sigma'$. By the previous step, $\psi(\sigma) = \psi(\beta \sigma \beta^{-1}) = \psi(\sigma')$ as claimed.

   At this point, we have shown that $\psi: K \to A'$ factors through the map $\tilde \psi: K \to \Z\pair{e_{ij}}$, where $\Z\pair{e_{ij}} \cong \Z^{\binom{d}{2}}$ is the free abelian group spanned by symbols $e_{ij}$ for $1 \le i < j \le d$. To complete the argument, it suffices to show that $\psi$ moreover factors through the map $r: \Z\pair{e_{ij}} \to \Z\pair{e_i}$ given by $r(e_{ij}) = e_i + e_j$, where $\Z\pair{e_i} \cong \Z^d$ is the free abelian group spanned by symbols $e_i$ for $1 \le i \le d$. By linear algebra, $\ker(r)$ is generated by elements of the form $e_{ij} + e_{k\ell} - e_{ik} - e_{j\ell}$ for distinct $4$-tuples of indices $i,j,k, \ell$. By \Cref{corollary:allBs}, any element of the form $\sigma_{ij}\sigma_{k\ell}\sigma_{ik}^{-1}\sigma_{j\ell}^{-1}$ for meridians $\sigma_{xy}$ enclosing boundary components $\Delta_x, \Delta_y$, is contained in $\PSBr(\tilde C) \cap \Gamma[\tilde C]$, proving the claim.
\end{proof}

\section{Proof of \Cref{theorem:main}}\label{S:mainproof}

We now have almost all of the ingredients needed to prove \Cref{theorem:main}. The method of {\em assemblage generating sets} for framed mapping class groups (\Cref{theorem:assemblages}), in combination with the local monodromy calculations of \Cref{S:monodromy} and the Main Lemma (\Cref{lemma:main}), will show that the monodromy group $\Gamma_L$ contains {\em some} $r'$-spin mapping class group, {\em a priori} possibly for some {\em refinement} of the $r$-spin structure induced by an $r^{th}$ root of the adjoint line bundle. Before proceeding to give this argument in \Cref{SS:proof}, we show in \Cref{SS:rspinmono} that any containment $\Gamma_L \le \Mod(E)[\phi']$ of the monodromy in some $r'$-spin mapping class group is necessarily induced by an $(r')^{th}$ root of the adjoint line bundle, providing the last necessary piece of the argument. 

For convenience, we recall the hypotheses of \Cref{theorem:main}: $X$ denotes a smooth projective algebraic surface with $\pi_1(X) = 1$, and $L \in \Pic(X)$ is a line bundle admitting a decomposition of the form $L = L_1 \otimes L_2$, with $L_1$ $6$-jet ample and $L_2$ very ample. Recall also that $\abs{L} = \P H^0(X;L)$ denotes the total space of the linear system, and $U_L \subset \abs{L}$ is the open locus of smooth curves in $\abs{L}$.

\subsection{$r$-spin monodromy and the Picard group}\label{SS:rspinmono}

\begin{lemma}\label{lemma:relpic}
    Let $X,L$ be as in \Cref{theorem:main}. Let $\cC_L \to  U_L$ denote the universal family, and let $p : \cC_L \to X$ denote the projection onto the second factor. Then the pullback $p^*: \Pic(X) \to \Pic(\mathcal C_L | U_L)$ is an isomorphism.
\end{lemma}
\begin{proof}
    The surjectivity of $p^*$ follows from \cite[Theorem 1]{woolf} (under our hypotheses, $L$ is sufficiently ample to conclude that the locus of reducible or non-reduced curves is of codimension $\ge 2$, as required for this result).
    
    It remains to show that $p^*$ is injective.
    Since $\pi_1 (X) = 0$, it follows that $\Pic(X)$ is torsion-free, the map $\Pic (X) \to H^2 (X, \Z)$ is injective, and $H^2 (X, \Z)$ is torsion free.
    It suffices to establish that the map $p^*: \Pic (X) \to \Pic (\mathcal {C}_L)$ is injective- this will show that $\Pic (\mathcal {C}_L) $ is torsion-free, and since $\Pic (U_L) $ is torsion, $\Pic(\mathcal{C}_L) = \Pic (\mathcal C_L | U_L).$
    
    Let 
    \[
    \bar {\mathcal{C}}_L := \{(C,x) \mid C \in \abs{L},\ x \in C\} \subset \abs{L} \times X
    \]
    denote the universal family of all (not just smooth) curves in $|L|$; then there is a factorization of $p$ as $\mathcal{C}_L \to \bar {\mathcal{C}}_L \to X$. The map $\bar{\mathcal{C}}_L \to X$ is a projective bundle, and by the projective bundle formula there is a short exact sequence  
    \[
    0 \to \Pic (X) \to \Pic (\bar{\mathcal{C}_L}) \to \Z \to 0,
    \]
    where the last map is induced by restricting to a fiber. 

    Let $Y = \bar{\mathcal{C}}_L \setminus \mathcal{C}_L$. We note that 
    \[
    \Pic (\mathcal{C}_L) \cong \Pic(\bar{\mathcal{C}}_L) / \Z Y.
    \]
    Thus it suffices to prove that any multiple of the divisor $Y$ is not in the image of $\Pic (X)$. But this immediately follows since $Y$ restricts to a nontrivial divisor on each fiber of the map $\mathcal{\bar C}_L \to X$.
\end{proof}

\begin{lemma}\label{prop:charlie}
    Suppose there is a containment $\Gamma_L \le \Mod (E) [\phi']$ for some $r'$-spin structure $\phi'$. Then there is a line bundle $L' \in \Pic (X)$ such that $(L')^{\otimes r'}  \cong \omega_X \otimes L$.
\end{lemma}
\begin{proof}
    As in \Cref{lemma:relpic}, let  $\mathcal{C}_L \to  U_L$ denote the universal family and $p^*$ denote the induced map  $\Pic (X) \to \Pic (\mathcal C_L | U_L)$.
    Supposing $\Gamma_L \le \Mod (E) [\phi']$, then there is $L_0' \in \Pic (\mathcal C_L | U_L) $ such that $(L'_0)^{\otimes r'}   = p^*(\omega_X \otimes L)$ - this is the element of the relative Picard group that restricts to each fiber to give the presumed monodromy-invariant $r'$-spin structure. By \Cref{lemma:relpic}, $p^*$ is a bijection and so there is $L' \in \Pic (X)$ satisfying  $p^* (L') = L_0'$ and hence  $(L')^{\otimes r'}  \cong \omega_X \otimes L$.
\end{proof}

\subsection{The proof}\label{SS:proof}

        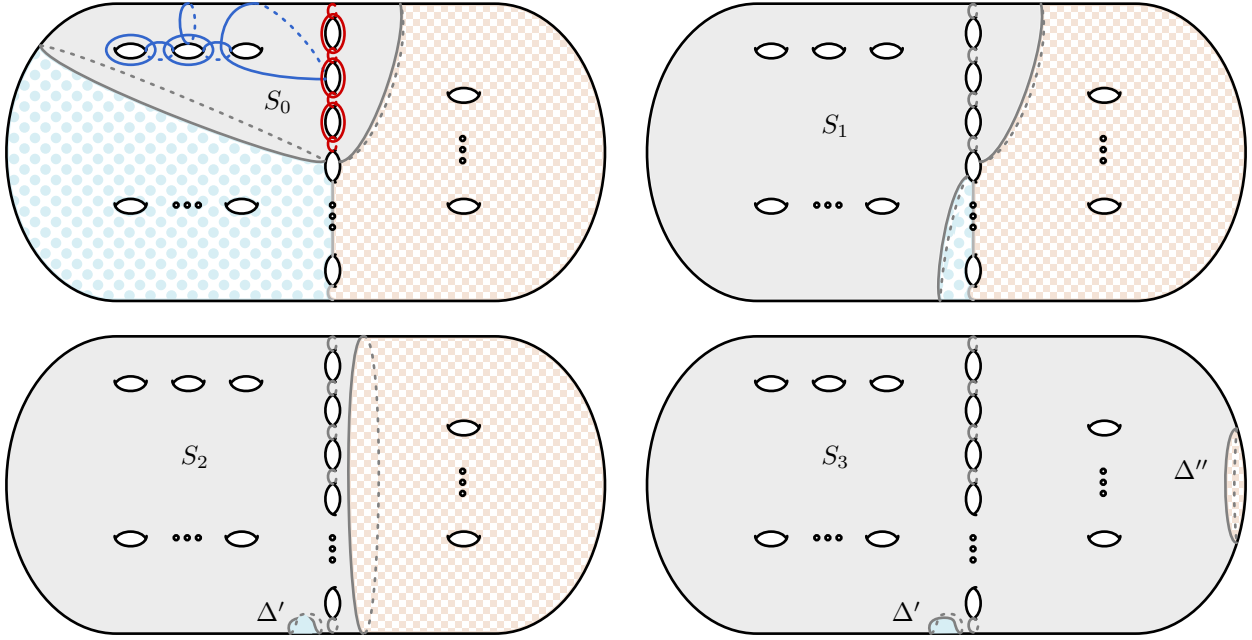
\begin{figure}[h]
\centering
		\labellist
        \small
        \pinlabel $S_0$ at 115 225
        \pinlabel $S_1$ at 350 215
        \pinlabel $S_2$ at 80 75
        \pinlabel $S_3$ at 350 75
        \pinlabel $\Delta'$ at 112 10
        \pinlabel $\Delta'$ at 380 10
        \pinlabel $\Delta''$ at 500 70
		\endlabellist
\includegraphics[width=\textwidth]{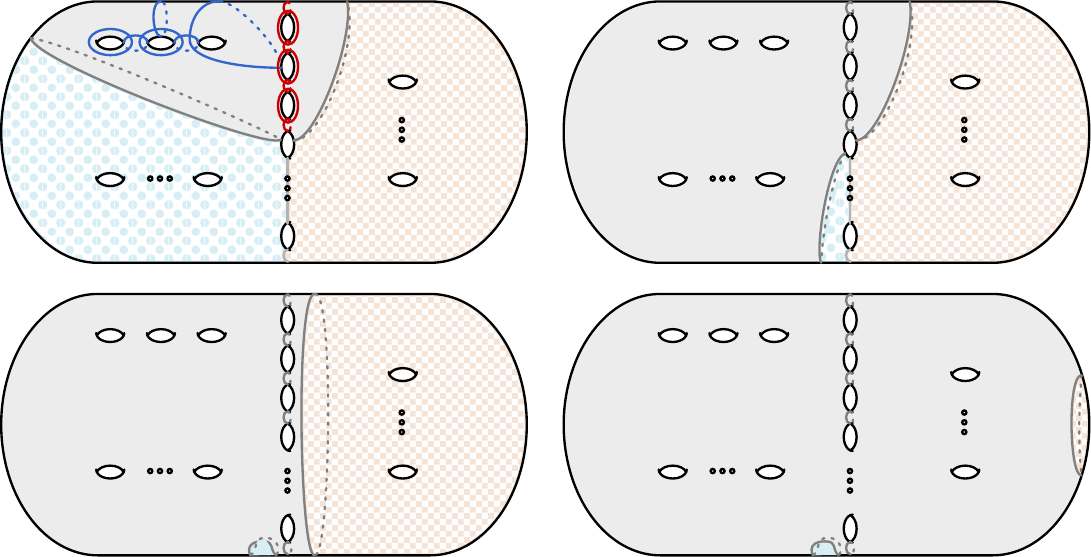}
\caption{The sequence $S_0 \subset S_1 \subset S_2 \subset S_3$ of framed subsurfaces of $E$ constructed via the method of assemblages. See \Cref{fig:0to1} for a description of the vanishing cycles used to enlarge $S_0$ to $S_1$, \Cref{fig:1to2} for the vanishing cycles used to enlarge from $S_1$ to $S_2$, and \Cref{fig:2to3} for the passage from $S_2$ to $S_3$.}
\label{fig:main1}
\end{figure}

\begin{proof}[Proof of \Cref{theorem:main}] 
    By \Cref{lemma:rspincontain}, the monodromy group $\Gamma_L = \im(\rho_L)$ is contained in the $r$-spin mapping class group $\Mod(E)[\phi_{M}]$ associated to the maximal $r^{th}$ root $M$ of the adjoint line bundle $L \otimes \omega_X$. We must establish the other containment.
    
    We will do this by an appeal to \Cref{theorem:assemblages}, constructing a $6$-assemblage of type $E$ consisting of vanishing cycles. To begin with, by \Cref{lemma:core}, there is a subsurface $S_0 \subset E$ homeomorphic to $\Sigma_6^2$ and vanishing cycles $\cC_0 = \{c_1, \dots, c_{13}\}$ supported on $S_0$ forming an $E$-arboreal spanning configuration on $S_0$. As shown in \Cref{fig:main1,fig:0to1,fig:1to2,fig:2to3} (and explained in detail in the captions), by repeated applications of \Cref{lemma:main}, this can then be successively extended to assemblages $\cC_0 \subset \cC_1 \subset \cC_2 \subset \cC_3$ of vanishing cycles on the chain of subsurfaces $S_0 \subset S_1 \subset S_2 \subset S_3 \subset E$. By \Cref{theorem:assemblages}, there is an equality
    \[
    \pair{T_{c_i} \mid c_i \in \cC_3} = \Mod(S_3)[\phi_3],
    \]
    where $\phi_3$ is the framing of $S_3$ compatible with the assemblage $\cC_3$ (\Cref{def:compatible}). Recall from \Cref{fig:1to2,fig:2to3} that $E \setminus S_3$ is the union of disks bounded by $\Delta'$ and $\Delta''$, and that $\phi_3(\Delta') = C \cdot (C+D+K_X)$ and $\phi_3(\Delta'') = D \cdot (C+D+K_X)$. Set $r' = \gcd\{C \cdot (C+D+K_X), D \cdot (C+D+K_X)\}$. \Cref{lemma:framedontospin} then shows that the inclusion $S_3 \into E$ induces a surjection
    \[
    f: \Mod(S_3)[\phi_3] \onto \Mod(E)[\bar \phi_3],
    \]
    where $\bar \phi_3$ is the $r'$-spin structure obtained by reducing $\phi_3$ mod $r'$. Observe that $r'$ is divisible by $r$, the order of the largest root of the adjoint line bundle $L\otimes \omega_X = \cO_X(C+D)\otimes \omega_X$.

    On the other hand, since each $c_i \in \cC_3$ is a vanishing cycle, it follows that $\im(f) \le \Gamma_L$, and hence
    \[
    \Mod(E)[\bar \phi_3] \le \Gamma_L \le \Mod(E)[\phi_{M}].
    \]
   We claim that $\Gamma_L$ is not contained in any proper subgroup of $\Mod(E)[\phi_M]$ of the form $\Mod(E)[\phi'']$ for any $r''$-spin structure $\phi''$: by \Cref{prop:charlie}, any such containment would imply the existence of a nontrivial root of $M$, contrary to hypothesis. Then by \Cref{lemma:nonconmax}, it follows that
    \[
    \Gamma_L = \Mod(E)[\phi_{M}]
    \]
    as claimed.
\end{proof}

        \begin{figure}[h!]
\centering
		\labellist
        \tiny
        \pinlabel $a_1'$ at 110 85
        \pinlabel $c$ at 125 100
        \pinlabel $a_2'$ at 370 85
		\endlabellist
\includegraphics[width=\textwidth]{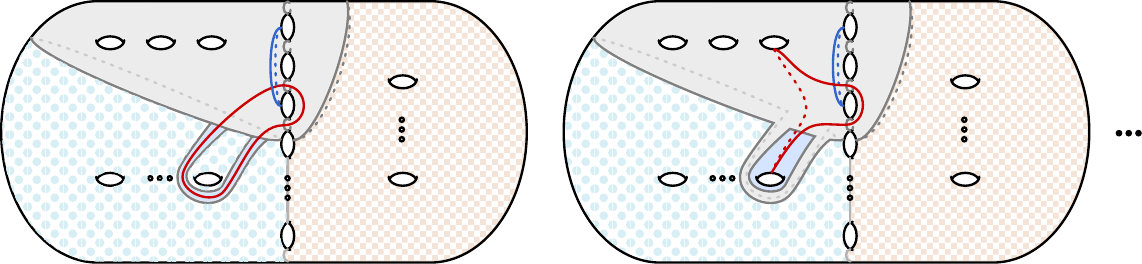}
\caption{Enlarging from $S_0$ to $S_1$. Let $a_1'$ be a simple closed curve with $a_1' \cap \tilde D$ contained in $S_0$, and such that $a_1' \cap \tilde C$ enters and exits $S_0$ exactly once. Let $c \subset S_0$ be boundary-adjacent, enclosing exactly one of the boundary components crossed by $a_1'$, and satisfying $i(a_1', c)=1$. By \Cref{lemma:main}, some $a_1:= T_c^{\ell}(a_1')$ is a vanishing cycle. Adding $a_1$ to the assemblage enlarges $S_0$ to a framed subsurface $S_{0,1}$. Then repeat, adding vanishing cycles $a_2,\dots,a_k$ to the assemblage, creating framed subsurfaces $S_{0,2} \subset \dots \subset S_{0,k} = S_1$.}
\label{fig:0to1}
\end{figure}
\newpage

        \begin{figure}[h!]
\centering
		\labellist
        \tiny
        \pinlabel $a_1'$ at 125 215
        \pinlabel $c$ at 125 240
        \pinlabel $\Delta_5$ at 400 210
        \pinlabel $\Delta'$ at 385 10
		\endlabellist
\includegraphics[width=\textwidth]{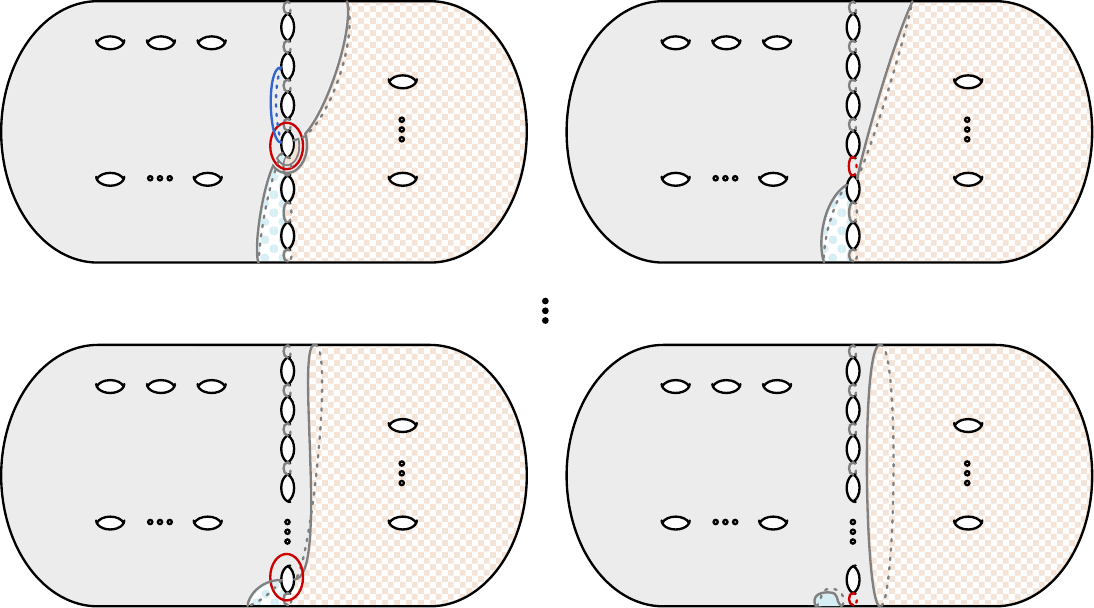}
\caption{Enlarging from $S_1$ to $S_2$. Choose $a_1'$ a simple closed curve joining the two boundary components of $S_1$, that also crosses the boundary components $\Delta_4$ and $\Delta_5$ of $\tilde C$. Note that $\Delta_4 \subset S_1$, and $\Delta_5$ lies outside. Choosing as before a boundary-adjacent curve $c \subset \tilde C$ enclosing $\Delta_3, \Delta_4$, by \Cref{lemma:main}, some $a_1:= T_c^\ell(a_1')$ is a vanishing cycle. Add it to the assemblage, enlarging $S_1$ to a framed subsurface $S_{1,1}$. Next, recall from \Cref{lemma:TDelta} that $\Delta_5$ is a vanishing cycle; by construction it enters and exits $S_{1,1}$ once, so it can be added to the assemblage, enlarging $S_{1,1}$ to a framed subsurface $S_{1,2}$. Repeat this process, adding vanishing cycles $a_2,\Delta_6, \dots, a_{d-4},\Delta_d$, creating framed subsurfaces $S_{1,3}, \dots, S_{1,2d-8} := S_2$. At the final stage, attaching $\Delta_d$ creates a boundary component $\Delta'$ of $S_2$ that is inessential in $E$. By homological coherence, since each $\Delta_1,\dots, \Delta_d$ is admissible, $\phi_2(\Delta') = \chi(C) - d -1$, where $\phi_2$ is the framing on $S_2$ specified by the assemblage. By the adjunction formula, $\chi(C) - d -1 = C \cdot (C+D+K_D)-1$.}
\label{fig:1to2}
\end{figure}
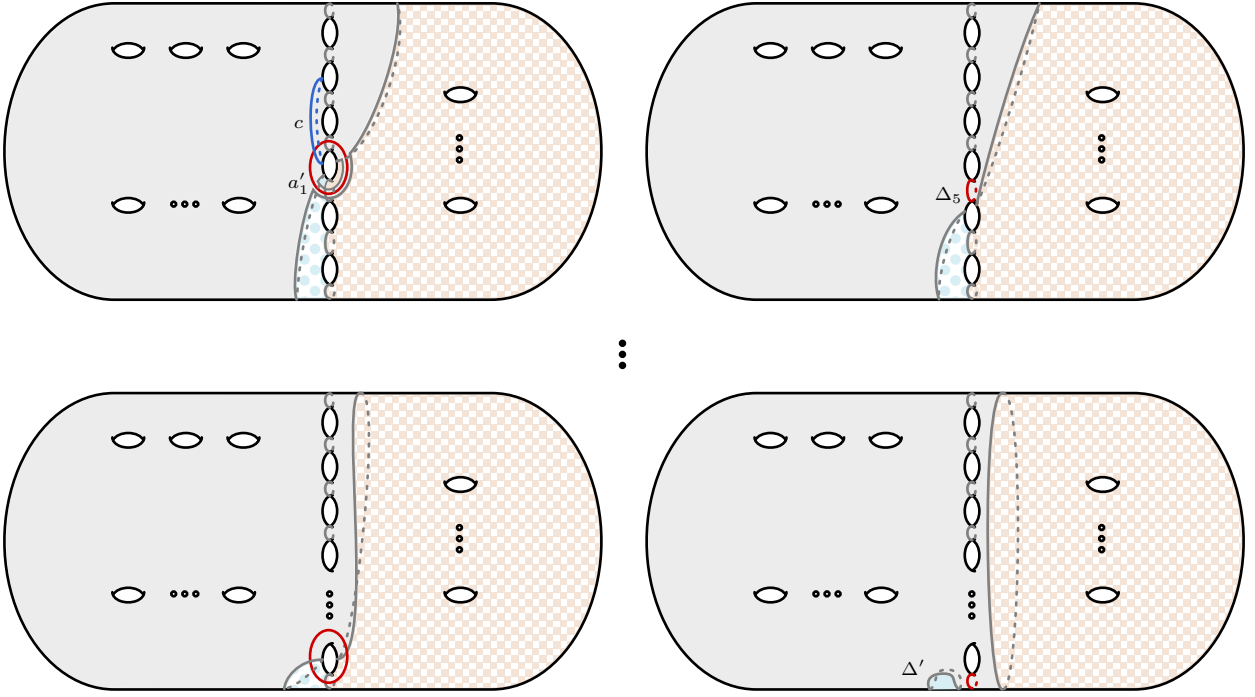
\newpage

        \begin{figure}[h!]
\centering
		\labellist
        \tiny
        \pinlabel $a_1'$ at 125 80
        \pinlabel $c$ at 123 110
		\endlabellist
\includegraphics{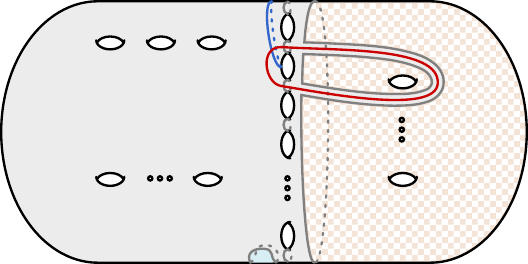}
\caption{Enlarging from $S_2$ to $S_3$. Choose a simple closed curve $a_1'$ entering and exiting $S_2$ once along its essential boundary component, such that $a_1'$ also enters and exits $\tilde C$ once, along distinct boundary components. As before, also choose $c \subset \tilde C$ a boundary-adjacent curve enclosing exactly one of the boundary components crossed by $a_1'$, satisfying $i(a_1',c) = 1$. Then by \Cref{lemma:main}, some $a_1:= T_c^\ell(a_1')$ is a vanishing cycle; add it to the assemblage, creating a new framed subsurface $S_{2,1}$. Repeat this process, adding vanishing cycles $a_2, \dots, a_k$ and corresponding framed subsurfaces $S_{2,2} \subset \dots \subset S_{2,k}:=S_3$. By construction, the complement of $S_3$ in $E$ consists of two disks: $\Delta'$ as discussed in \Cref{fig:1to2}, and $\Delta'' \subset \tilde D$. A similar calculation as before shows that $\phi_3(\Delta'') = \chi(D) - d -1 = D \cdot (C+D+K_X)-1$.}
\label{fig:2to3}
\end{figure}

\bibliographystyle{alpha}
\bibliography{bibliography}

\end{document}